%% file: etats.tex
\documentclass[utf8,12pt]{aclcours}

\def\prooffont{\normalfont}
\usepackage{natbib}
  
\usepackage{smfhref}
\makeatletter
  \def\NAT@nmfmt#1{{\scshape\NAT@up#1}}
  \let\c@equation\c@paragraph
  \def\theequation{\thesection.\arabic{paragraph}}
\makeatother

\def\thesection{\thechapter.\arabic{section}}
\def\thesubsection{\Alph{subsection}}

\title{Th\'eor\`emes d'\'equidistribution pour les syst\`emes dynamiques 
d'origine arithm\'etique}
\subtitle{\emph{Quelques aspects des syst\`emes dynamiques polynomiaux},\\
\'Etats de la Recherche, mai 2006}

\begin{document}
\maketitle

{\makeatletter\let\Ref@Option\Ref@Option@Ignore\makeatother
\tableofcontents}

\let\bar\overline
\def\abs#1{\left|{#1}\right|}
\def\norm#1{\left\|{#1}\right\|}
\def\P{{\mathbf P}}
\def\Card{\operatorname{card}}
\def\kod{\operatorname{kod}}
\def\Aut{\operatorname{Aut}}
\let\hra\hookrightarrow
\let\ra\rightarrow
\def\Pic{\operatorname{Pic}}
\def\div{\operatorname{div}}
\def\Ample{\operatorname{Ample}}
\def\Picplus{\operatorname{Pic^+}}
\def\Gal{\operatorname{Gal}}
\def\Spec{\operatorname{Spec}}
\def\A{{\mathbf A}}
\def\gm{{\mathbf G}_{\mathrm m}}
\def\Spec{\operatorname{Spec}}
\def\ord{\operatorname{ord}}
\def\abs#1{\left|#1\right|}
\def\hdeg{\mathop{\widehat{\mathrm {deg}}}}
\def\hdiv{\mathop{\widehat{\operatorname{div}}}}
\def\hvol{\mathop{\widehat{\operatorname{vol}}}}
\def\vol{\operatorname{vol}}
\def\ddc{\mathop{\mathrm d\mathrm d^{c}}}
\def\hZ^#1{\widehat{\mathrm{Z}}\vphantom{\operatorname{Z}}^{#1}}
\def\CH^#1{\operatorname{CH}}
\def\hCH^#1{\widehat{\mathrm{CH}}\vphantom{\operatorname{CH}}^{#1}}
\def\hdiv{\widehat{\mathrm{div}}}
\def\hPic{\mathop{\widehat{\operatorname{Pic}}}}

\include{intro}

\include{hauteurs}

\include{sysdyn}
\include{equip1}

\nocite{hindry-silverman2000,serre1997}
\bibliographystyle{mynat}
\bibliography{biblio}

\end{document}

%% file: intro.tex
\chapter*{Introduction}

Ce texte est une introduction à quelques problèmes
arithmétiques liés aux systèmes dynamiques issus
de la géométrie algébrique.
Il se comporte de trois chapitres assez autonomes,
issus des trois exposés que j'avais faits
lors des \emph{États de la recherche} en mai~2006 à Rennes.

Le premier chapitre, \emph{Hauteurs sur l'espace projectif},
explique la théorie des hauteurs, initiée par \textsc{Weil}
et \textsc{Northcott}, et en démontre les principales
propriétés fondamentales, notamment
la variation de la hauteur sous l'effet d'un morphisme
et le théorème de finitude. La théorie est bien plus simple
lorsque l'on se borne au cas du corps des nombres rationnels,
car toutes les subtilités liées à la théorie algébrique des nombres
s'évanouissent alors. Nous nous limitons d'ailleurs à ce cas
dans le premier paragraphe, et en profitons pour donner un aperçu
de ce que cette théorie peut dire des systèmes dynamiques polynomiaux,
en particulier de leurs points prépériodiques. C'était
en effet la motivation initiale du théorème de finitude.
Une dernier paragraphe, d'esprit plus analytique, 
exprime la hauteur d'un point comme somme de termes locaux; le
formalisme des fonctions de Green que nous introduisons 
est inspiré du langage de la géométrie d'Arakelov.

Le second chapitre est consacré aux systèmes dynamiques polynomiaux
et, principalement, à ceux qui sont \emph{polarisés}
au sens de~\cite{zhang2006}. Nous expliquons un certain nombre
de conjectures arithmétiques les concernant, en tâchant de décrire
quelques exemples et quelques démonstrations. De fait, la résolution
de certaines de ces conjectures dans l'exemple du système
dynamique donné par la multiplication par~$2$ dans une variété abélienne
est l'une des grandes avancées du sujet dans les années 1980--2000.
Nous introduisons enfin le théorème d'équidistribution
de~\cite{szpiro-u-z97} et ébauchons son application à la
résolution des conjectures évoquées, ainsi que quelques généralisations.
Peu de temps avant que ce texte n'entre sous presse, j'ai appris
l'existence du contre-exemple de~\cite{ghioca-tucker2009}
à certaines de ces conjectures; nous en profitons pour le décrire.

La preuve de ce théorème d'équidistribution
requiert tout l'arsenal de la géométrie d'Arakelov 
et sort du cadre de ces notes.  Dans
un  dernier chapitre, nous le démontrons dans le cas particulier
d'un système dynamique associé à une fraction rationnelle
de degré~$\geq 2$ en une variable. Suivant
la méthode de~\cite{bilu97}, nous montrons enfin
comment un cas particulier implique certaines des conjectures
du chapitre~2 pour les systèmes dynamiques toriques.

Chacun des chapitres se termine par quelques exercices et compléments,
parfois issus de la littérature récente. 

Depuis une dizaine d'années, l'étude arithmétique des systèmes
dynamiques polynomiaux s'est considérablement développée.
L'objet initial de ces exposés était d'expliquer
à un public principalement issu de la dynamique holomorphe 
les théorèmes d'équidistribution en géométrie d'Arakelov.
Le lecteur interessé trouvera dans l'ouvrage~\cite{silverman2007}
de nombreux développements  qui n'ont pu trouver leur place
dans ce texte.

\medskip

Je remercie les organisateurs de m'avoir invité à donner ce cours,
et le public, nombreux, pour sa participation.
Je remercie aussi P.~\textsc{Autissier}, M.~\textsc{Baker}, S.~\textsc{Cantat}, T.-C.~\textsc{Dinh}, V.~\textsc{Guedj},
L.~\textsc{Moret-Bailly}, N.~\textsc{Sibony}, ainsi que le rapporteur,
pour leurs commentaires pendant la conférence
ou sur des versions préliminaires de ce texte. 
Je remercie enfin D.~\textsc{Ghioca} et T.~\textsc{Tucker} de m'avoir
autorisés à inclure leur contre-exemple, non encore publié.

%% file: hauteurs.tex
\chapter{Hauteurs sur l'espace projectif}
\label{chap.hauteurs}

La m\'ethode de {\og descente infinie\fg}
initi\'ee par Pierre de \textsc{Fermat} (1601--1665)
dans l'\'etude des \'equations diophantiennes
repose sur trois piliers:
\begin{itemize}
\item une notion de \emph{taille}  d'une solution d'une telle \'equation; 
\item \`a partir d'une solution donn\'ee, la construction d'une solution
de taille moindre ;  
\item le fait que ces tailles --- usuellement des nombres entiers --- 
ne peuvent diminuer ind\'efiniment.
\end{itemize}
Cette m\'ethode a permis \`a \textsc{Fermat} d'\'etablir  
des r\'esultats n\'egatifs, 
par exemple l'inexistence de triangles rectangles \`a c\^ot\'es entiers dont
l'aire soit un carr\'e parfait, probl\`eme
qui se ram\`ene \`a {\og l'\'equation de Fermat\fg} de degr\'e~$4$.
Elle intervient aussi dans la d\'emonstration
par Ernst \textsc{Kummer} (1810--1893)
du grand th\'eor\`eme de \textsc{Fermat} 
pour les nombres premiers r\'eguliers\footnote{Ce sont les
nombres premiers~$p$ qui ne divisent pas le num\'erateur
d'un des nombres de Bernoulli~$B_2, B_4,\ldots, B_{p-3}$,
ou plus conceptuellement, les nombres premiers~$p$
tel que le groupe des classes d'id\'eaux du corps
cyclotomique~$\mathbf Q(\zeta_p)$ soit d'ordre premier \`a~$p$.}.
Plus remarquablement peut-\^etre,
elle a aussi permi de prouver des r\'esultats positifs;
citons par exemple la d\'emonstration (1747) de Leonhard \textsc{Euler}
(1707--1783)
que tout nombre premier congru \`a~$1$ modulo~$4$
est somme de deux carr\'es (un th\'eor\`eme annonc\'e par~\textsc{Fermat}
lui-m\^eme en 1640).

Cette m\'ethode a deux avatars modernes. Le premier, la \emph{th\'eorie des
hauteurs}, auxquelles ce texte est consacr\'e, est une version
g\'eom\'etris\'ee de la notion de taille d'une solution d'une
\'equation diophantienne. Elle fut d\'evelopp\'ee
\`a la fin des ann\'ees~40 par Andr\'e \textsc{Weil} (1906--1998)
et Douglas Geoffrey \textsc{Northcott} (1916--2005).
C'est en effet l'un des deux ingr\'edients de la preuve
du th\'eor\`eme de \textsc{Mordell}--\textsc{Weil} selon
lequel les points rationnels d'une vari\'et\'e ab\'elienne
d\'efinie sur un corps de nombres forment un groupe ab\'elien
de type fini (voir par exemple~\cite{serre1997}).
\textsc{Northcott} utilisa ce concept de hauteur pour \'etablir des
propri\'et\'es arithm\'etiques de certains syst\`emes dynamiques;
nous y reviendrons amplement.

L'autre ingr\'edient de la d\'emonstration du th\'eor\`eme de \textsc{Mordell--Weil}
r\'esulte du th\'eor\`eme de finitude de \textsc{Hermite--Minkowski}
en th\'eorie alg\'ebrique des nombres et d'un argument de cohomologie galoisienne. 
Son \'elaboration moderne, \emph{vari\'et\'es de descente},
lois de r\'eciprocit\'e et \emph{principes local-global}, 
sort largement du cadre de cet article. Je renvoie 
au survol~\citeyearpar{colliot-thelene1992}
de J.-L.~\textsc{Colliot-Th\'el\`ene}  pour une premi\`ere
introduction.

\medskip

Revenons aux hauteurs.
Dans les applications modernes \`a l'arithm\'etique, 
il est souvent utile, voire crucial, d'assouplir
le strict point de vue {\og \'equationnel\fg} des classiques
pour mieux exploiter les propri\'et\'es g\'eom\'etriques
des objets \'etudi\'es.
On est ainsi plut\^ot amen\'e \`a d\'efinir la hauteur 
d'un point de l'espace projectif, voire d'une vari\'et\'e
projective, et \`a \'etudier le comportement de cette hauteur
par des morphismes de vari\'et\'es alg\'ebriques.
Il faudra \'etendre les d\'efinitions na\"{\i}ves que l'on peut
adopter pour les nombres entiers au cas des nombres alg\'ebriques.
Nous commen\c{c}ons cependant ce cours par le cas des points rationnels:
il fait d\'ej\`a appara\^{\i}tre les id\'ees principales tout en \'evitant
les complications dues \`a la th\'eorie alg\'ebrique des nombres.

\section{Hauteur d'un point rationnel}
\label{sec.hauteur-Q}

\subsection{D\'efinition}

Soit donc $k$ un entier naturel et notons $\mathbf P^k$
l'espace projectif de dimension~$k$. Voyons-le comme \emph{sch\'ema},
c'est-\`a-dire comme la donn\'ee, pour tout anneau~$A$,
de l'ensemble $\mathbf P^k(A)$ de ses points \`a coordonn\'ees dans~$A$.
De fait, nous n'aurons besoin pour l'instant que du cas o\`u $A$
est un corps~$F$ : alors, $\mathbf P^k(F)$ n'est autre
que l'ensemble des droites vectorielles de l'espace~$F^{k+1}$.
Un point $x\in\mathbf P^k(F)$ poss\`ede ainsi 
$k+1$ coordonn\'ees $(x_0,\dots,x_k)$ non toutes nulles 
--- celles d'un vecteur directeur
quelconque de la droite correspondante --- bien d\'efinies \`a un
facteur multiplicatif non nul pr\`es.
Si $(x_0,\dots,x_k)$ est un \'el\'ement non nul de~$F^{k+1}$, on
notera $[x_0:\cdots:x_k]$ le point correspondant de~$\mathbf P^k(F)$;
nous dirons que $(x_0,\dots,x_k)$ en sont des \emph{coordonn\'ees homog\`enes}.

Consid\'erons dans ce num\'ero le cas du corps~$\Q$ des nombres rationnels.
Soit ainsi $x$ un point de~$\mathbf P^k(\Q)$ --- on parle
de \emph{point rationnel}.
Parmi la multiplicit\'e de ses coordonn\'ees homog\`enes, on peut
en choisir certaines plus particuli\`erement. Il est
en effet loisible de multiplier les~$x_i$ par un d\'enominateur
commun ; ce sont alors des entiers relatifs.
On les divise alors par leur plus grand diviseur commun,
de sorte \`a obtenir une famille $(x_0,\dots,x_k)$
de coordonn\'ees homog\`enes form\'ee d'entiers relatifs premiers
entre eux dans leur ensemble.
Observons alors que deux telles familles ne d\'efinissent le m\^eme point
que si elles diff\`erent l'une l'autre par multiplication par~$\pm 1$.

Cela montre la correction de la d\'efinition suivante:
\begin{defi}
Soit $x=[x_0:\cdots:x_k]$ un point de~$\mathbf P^k(\Q)$,
dont les coordonn\'ees homog\`enes sont des entiers premiers entre eux
dans leur ensemble.
On appelle \emph{hauteur exponentielle}  de~$x$ le nombre entier
$H(x)=\max(\abs{x_0},\dots,\abs{x_k})$.
La \emph{hauteur logarithmique} de~$x$ est d\'efinie par 
la formule
\[ h(x) = \log H(x) = \log \max(\abs{x_0},\dots,\abs{x_k}). \]
\end{defi}

Les deux notions de hauteurs, exponentielle et logarithmique,
ont leur int\'er\^et suivant les contextes. Dans ce texte,
nous appellerons tout simplement \emph{hauteur} la hauteur logarithmique.

De la d\'efinition r\'esulte imm\'ediatement une propri\'et\'e
de \emph{finitude}, facile mais fondamentale.

\begin{prop}\label{prop.finitude-Q}
Pour tout nombre r\'eel~$B$, l'ensemble des points de 
$\mathbf P^k(\Q)$ de hauteur au plus~$B$ est fini.
\end{prop}
\begin{proof}
En effet, un tel point est d\'etermin\'e par le choix de~$k+1$ entiers
relatifs $(x_0,\dots,x_k)$, non tous nuls, et v\'erifiant
$\abs{x_i}\leq \mathrm e^B$ pour tout~$i$. Il n'y a qu'un nombre
fini de telles familles d'entiers, d'o\`u la proposition.
\end{proof}

Il est en fait possible, voir~\cite{schanuel79}, d'\'etablir
le comportement asymptotique du nombre~$N(B)$ de points
de~$\mathbf P^k(\Q)$ de hauteur au plus~$B$ --- il est commode ici
de consid\'erer la hauteur exponentielle. On trouve
\[ \Card \{x\in\mathbf P^k(\Q)\sozat H(x)\leq B\}
   \simeq \frac{2^k}{\zeta(k+1)} B^{k+1}, \]
o\`u $\zeta$ d\'esigne la fonction z\^eta de \textsc{Riemann}.
On doit \`a Yuri \textsc{Manin} d'avoir entrevu comment cette asymptotique
peut se g\'en\'eraliser \`a des vari\'et\'es plus g\'en\'erales. L'exposant~$k+1$
doit par exemple \^etre interpr\'et\'e comme l'ordre du p\^ole 
de la forme diff\'erentielle 
$d(x_1/x_0)\wedge \dots \wedge d(x_k/x_0)$ sur~$\mathbf P^k$
le long de l'hyperplan d'\'equation~$x_0=0$.
Je renvoie \`a~\cite{peyre2002} pour une introduction \`a ce th\`eme.

\subsection{Irrationalit\'e des points pr\'ep\'eriodiques}
\label{subsec.irrat-prep}

Venons-en maintenant au th\`eme de ces \emph{\'Etats de la recherche}
et consid\'erons un syst\`eme dynamique polynomial sur~$\P^k$,
c'est-\`a-dire une application~$f$ de~$\mathbf P^k$ dans lui-m\^eme qui
applique un point~$x$ de coordonn\'ees homog\`enes~$[x_0:\cdots:x_k]$
sur le point~$f(x)$ dont des coordonn\'ees homog\`enes sont
\[  [f_0(x_0,\dots,x_k):f_1(x_0,\dots,x_k):\cdots:f_k(x_0,\dots,x_k)], \]
o\`u les~$f_i$ sont des polyn\^omes \`a coefficients dans~$\Q$ (pour le moment).
Pour que la d\'efinition fasse sens et d\'efinisse
un \emph{endomorphisme de l'espace projectif},
il est n\'ecessaire et suffisant que les polyn\^omes~$f_i$
soient \emph{homog\`enes} de m\^eme degr\'e, disons~$d$, 
et qu'ils n'aient pas de z\'ero commun dans~$\P^k$.
Par l\`a, nous entendons sans z\'ero commun non seulement dans~$\P^k(\Q)$,
mais aussi dans~$\P^k(\C)$, condition bien plus forte.
Ils sont alors sans facteur commun.
Nous dirons pour abr\'eger que $f$ est un \emph{endomorphisme
de~$\P^k$} et que $d$ est son degr\'e.

\begin{prop}[Northcott]\label{prop.fonctorialite-Q}
Soit $f\colon\P^k\ra\P^k$ un endomorphisme de degr\'e~$d$ de~$\P^k$.
Il existe alors un nombre r\'eel~$c$ (ne d\'ependant que de~$f$)
tel que tout point $x\in\mathbf P^k (\Q)$ satisfasse les in\'egalit\'es
\[  d h(x) - c \leq h(f(x)) \leq d h(x)+c. \]
\end{prop}
\begin{proof}
Il est loisible de multiplier les polyn\^omes~$f_i$ par un d\'enominateur
commun de leurs coefficients; ce sont alors des polyn\^omes
\`a coefficients dans~$\Z$.

Consid\'erons l'\'enonc\'e suivant:
pour tout polyn\^ome~$g\in\Z[X_0,\dots,X_k]$, suppos\'e homog\`ene de degr\'e~$d$,
il existe un nombre r\'eel~$c(g)$ tel que l'on ait
\begin{equation}
\label{evident}
 \abs{g(x_0,\dots,x_k)} \leq c(g) \max(\abs{x_0},\dots,\abs{x_k})^d
\end{equation}
pour tout $(x_0,\ldots,x_k)\in\R^{k+1}$.
Le cas d'un mon\^ome est \'evident; 
par r\'ecurrence sur le nombre de mon\^omes de~$g$,
le cas g\'en\'eral en r\'esulte gr\^ace \`a l'in\'egalit\'e triangulaire.

Soit $x$ un point de~$\P^k(\Q)$ de coordonn\'ees homog\`enes
$[x_0:\cdots:x_k]$, suppos\'ees enti\`eres et premi\`eres entre elles.
Le point~$f(x)$ a pour coordonn\'ees homog\`enes la famille
$[f_0(x):\cdots:f_k(x)]$. Celles-ci sont enti\`eres mais pas forc\'ement
premi\`eres entre elles ; notons donc~$\delta$ leur pgcd (suppos\'e~$\geq 1$).
On a alors
\begin{align*}
 h(f(x)) &=\log \max\left( \abs{f_0(x)/\delta},\dots,\abs{f_k(x)/\delta}\right)\\
& \leq \log \max(\abs{f_0(x)},\dots,\abs{f_k(x)}) \\
& \leq \log \max(c(f_0),\dots,c(f_k)) + d \log \max(\abs{x_0},\dots,\abs{x_k})\\
& \leq c + d h(x),
\end{align*}
o\`u $c=\log\max_i c(f_i)$.

L'autre in\'egalit\'e est plus subtile. Comme les~$f_i$ sont sans z\'ero
commun dans~$\P^k(\C)$, leur seul z\'ero commun  dans~$\C^{k+1}$
est $(0,\dots,0)$.
Un polyn\^ome~$g$ sur~$\C^{k+1}$ qui s'annule l\`a o\`u les~$f_i$ s'annulent
simultan\'ement n'est pas n\'ecessairement une combinaison des~$f_i$.
Toutefois, le th\'eor\`eme des z\'eros de \textsc{Hilbert} 
entra\^{\i}ne que c'est le cas d'une puissance de~$g$:
\begin{theo}[Th\'eor\`eme des z\'eros de Hilbert]
Soit $F$ un corps alg\'ebriquement clos, soit $P_1,\dots,P_n$ des polyn\^omes
en $k$ variables \`a coefficients dans~$F$.
Soit $P\in F[X_1,\dots,X_k]$ un polyn\^ome
qui s'annule en tout point $x=(x_1,\dots,x_k)$ de~$F^k$ 
tel que $P_1(x)=\cdots=P_n(x)$. Il existe alors un entier~$m\geq 1$
et des polyn\^omes $Q_1,\dots,Q_n\in F_0[X_1,\dots,X_k]$
tels que $P^m = P_1Q_1+\cdots+P_nQ_n$.

De plus, si les coefficients des polyn\^omes~$P$, $P_1,\dots,P_n$
appartiennent \`a un sous-corps~$F_0$ de~$F$,
on peut choisir les polyn\^omes~$Q_i$ de sorte que
leurs coefficients appartiennent \`a~$F_0$.

Enfin, si les polyn\^omes $P$ et~$P_1,\dots,P_n$ sont homog\`enes,
on peut choisir les polyn\^omes~$Q_i$ homog\`enes.
\end{theo}
\begin{proof}
La premi\`ere partie de l'\'enonc\'e est le th\'eor\`eme classique
dont le lecteur trouvera une d\'emonstration
dans tout livre d'alg\`ebre commutative \'el\'ementaire.

Les deux ajouts s'en d\'eduisent en consid\'erant 
une base de~$F$ comme $F_0$-espace vectoriel
et les composantes homog\`enes des polyn\^omes~$Q_i$.
\end{proof}

Appliquons ceci aux polyn\^omes~$X_j$, pour $0\leq j\leq k$.
Il existe donc un entier~$t$ et des polyn\^ome $g_{ij}$
\`a coefficients rationnels
tels que $X_i^t=\sum_{i=0}^k f_i g_{ij}$. Quitte \`a \^oter
de~$g_{ij}$ les termes de degr\'es autres que~$t-d$, on peut supposer
que chaque~$g_{ij}$ est homog\`ene de degr\'e~$t-d$. 
Soit $D$ un d\'enominateur commun des coefficients des polyn\^omes~$g_{ij}$.
La relation
\begin{equation} \label{nullstellensatz}
D x_i^t= \sum_{i=0}^k f_i(x) D g_{ij}(x) \end{equation}
et l'in\'egalit\'e~\eqref{evident} appliqu\'ee aux polyn\^omes~$D g_{ij}$
entra\^{\i}ne une majoration de la forme
\[ D \abs{x_i}^t \leq c_1 \max (\abs{f_0(x)},\dots,\abs{f_k(x)})
   \max (\abs{x_0},\dots,\abs{x_k})^{t-d}, \]
o\`u $c_1$ est un nombre entier.
De l\`a en resulte l'in\'egalit\'e
\[ \max(\abs{x_0},\dots,\abs{x_k})^t \leq c_2 \max (\abs{f_0(x)},\dots,\abs{f_k(x)}) \max(\abs{x_0},\dots,\abs{x_k})^{t-d} \]
puis
\[  d \log \max(\abs{x_0},\dots,\abs{x_k}) \leq c_3 +
  \log \max(\abs{f_0(x)},\dots,\abs{f_k(x)}). \]
Comme les~$x_i$ sont premiers entre eux, 
la relation~\eqref{nullstellensatz} implique que le pgcd des~$f_j(x)$
divise~$D$. A fortiori, $\delta\leq D$
et 
\[ d h(x) \leq c_3+\log D + \log \max(\abs{f_0(x)/\delta},\dots,\abs{f_k(x)/\delta}) \leq c_4+ h( f(x)), \]
ainsi qu'il fallait d\'emontrer.
\end{proof}

Apr\`es cette d\'emonstration, insistons sur le fait que l'in\'egalit\'e
de gauche --- la minoration de la hauteur de l'image de~$x$ ---
a requis le th\'eor\`eme des z\'eros de \textsc{Hilbert} et donc
l'hypoth\`ese  que les polyn\^omes~$f_i$ sont sans z\'ero commun dans~$\P^k(\C)$.
La seule absence de z\'ero commun dans~$\P^k(\Q)$ n'aurait pas une telle
cons\'equence alg\'ebrique et ne permettrait pas d'\'etablir la minoration 
voulue, voir par exemple l'exercice~\ref{exo.cremona.2}
du chapitre~\ref{chap.sysdyn} pour un contre-exemple explicite.

\medskip

On doit \`a~\cite{northcott1950} d'avoir mis en \'evidence
l'importante propri\'et\'e de finitude
\'enonc\'ee dans la prop.~\ref{prop.finitude-Q}, pr\'ecis\'ement en vue de la
cons\'equence suivante.

\'Etant donn\'e un syst\`eme dynamique $f\colon X\ra X$
d'un ensemble~$X$, nous dirons qu'un point $x\in X$
est \emph{p\'eriodique} pour~$f$ s'il existe
un entier~$n\geq 1$ tel que $f^n(x)=x$, o\`u $f^n$ d\'esigne
le $n$-i\`eme it\'er\'e $f\circ\cdots\circ f$ de~$f$.
Nous dirons qu'un point~$x$ est \emph{pr\'ep\'eriodique}
si l'un de ses it\'er\'es est p\'eriodique ; cela revient
exactement \`a dire que l'orbite $\{x,f(x),f^2(x),\ldots\}$
de~$x$ est un ensemble fini.

\begin{coro}[Northcott]
\label{coro.northcott-Q}
Soit $f\colon\P^k\ra\P^k$ un endomorphisme de degr\'e~$d$
\`a coefficients rationnels. Supposons $d\geq 2$.
Il existe un nombre r\'eel~$m(f)$ tel que pour tout 
point~$x$ de~$\P^k(\Q)$ qui est pr\'ep\'eriodique pour~$f$,
on ait $h(x)\leq m(f)$. 
En particulier, il n'y a dans~$\P^k(\Q)$ qu'un nombre fini de points
pr\'ep\'eriodiques pour~$f$.
\end{coro}
\begin{proof}
L'encadrement de la prop.~\ref{prop.fonctorialite-Q}
entra\^{\i}ne par r\'ecurrence l'in\'egalit\'e
\[   d^n h(x)- c \frac{d^n-1}{d-1} \leq 
h(f^n(x)) \leq d^n h(x) + c \frac{d^{n}-1}{d-1} , \]
valable pour tout entier~$n\geq 0$ et tout point $x\in\P^k(\Q)$.
Si un point $x$ est pr\'ep\'eriodique,
son orbite $\{x,f(x),f^2(x),\dots\}$ est finie,
donc il existe des entiers $n\geq 0$ et~$p\geq 1$ tels
que $f^n(x)=f^{n+p(x)}$.
Posons $y=f^n(x)$.  On a donc $f^p(y)=y$ et 
\[ h(x)  \leq \frac1{d^n}h(y) + c \frac1{d-1} \leq h(y)+ c\frac1{d-1}.\]
De m\^eme,
\[ h(y) \leq  \frac1{d^p}h(f^p(y))+c\frac1{d-1}
\leq \frac1{d} h(y) + c\frac1{d-1}, \]
d'o\`u 
\[ h(y)\leq  c \frac d{(d-1)^2}. \]
Finalement, la hauteur de tout point pr\'ep\'eriodique $x\in\P^k(\Q)$ v\'erifie
\[ h(x) \leq c \frac{2d-1}{(d-1)^2}, \]
ce qui d\'emontre le corollaire.
\end{proof}

On pourra observer que la borne obtenue pour la hauteur
d'un point pr\'ep\'eriodique est assez explicite. Outre le degr\'e~$d$,
elle fait intervenir les coefficients de l'endomorphisme~$f$ 
via la constante~$c$
de la proposition~\ref{prop.fonctorialite-Q}.
Expliciter cette derni\`ere constante  est possible,
mais subtil : s'il est \'evident d'expliciter
 la majoration fournie par cette proposition,
la minoration repose sur l'utilisation du th\'eor\`eme
des z\'eros de Hilbert dont les premi\`eres
versions effectives n'ont \'et\'e d\'emontr\'ees que dans
la fin des ann\'ees~1980 (voir~\cite{teissier1990} pour un survol
de ce probl\`eme).

\subsection{Hauteur normalis\'ee}

Bien d'autres fonctions que la fonction~$h$ introduite ici
sont d'une utilit\'e comparable pour l'arithm\'etique.
Notons par exemple que le choix d'un autre syst\`eme
de coordonn\'ees sur l'espace projectif (c'est-\`a-dire
la composition avec une homographie) fournirait une autre
fonction hauteur $h'$, certes telle que $h'-h$ est born\'ee
en vertu de la proposition~\ref{prop.fonctorialite-Q}.
La situation des syst\`emes dynamiques fournit 
une variante tr\`es commode de la hauteur,
syst\'ematis\'ee par~\cite{call-s93} mais dont le principe
remonte \`a \textsc{N\'eron} et \textsc{Tate}. 

Avant de continuer, observons l'exemple simple du syst\`eme
dynamique sur~$\P^1$ donn\'e par l'\'el\'evation des coordonn\'ees
homog\`enes \`a la puissance~$d$, autrement dit
$f([x_0:x_1])=[x_0^d:x_1^d]$,
soit encore $f_0=X_0^d$ et $f_1=X_1^d$.
Les points pr\'ep\'eriodiques dans~$\P^1(\C)$ sont les points $[0:1]$,
$[1:0]$, ainsi que les points de coordonn\'ees homog\`enes~$[1:\zeta]$,
o\`u $\zeta$ est une racine de l'unit\'e.
Parmi ces points, seuls $[0:1]$, $[1:0]$, $[1:1]$ et~$[1:-1]$
sont rationnels.
Observons aussi que l'in\'egalit\'e reliant $h(x)$ et $h(f(x))$
est dans ce cas une \emph{\'egalit\'e}:
\[ h([x_0^d:x_1^d])= d h([x_0:x_1]) .\]
En modifiant l\'eg\`erement la hauteur, 
nous allons g\'en\'eraliser cette \'egalit\'e.

\begin{prop}\label{prop.tate-Q}
Soit $f\colon\P^k\ra\P^k$ un endomorphisme de l'espace
projectif donn\'e par des polyn\^omes homog\`enes de degr\'e~$d$
sans z\'ero commun dans~$\P^k(\C)$.
Supposons $d\geq 2$.
Il existe alors une unique fonction $\hat h\colon\P^k(\Q)\ra\R$
telle que $\hat h-h$ soit born\'ee et telle que $\hat h(f(x))=dh(x)$
pour tout $x\in\P^k(\Q)$.
\end{prop}
\begin{proof}
Munissons l'espace des fonctions born\'ees de~$\P^k(\Q)$ dans~$\R$
de la norme uniforme et munissons l'espace affine~$E$ des fonctions
telles que $\phi-h$ soit born\'ee de la distance induite.
C'est un espace m\'etrique complet.
L'application $T\colon \phi\mapsto \frac1d \phi\circ f$
est lin\'eaire et applique~$E$ dans lui-m\^eme, car $T(h)-h$ est born\'ee.
Cette application est contractante, de constante de Lipschitz au plus~$1/d<1$.
Elle poss\`ede donc un unique point fixe dans~$E$.
\end{proof}

La fonction~$\hat h$ est appel\'ee \emph{hauteur normalis\'ee}.
Notons la formule
\begin{equation}
\label{eq.normalisee-Q}
\hat h(x) = \lim_{n\ra\infty} \frac1{d^n} h(f^n(x)) .
 \end{equation}
La d\'emonstration directe de la convergence de la la limite
permet de donner une d\'emonstration d'apparence un peu plus constructive
de la proposition pr\'ec\'edente (mais essentiellement identique).
La hauteur normalis\'ee v\'erifie les propri\'et\'es suivantes:

\begin{prop}\label{prop.hauteur-normalisee-Q}
\begin{enumerate}
\item on a $\hat h(x)\geq 0$ pour tout $x\in\P^k(\Q)$;
\item un point $x\in\P^k(\Q)$ v\'erifie $\hat h(x)=0$
si et seulement s'il est pr\'ep\'eriodique ;
\item pour tout nombre r\'eel~$B$, l'ensemble des points $x\in\P^k(\Q)$
tels que $\hat h(x)\leq B$ est fini.
\end{enumerate}
\end{prop}
\begin{proof}
La propri\'et\'e~\emph c) r\'esulte imm\'ediatement de la propri\'et\'e
analogue pour~$h$ et de ce que $\hat h-h$ est born\'ee.
Comme $h$ est positive ou nulle,
la propri\'et\'e~\emph a) est manifeste sur la formule~\eqref{eq.normalisee-Q}
ci-dessus. Elle se d\'eduit aussi, ainsi que la propri\'et\'e~\emph b),
de l'assertion de finitude.

Soit en effet $x\in\P^k(\Q)$. On a $\hat h(f^n(x))=d^n \hat h(x)$.
Si $x$ est pr\'ep\'eriodique, il existe des entiers~$n\geq 0$ et~$p\geq 1$
tels que $f^n(x)=f^{n+p}(x)$. Par suite,
$d^n \hat h(x)=d^{n+p}\hat h(x)$, d'o\`u $\hat h(x)=0$ car $d\geq 2$.
Inversement, si $\hat h(x)\leq 0$, les termes de la suite $((f^n(x))$ 
forment un ensemble de points de hauteur normalis\'ee n\'egative
ou nulle,
donc fini d'apr\`es~\emph c). Il existe donc des entiers~$n\geq 0$
et~$p\geq 1$ tels que $f^n(x)=f^{n+p}(x)$.
Autrement dit, $x$ est pr\'ep\'eriodique et $\hat h(x)=0$.
\end{proof}

Tels qu'\'enonc\'es ci-dessus, c'est-\`a-dire restreints au cas
du corps des nombres rationnels, les r\'esultats pr\'ec\'edents
ne sont pas suffisants.  
Concernant par exemple les points pr\'ep\'eriodiques,
ils ne donnent des renseignements
que sur ceux qui poss\`edent un syst\`eme de coordonn\'ees homog\`enes rationnelles,
et ceux-ci sont en nombre fini d'apr\`es le th\'eor\`eme de finitude.
Pourtant, ainsi que le montre
l'exemple du syst\`eme dynamique sur~$\P^1$ donn\'e par l'\'el\'evation des coordonn\'ees 
\`a la puissance~$d$, l'ensemble des points pr\'ep\'eriodiques
(\`a coordonn\'ees homog\`enes complexes, mais en fait alg\'ebriques)
est infini. Ainsi que nous le verrons
plus bas, il est m\^eme dense pour la topologie de Zariski
(prop.~\ref{theo.per.dense} du chapitre~\ref{chap.sysdyn}).
De plus, les \'enonc\'es pr\'ec\'edents ne concernent que l'espace
projectif et il convient d'\'etablir les propri\'et\'es g\'en\'erales
des hauteurs et d'en d\'egager les applications aux syst\`emes dynamiques
polynomiaux d'une vari\'et\'e alg\'ebrique arbitraire.

\section{Hauteur d'un point alg\'ebrique}

Dans ce paragraphe, j'explique comment d\'efinir la hauteur
d'un point de l'espace projectif dont un syst\`eme
de coordonn\'ees homog\`enes est form\'e de nombres alg\'ebriques. 
Il existe dans la litt\'erature
de nombreuses autres pr\'esentations, plus ou moins \'el\'ementaires,
citons notamment \cite{lang1983,serre1997,hindry-silverman2000,bombieri-gubler2006}.

\subsection{Quelques rappels de th\'eorie alg\'ebrique des nombres}

Notons $\bar\Q$ le corps des nombres alg\'ebriques: par d\'efinition,
c'est l'ensemble des nombres complexes qui sont annul\'es par un polyn\^ome
unitaire \`a coefficients rationnels. Un tel nombre complexe
est annul\'e par un polyn\^ome unitaire de degr\'e minimal : son polyn\^ome
minimal; ce polyn\^ome est irr\'eductible
et son degr\'e  est appel\'e degr\'e du nombre alg\'ebrique.

On appellera \emph{corps de nombres} un sous-corps~$K$ de~$\C$
qui est de fimension finie comme~$\Q$-espace vectoriel;
cette dimension est aussi appel\'ee degr\'e et not\'ee $[K:\Q]$.
Un tel corps~$K$ est en fait constitu\'e de nombres alg\'ebriques
(si $x\in K$, le polyn\^ome minimal de l'endomorphisme $\Q$-lin\'eaire
de multiplication par~$x$ dans~$K$ annule~$x$).

Si $a$ est un nombre alg\'ebrique de degr\'e~$d$, le sous-anneau 
$\Q[a]$ engendr\'e par~$a$  dans~$\C$ est une $\Q$-alg\`ebre
isomorphe \`a~$\Q[X]/(P)$, o\`u $P$ est le polyn\^ome minimal de~$a$.
Par suite, $\Q[a]$ est un $\Q$-espace vectoriel de dimension~$d$, 
et aussi un corps.  C'est donc un corps de nombres.
Plus g\'en\'eralement,
si $a_1,\dots,a_r$ sont des nombres alg\'ebriques,
le sous-corps de~$\C$, $\Q(a_1,\dots,a_r)$, qu'ils engendrent
est un corps de nombres. En fait, tout corps de nombres~$K$
est de la forme~$\Q(a)$ pour un certain \'el\'ement~$a$
(\emph{th\'eor\`eme de l'\'el\'ement primitif}).
Notons~$P$ le polyn\^ome minimal de~$a$; son degr\'e est le
degr\'e de~$K$. Comme $P$ est irr\'eductible (et~$\Q$ de caract\'eristique z\'ero),
ses racines complexes sont deux \`a deux distinctes ;
notons-les $a_1,\dots,a_D$. Ce sont les conjugu\'es de~$a$
(qui est l'un d'entre eux).

Pour tout~$i\in\{1,\dots,D\}$,
il existe un unique homomorphisme de corps $\sigma_i\colon K\ra\C$
qui applique $a$ sur~$a_i$. En outre, tout homomorphisme
de corps de~$K$ dans~$\C$ est de cette forme. Rappelons si besoin
est qu'un tel homomorphisme est injectif; on dira que c'est
un \emph{plongement} de~$K$ dans~$\C$.
Si $L$ est un sous-corps de~$K$ de degr\'e~$d$,
chacun des $d$ plongements de~$L$ dans~$\C$ se prolonge
en exactement $D/d$ plongements de~$K$ dans~$\C$.
(L'entier $D/d$ est \'egal \`a la dimension $[K:L]$ de~$K$ comme $L$-espace
vectoriel.)

Un entier alg\'ebrique est un nombre alg\'ebrique qui est annul\'e
par un polyn\^ome unitaire \`a coefficients entiers ; d'apr\`es
le \emph{lemme de Gau\ss{}}, il revient
au m\^eme d'exiger que son polyn\^ome minimal soit \`a coefficients entiers.
Notons que si $x$ est une racine du polyn\^ome \`a coefficients
entiers $P=a_0X^d+\dots+a_d$, avec $a_0\neq 0$, 
alors 
\[ 0=a_0^{d-1}P(x)=(a_0 x)^d + a_0a_1 (a_0 x)^{d-1}+\dots + a_0^{d-1}a_d, \]
ce qui montre que $a_0x$ est un entier alg\'ebrique.
L'ensemble des entiers alg\'ebriques est un sous-anneau de~$\bar\Q$
dont $\bar\Q$ est le corps des fractions.
Si $a_1,\dots,a_r$ sont des entiers alg\'ebriques,
le sous-anneau qu'ils engendrent dans~$\C$,
$\Z[a_1,\dots,a_r]$, est un $\Z$-module libre de rang fini,
\'egal  au degr\'e de~$\Q(a_1,\dots,a_r)$.
Si $K$ est un corps de nombres, l'ensemble des entiers alg\'ebriques
appartenant \`a~$K$ est de cette forme; on le notera $\Z_K$.

Soit $K$ un corps de nombres. La \emph{norme}
d'un \'el\'ement $a$ de~$K$, not\'ee $N_K(a)$, voire $N(a)$
s'il n'y a pas d'ambigu\"{\i}t\'e sur~$K$,
est d\'efinie comme le d\'eterminant
de l'endomorphisme~$\Q$-lin\'eaire de~$K$ donn\'e par la multiplication par~$a$.
Il y a plusieurs fa\c{c}ons de la calculer.

Soit d'abord $P$ le polyn\^ome minimal de~$a$ et~$d$ son degr\'e.
Soit $(x_1,\dots,x_r)$ une base de~$K$ vu comme $\Q(a)$-espace
vectoriel. Alors, la famille $(a^i x_j)$, pour $0\leq i\leq d-1$
et $1\leq j\leq r$, est une base de~$K$ sur~$\Q$. En particulier,
$D=dr$.  Dans cette base, la multiplication par~$a$ est diagonale
par blocs, chacun \'etant la matrice compagnon~$C_P$ du polyn\^ome~$P$.
On a ainsi $N(a)=(\det C_P)^r$.
Le d\'eterminant de~$C_P$ est \'egal au produit
des racines complexes de~$P$; ces racines ne sont autres que
les conjugu\'es $a_1,\dots,a_d$ de~$a$. Si $(\sigma_i)$, pour $1\leq i\leq d$,
d\'esigne la famille des plongements de~$\Q(a)$ dans~$\C$, on a donc
\[ N(a) = \big( a_1\cdots a_d \big)^r = \big ( \prod_{i=1}^d \sigma_i(a) \big)^r
          = \prod_{\sigma\colon K\hra\C} \sigma(a). \]
Supposons de plus que $a$ soit un entier alg\'ebrique non nul.
Alors, $a\Z_K$ est un sous-$\Z$-module de rang~$D$ de~$K$.
D'apr\`es le th\'eor\`eme des diviseurs \'el\'ementaires,
il existe une $\Z$-base  de~$\Z_K$, disons $(x_1,\dots,x_D)$,
et des nombres entiers $u_1,\dots,u_D$ tels que $(u_1 x_1,\dots,u_Dx_D)$
soit une $\Z$-base de~$a\Z_K$. Comme le d\'eterminant de la matrice
de passage d'une $\Z$-base de~$\Z_K$ \`a une autre est
\'egal \`a $\pm 1$, le d\'eterminant de la multiplication par~$a$
est \'egal au produit $u_1\dots u_D$, ou \`a l'oppos\'e.
D'autre part, $u_1\dots u_D$ appara\^{\i}t comme le cardinal
du $\Z$-module --- en fait de l'anneau --- quotient $\Z_K/a\Z_K$.
D'o\`u la seconde formule,
\[ N(a) =  \pm \Card (\Z_K/a\Z_K). \]
Si $I$ est un id\'eal non nul de~$\Z_K$, l'anneau quotient
$\Z_K/I$ est fini (si $a\in I$, l'anneau $\Z_K/I$ est un quotient
de l'anneau $\Z_K/(a)$); par d\'efinition, son cardinal
est appel\'ee la \emph{norme} de~$I$.
La formule pr\'ec\'edente pour $N(a)$ se g\'en\'eralise en 
$N(aI)=\abs{N(a)}N(I)$, pour $a\in \Z_K$ et $I$ un id\'eal de~$\Z_K$.

Si $L$ est un corps de nombres qui contient~$K$
et $a$ un \'el\'ement de~$K$, on a $N_L(a)=N_K(a)^\delta$,
o\`u $\delta=[L:K]$ d\'esigne le quotient du degr\'e de~$L$ par celui de~$K$.
Plus g\'en\'eralement, si $I$ est un id\'eal non nul de~$\Z_K$,
l'id\'eal $I\Z_L$ de~$\Z_L$ est de norme $N_K(I)^\delta$.
(Si $I=a_1\Z_K+\dots+a_r\Z_K$, notons que $I\Z_L=a_1\Z_L+\dots+a_r\Z_L$.)

\subsection{D\'efinition de la hauteur}

Soit $x$ un point de $\P^k(\bar\Q)$ et soit $[x_0:\cdots:x_k]$
des coordonn\'ees homog\`enes de~$x$, o\`u $x_i\in\bar\Q$ pour tout~$i$.
Soit $K$ un corps de nombres contenant les~$x_i$;
notons~$D$ le degr\'e de~$K$ et $\sigma_1,\dots,\sigma_D$
les $D$ homomorphismes de~$K$ dans~$\C$.
Quitte \`a multiplier les~$x_i$ par un m\^eme entier naturel non nul,
on peut supposer que ce sont des entiers alg\'ebriques,
donc des \'el\'ements de~$\Z_K$.
On d\'efinit alors la hauteur de~$x$ par l'expression
\begin{equation} \label{eq.def.hauteur}
  h(x)  = -\frac1D \log N_K(x_0\Z_K+\dots+x_k\Z_K)
  + \frac1D\sum_{j=1}^D \log \max(\abs{\sigma_j(x_0)},
\dots\abs{\sigma_j(x_k)}). \end{equation}
La correction de cette d\'efinition d\'epend de deux v\'erifications suppl\'ementaires:
l'ind\'ependance vis \`a vis du choix du corps de nombres~$K$
et des coordonn\'ees homog\`enes~$x_0,\dots,x_k$.
Notons $h_K(x_0,\dots,x_k)$ 
le membre de droite de l'\'equation~\eqref{eq.def.hauteur}
et commen\c{c}ons par d\'emontrer qu'il ne d\'epend pas
du choix d'un syst\`eme de coordonn\'ees homog\`enes.
Si $(x'_0,\dots,x'_k)$ est un second syst\`eme de coordonn\'ees homog\`enes
d\'efinissant le point~$x$, sujettes aux conditions $x'_i\in \Z_K$
pour tout~$i$,
il existe un \'el\'ement $u\in K^*$ tel que $x'_i=u x_i$
pour tout~$i$. \'Ecrivons $u$ comme une fraction $a/b$
d'\'el\'ements de~$\Z_K$, on voit que l'on a $bx'_i=ax_i$
pour tout~$i$.

Notons alors 
que 
\[ N_K(ax_0\Z_K+\dots+ax_k\Z_K)=\abs{N_K(a)} N_K(x_0\Z_K+\dots+x_k\Z_K) \]
et
\begin{align*}
 \sum_{j=1}^D \log \max(\abs{\sigma_j(ax_0)}, \dots\abs{\sigma_j(ax_k)})
 & = \sum_{j=1}^D \log\abs{\sigma_j(a)} + \sum_{j=1}^D \log \max(\abs{\sigma_j(x_0)}, \dots\abs{\sigma_j(x_k)}) \\
& = \log \abs{N_K(a)} + \sum_{j=1}^D \log \max(\abs{\sigma_j(x_0)},
\dots\abs{\sigma_j(x_k)}).
\end{align*}
Par suite,
$  h_K(ax_0,\dots,ax_k)=h_K(x_0,\dots,x_k)$.
De m\^eme, $h_K(bx'_0,\dots,bx'_k)=h_K(x'_0,\dots,x'_k)$,
d'o\`u l'ind\'ependance par rapport au choix des coordonn\'ees
homog\`enes.

La d\'emonstration de l'ind\'ependance de la hauteur par rapport au
choix du corps de nombres est plus simple.
Si $K$ et $K'$ sont des corps de nombres contenant chacun
un syst\`eme de coordonn\'ees homog\`enes du point~$x$,
le corps~$L$ engendr\'e par ces deux corps en est un \'egalement,
qui contient \`a la fois~$K$ et~$K'$.

Soit $E$ le degr\'e de~$L$; c'est un multiple de~$D$.
De fait, la restriction \`a~$K$ d'un homomorphisme $\sigma\colon L\ra\C$
est un plongement de~$K$ dans~$\C$ et chacun des~$D$ plongements
de~$K$ dans~$\C$ s'\'etend d'exactement $E/D$ mani\`eres distinctes
en un plongement de~$E$ dans~$\C$.
Par cons\'equent, dans la formule d\'efinissant $h_L(x_0,\dots,x_k)$,
la somme
\[  \sum_{\sigma\colon L\hra\C} \log\max(\abs{\sigma(x_0)},
\dots,\abs{\sigma(x_k)}) \]
reprend $E/D$ fois chaque terme de la somme correspondante
sur le corps~$K$.
Compte tenu des facteurs de normalisation $1/D$ et $1/E$,
ces parties des deux formules donnent le m\^eme r\'esultat.
La d\'emonstration de l'\'egalit\'e $h_K(x_0,\dots,x_k)
=h_L(x_0,\dots,x_k)$ est donc d\'etermin\'ee,
compte tenu de l'\'egalit\'e de normes d'id\'eaux
\[ N_L(x_0\Z_L+\dots+x_k\Z_L) = N_K(x_0\Z_K+\dots+x_k\Z_K)^{E/D} . \]

Une cons\'equence de la formule d\'efinissant la hauteur
est son invariance  sous l'action du groupe de Galois
$\Gal(\bar \Q/\Q)$. On d\'efinit en effet une action 
de $\Gal(\bar\Q/\Q)$ sur~$\P^n(\bar\Q)$ en posant
$\tau(x)=[\tau(x_0):\cdots:\tau(x_k)]$
pour tout point $x\in\P^k(\bar\Q)$, de coordonn\'ees
homoh\`enes $[x_0:\cdots:x_k]$ dans~$\bar\Q$,
et tout \'el\'ement $\tau\in\Gal(\bar\Q/\Q)$.

\begin{prop}\label{prop.h-galois}
Pour tout $x\in\P^k(\bar\Q)$
et tout automorphisme $\tau\in\Gal(\bar\Q/\Q)$, on a
$h(\tau(x))=h(x)$.
\end{prop}
\begin{proof}
Soit $x$ un point de $\P^k(\bar\Q)$
et soit $K$ une extension galoisienne de~$\Q$
contenant un syst\`eme $[x_0:\cdots:x_k]$ de coordonn\'ees
homog\`enes de~$x$.
Par d\'efinition, $K$ est stable par tout automorphisme 
$\tau\in\Gal(\bar\Q/\Q)$ et $\tau|_K$ est un automorphisme de~$K$,
ainsi que de son anneau d'entiers~$\Z_K$.
Il en r\'esulte que l'id\'eal $x_0\Z_K+\dots+x_k\Z_K$
a m\^eme norme que son image $\tau(x_0)\Z_K+\dots+\tau(x_k)\Z_K$
par~$\tau$.
De m\^eme, les plongements de~$K$ dans~$\C$ sont tous obtenus
par composition d'un plongement fixe par les \'el\'ements
de $\Gal(K/\Q)$. Par cons\'equent, les termes des sommes
\[ \sum_{\sigma\colon K\hra\C} \log \max(\abs{\sigma(x_0)},
\dots,\abs{\sigma(x_k)}) \quad\text{et}
\sum_{\sigma\colon K\hra\C} \log \max(\abs{\sigma(\tau(x_0))},
\dots,\abs{\sigma(\tau(x_k))}) \]
ne diff\`erent que par l'ordre.
La proposition en r\'esulte.
\end{proof}

Les paragraphes suivants donnent des variantes de la d\'efinition
de la hauteur, ainsi que des exemples.

\subsection{Propri\'et\'e de B\'ezout pour les nombres alg\'ebriques}

Il n'est pas vrai que l'anneau des nombres
alg\'ebriques $\bar\Z$ soit un anneau principal.
Ce n'est en effet m\^eme pas un anneau noeth\'erien:
l'id\'eal $I$ de~$\bar\Z$ form\'e des \'el\'ements dont
une puissance est un multiple de~$2$  n'est par exemple
pas de type fini. Montrons par l'absurde qu'il n'est
pas principal : soit $a\in I$ tel que $I=(a)$.
Comme $2^{1/m}\in I$, $2^{1/m}$ est multiple de~$a$.
Comme $2^{(m-1)/m}=2/2^{1/m}$ n'est pas multiple de~$2$ dans~$\bar\Z$,
$a^{m-1}$ n'est pas multiple de~$2$.
Comme $m$ est arbitraire, aucune puissance de~$a$ n'est multiple
de~$2$, ce qui est absurde.

En revanche, il est vrai que les id\'eaux de type fini 
de~$\bar\Z$ sont principaux (ce qu'on r\'esume parfois en
disant que $\bar\Z$ a la propri\'et\'e de B\'ezout). Cela se
d\'eduit facilement
de ce que le groupe des classes d'id\'eaux d'un anneau
d'entiers de corps de nombres est de torsion.
Par suite, quitte \`a consid\'erer une extension convenable
$L$ de~$K$, on peut choisir les coordonn\'ees homog\`enes
$[x_0:\cdots:x_k]$
d'un point~$x$ de~$\P^k(K)$ comme suit: ce sont
des \'el\'ements de~$\Z_L$ et l'id\'eal $x_0\Z_L+\dots+x_k\Z_L$
de~$\Z_L$ est \'egal \`a~$\Z_L$.
Dans la formule~\eqref{eq.def.hauteur}, cela fait dispara\^{\i}tre
le premier terme.

\subsection{Valuations et hauteurs}

Les id\'eaux d'un anneau d'entiers de corps de nombres
ne sont pas toujours faciles \`a manipuler, notamment
parce qu'ils ne sont pas n\'ecessairement principaux.
Le premier terme de  la hauteur qui met en jeu la norme d'un id\'eal
est parfois malcommode.
Dans ce paragraphe, nous l'\'ecrivons comme une somme
de termes formellement analogues au second terme,
mais o\`u apparaissent d'autres corps que le corps des nombres
complexes.

\begin{defi}
Soit $F$ un corps. Une \emph{valeur absolue}
sur~$F$ est une application $\abs{\cdot}\colon F\ra\R_+$
qui v\'erifie les propri\'et\'es suivantes: pour tous $a$ et~$b$ dans~$F$,
\begin{enumerate}
\item $\abs{ab}=\abs a \abs b$ (multiplicativit\'e); 
\item $\abs{a+b}\leq\abs a+\abs b$ (in\'egalit\'e triangulaire);
\item $\abs a=0$ \'equivaut \`a $a=0$.
\end{enumerate}
\end{defi}

Pour la th\'eorie alg\'ebrique des nombres, la classification
des valeurs absolues sur~$\Q$ est d'une importance capitale.
Donnons-en des exemples. Il y a d\'ej\`a la restriction \`a~$\Q$
de la valeur absolue usuelle.
qu'on note parfois $\abs{\cdot}_\infty$ pour la distinguer de
celles que je vais bient\^ot d\'efinir.
Notons que cette valeur absolue $\abs{\cdot}_\infty$ est \emph{archim\'edienne},
car pour tout $a\in\Q^*$ et tout $M>0$, il existe $n\in\N$ tel que $\abs{na}_\infty>M$
(axiome d'Archim\`ede).

D'autre part, soit $p$ un nombre premier.
Tout nombre rationnel non nul~$a$ s'\'ecrit de mani\`ere
unique sous la forme $p^n u $, avec $n\in\Z$,
et $u\in\Q$ le quotient de deux entiers relatifs premiers \`a~$p$.
Posons alors $\abs a_p=p^{-n}$. Posons aussi, comme il se doit, $\abs 0_p=0$.
Cela d\'efinit une valeur absolue sur~$\Q$ que l'on appelle
la \emph{valeur absolue $p$-adique}. Seule l'in\'egalit\'e triangulaire
n'est pas \'evidente ; \'ecrivons donc $a=p^n\frac{a'}{a''}$
et $b=p^m \frac{b'}{b''}$, avec $n$, $m\in\Z$, $a'$, $a''$,
$b'$, $b''$ des nombres entiers relatifs premiers \`a~$p$. Supposons aussi,
ce qui est loisible, $n\leq m$.
Alors, 
\[ a+b = p^n \left( \frac{a'}{a''}+ p^{m-n}\frac{b'}{b''}\right)
 = p^n \frac{a'b''+p^{m-n}b'a''}{a''b''}. \]
Dans cette derni\`ere fraction, le d\'enominateur est premier \`a~$p$,
mais pas forc\'ement le num\'erateur. Par suite,
$ a+b$ s'\'ecrit sous la forme $p^s \frac{c'}{c''}$
avec $s\geq n$ et 
\[  \abs{a+b}_p=p^{-s}\leq p^{-n}=\max(p^{-n},p^{-m})
 = \max(\abs{a}_p,\abs{b}_p). \]
Ainsi, $\abs{\cdot}_p$ v\'erifie une in\'egalit\'e plus
forte  que l'in\'egalit\'e triangulaire : on dit que
c'est une \emph{valeur absolue ultram\'etrique.}
En particulier, pour tout $a\in \Q$ et tout $n\in\N^*$,
$\abs{na}_p\leq \abs{a}_p$ : cete valuation ne v\'erifie pas
l'axiome d'Archim\`ede.

\begin{theo}[Ostrowski]
Les valeurs absolues sur le corps~$\Q$ des nombres rationnels
sont les suivantes:
\begin{itemize}
\item la valeur absolue dite triviale pour laquelle
$\abs 0=0$ et $\abs a=1$ si $a\neq 0$;
\item 
 pour tout nombre r\'eel~$s$ tel que $0<s\leq 1$,
la puissance $\abs{\cdot}_\infty ^s$
de la valeur absolue archim\'edienne;

\noindent
\hskip-\leftskip
et, pour tout nombre premier~$p$,

\item 
pour tout nombre r\'eel~$s>0$,
la puissance $\abs{\cdot}_p^s$
de la valeur absolue $p$-adique.
\end{itemize}
\end{theo}

Une valeur absolue sur un corps d\'efinit une distance
et donc une topologie. Deux valeurs absolues d\'efinissent
la m\^eme topologie si et seulement si l'une est une puissance
(non nulle) de l'autre. La valeur absolue triviale
fournit la topologie discr\`ete.
Par suite, dans la liste ci-dessus, on ne s'int\'eressera
qu'aux valeurs absolues $p$-adiques standard et \`a la valeur
absolue archim\'edienne standard. Elles sont reli\'ees
par la \emph{formule du produit,} qui n'est autre
qu'une reformulation de la d\'ecomposition en facteurs premiers.

\begin{prop}\label{prop.formule.produit}
Pour tout $a\in\Q^*$, $\abs{a}_\infty \prod_p \abs{a}_p=1$.
\end{prop}
\begin{proof}
Soit $a=\eps \prod_p p^{n_p}$ la d\'ecomposition en facteurs
premiers de~$a$, avec $\eps\in\{-1,+1\}$ et $n_p\in\Z$
pour tout nombre premier~$p$, presque tous \'etant nuls.
Alors, pour tout nombre premier~$p$, on a $\abs a_p=p^{-n_p}$,
tandis que $\abs a_\infty=\prod_p p^{n_p}$.
Le produit de toutes ces quantit\'es est bien \'egal \`a~$1$.
\end{proof}

Fixons une valeur absolue $\abs\cdot$ sur~$\Q$.
Si $K$ est un corps de nombres, il y a en g\'en\'eral
plusieurs fa\c{c}ons d'\'etendre la valeur absolue donn\'ee
en une valeur absolue sur~$\Q$. 
Prenons par exemple la valeur absolue archim\'edienne
et le corps $K=\Q(\sqrt 2)$. 
Du point de vue alg\'ebrique, ind\'ependamment
de l'ordre usuel sur les nombres r\'eels, 
$\sqrt 2$ et $-\sqrt 2$ sont indissociables;
on peut donc prolonger la valeur absolue de deux fa\c{c}ons diff\'erentes,
en posant $\abs{a+b\sqrt 2}_{\infty}' = \abs{a+b\sqrt 2}$
ou $\abs{a+b\sqrt 2}_\infty''=\abs{a-b\sqrt 2}$,
pour $a$ et~$b$ dans~$\Q$.

Toujours dans le cas de la valeur absolue archim\'edienne,
mais dans le cas d'un corps de nombres~$K$ g\'en\'eral,
les diff\'erentes extensions \`a~$K$ de la valeur absolue archim\'edienne
de~$\Q$ correspondent exactement aux valeurs absolues
$\abs{\cdot}_\sigma$ d\'efinies par $\abs{a}_\sigma=\abs{\sigma(a)}$,
o\`u $\sigma$ parcourt l'ensemble des plongements de $K$ dans~$\C$.
(Deux plongements conjugu\'es d\'efinissent la m\^eme valeur absolue.)

Dans le cas des valeurs absolues $p$-adiques,
un corps, que l'on note $\C_p$, joue le m\^eme r\^ole que~$\C$
pour la valeur absolue archim\'edienne.
Observons que $\R$ est le compl\'et\'e du groupe ab\'elien~$\Q$ pour la topologie
d\'efinie par la valeur absolue archim\'edienne,
qu'il est muni d'une valeur absolue (celle que tout le monde
conna\^{\i}t, la seule qui \'etend celle de~$\Q$)  et que $\C$
est sa cl\^oture alg\'ebrique, munie de l'\emph{unique} valeur absolue
qui \'etend la valeur absolue archim\'edienne sur~$\R$.

Si  $p$ est un nombre premier, on commence
par d\'efinir le compl\'et\'e $\Q_p$ de~$\Q$ pour la topologie
d\'efinie par la valeur absolue $p$-adique. C'est un corps et sa topologie
est compatible avec la structure de corps.
Ensuite, on consid\`ere une cl\^oture alg\'ebrique $\bar{\Q_p}$
de~$\Q_p$ : comme $\Q_p$ est complet, on d\'emontre
que $\bar{\Q_p}$ poss\`ede une \emph{unique} valeur absolue,
toujours not\'ee $\abs{\cdot}_p$ qui \'etend la valeur absolue $p$-adique
sur~$\Q$.
Ce corps n'est cependant pas complet pour la topologie
$p$-adique --- on note alors $\C_p$ son compl\'et\'e muni de la valeur
absolue $\abs{\cdot}_p$ qui \'etend la valeur absolue $p$-adique.
C'est un corps complet par construction, et alg\'ebriquement clos
par th\'eor\`eme.

Si $K$ est un corps de nombres, les extensions
\`a~$K$ de la valeur absolue $p$-adique sur~$\Q$
sont de la forme $a\mapsto \abs{\sigma(a)}_p$,
o\`u $\sigma$ d\'ecrit l'ensemble des plongements
de~$K$ dans~$\C_p$. (L'image d'un tel plongement est bien s\^ur contenue
dans~$\bar\Q_p$ et deux plongements qui diff\`erent
par composition d'un automorphisme de~$\bar{\Q_p}$
fournissent la m\^eme valeur absolue.)

\begin{prop}
Soit $K$ un corps de nombres et soit $x\in\P^k(K)$
un point de coordonn\'ees homog\`enes $[x_0:\cdots:x_k]$
(appartenant \`a~$K$).
Alors, 
\begin{equation}\label{def.hauteur-p}
 h(x) = \frac1{[K:\Q]} \sum_{p\leq \infty}
         \sum_{\sigma\colon K\hra\C_p}
               \log\max(\abs{\sigma(x_0)}_p,\dots,\abs{\sigma(x_k)}_p).
\end{equation}
\end{prop}
La sommation sur {\og $p\leq\infty$\fg} signifie que l'on somme
sur les nombres premiers~$p$ et sur un symbole suppl\'ementaire~$\infty$,
avec $\C_\infty=\C$.
Dans la formule pr\'ec\'edente, on peut aussi regrouper
les termes selon les valeurs absolues induites sur le corps~$K$
par les plongements
qui indexent cette somme.

Notons ainsi $M_K$ l'ensemble des valeurs absolues sur~$K$
qui \'etendent une des valeurs absolues $p$-adique ou archim\'edienne
sur~$K$. Un \'el\'ement de~$M_K$ est appel\'e \emph{place} du corps~$K$.
Pour chaque valeur absolue $\abs{\cdot}_v\in M_K$, soit $p_v$
le nombre premier ou l'infini correspondant \`a la valeur
absolue induite sur le sous-corps~$\Q$ de~$K$ et soit $\eps_v$
le nombre de plongements~$\sigma$ de~$K$ dans~$\C_{p_v}$ 
induisant cette valeur absolue sur~$K$, autrement dit
tels que l'on ait $\abs{x}_v=\abs{\sigma(x)}_{p_v}$ pour tout~$x\in K$.

Avec les notations de la proposition, on a ainsi
\begin{equation}\label{eq.def.hauteur-v}
 h(x) = \frac1{[K:\Q]} \sum_{v\in M_K} 
\eps_v \log \max(\abs{x_0}_v,\dots,\abs{x_k}_v).
\end{equation}

Sous-jacent \`a la v\'eracit\'e de la formule pr\'ec\'edente est le fait
que le membre de droite ne d\'epend pas du choix des coordonn\'ees
homog\`enes d\'efinissant le point. De m\^eme que l'\'enonc\'e
analogue avec la premi\`ere d\'efinition de la hauteur d\'ecoulait
d'une relation entre la norme d'un \'el\'ement
et ses valeurs absolues archim\'ediennes,
cela r\'esulte de la \emph{formule du produit} qui
relie toutes les valeurs absolues d'un \'el\'ement non nul d'un corps
de nombres:
pour tout $a\in K^*$,
\begin{equation} \prod_{v\in M_K} \abs{a}_v^{\eps_v} 
  = \prod_{p\leq\infty} \prod_{\sigma\colon K\hra\C_p}
  \abs{\sigma(a)}_p = 1. \end{equation}
En fait, le facteur index\'e par~$p$ dans l'\'equation
pr\'ec\'edente est donn\'e par
\[ \prod_{\sigma K\hra\C_p} \abs{\sigma(a)}_p = \abs{N_K(a)}_p, \]
o\`u $N_K(a)$ est la norme de~$a$.
Cela ram\`ene la formule du produit dans le corps de nombres~$K$ \`a celle, 
d\'ej\`a mentionn\'ee (prop.~\ref{prop.formule.produit}),
dans le corps des nombres rationnels.

\subsection{Mesure de Mahler et hauteur}

Consid\'erons ici le cas de la droite projective~$\P^1$.
Notons $\infty$ le  point de coordonn\'ees homog\`enes $[0:1]$.
Un point $x$ de~$\P^1(\bar\Q)$ a deux coordonn\'ees homog\`enes $[x_0:x_1]$;
si $x\neq\infty$, $x_0\neq 0$ et $\xi=x_1/x_0$ est un \'el\'ement
de~$\bar\Q$. Dire que $x\in\P^1(K)$, pour un sous-corps~$K$ de~$\C$,
signifie exactement que $\xi\in K$. 

La hauteur du point~$\infty$ est \'egale \`a~$0$.
Si $\xi$ est un nombre alg\'ebrique, on note (abusivement)
$h(\xi)$ la hauteur du point~$[1:\xi]$ et on dit
aussi que c'est la hauteur de~$\xi$.
Montrons comment elle est reli\'ee au polyn\^ome minimal de~$\xi$.
Il nous faut tout d'abord rappeler une d\'efinition:
La \emph{mesure de Mahler} d'un polyn\^ome~$P\in\C[X]$ est donn\'ee par
la formule
\begin{equation} 
\mathrm M(P) = \exp\left(\int_0^{1} \log \abs{P(\mathrm e^{2\mathrm i\pi t})}\,dt\right).
\end{equation}

\begin{prop}
Soit $\xi$ un nombre alg\'ebrique, soit $P$ son polyn\^ome minimal
et soit $d$ son degr\'e. On a 
\[ h(\xi) = \frac1d \log \mathrm M(P).\]
\end{prop}
\begin{proof}
Par d\'efinition
\[ h(\xi) = h([1:\xi])= \frac1d \sum_{p\leq\infty} \sum_{\sigma\colon K\hra\C_p}
        \log \max (1,\abs{\sigma(\xi)}_p), \]
o\`u la somme est sur l'ensemble des nombres premiers et le symbole~$\infty$;
notons $h_p(\xi)$ la somme correspondante,
de sorte que $h(\xi)=\sum_{p\leq\infty} h_p(\xi)$.
Posons aussi $P=a_0X^d+a_1X^{d-1}+\cdots+a_d$.

Nous allons d'abord montrer que pour tout nombre premier~$p$,
on a $h_p(\xi)=-d \log\abs{a_0}_p$.
Notons en effet $\xi_1,\dots,\xi_d$ les racines de~$P$
dans~$\C_p$. Ce ne sont autres que les images de~$\xi$
par les diff\'erents plongements du corps~$\Q(\xi)$ dans~$\C_p$.
Supposons, ce qui est loisible, que
$\abs{\xi_1}_p\geq \abs{\xi_2}_p \geq \dots \geq \abs{\xi_d}_p$
et soit $r$ le plus grand entier tel que $\abs{\xi_r}_p\geq 1$.
Les relations coefficients-racines s'\'ecrivent
\[ \frac{a_k}{a_0}
      = (-1)^k\sum_{i_1<\dots<i_k} \xi_{i_1}\dots\xi_{i_k} \]
et entra\^{\i}nent la majoration
\[ \abs{ \frac{a_k}{a_0} }_p \leq \abs{\xi_1}_p \dots \abs{\xi_k}_p
 \leq \prod_{i=1}^d \max(1,\abs{\xi_i}_p), \]
autrement dit
\[ \abs{a_k}_p \leq \abs{a_0}_p \prod_{i=1}^d \max(1,\abs{\xi_i}_p). \]
Pour $k=r$, remarquons que dans l'expression
\[ \frac{a_r}{a_0}
      = (-1)^r\sum_{i_1<\dots<i_r} \xi_{i_1}\dots\xi_{i_r}, \]
le terme $\xi_1\dots\xi_r$ est de valeur absolue strictement sup\'erieure
\`a tous les autres; l'in\'egalit\'e ultram\'etrique entra\^{\i}ne alors
\[  \abs{ \frac{a_r}{a_0}}_p = \abs{\xi_1}_p \cdots \abs{\xi_r}_p
 = \prod_{i=1}^d \max(1,\abs{\xi_i}_p). \]
Par suite,
\[  \max(\abs{a_0}_p, \dots,\abs{a_d}_p)
      = \max(\abs{a_0}_p,\abs{a_r}_p) = \abs{a_0}_p \prod_{i=1}^d \max(1,\abs{\xi_i}_p). \]
Enfin, comme les coefficients $a_0,\dots,a_d$ sont
premiers entre eux, l'un d'entre eux n'est pas multiple de~$p$
et le membre de gauche de l'\'egalit\'e pr\'ec\'edente
est \'egal \`a~$1$.  Ainsi,
\[ h_p(\xi)=\frac1d \sum_{i=1}^d \log \max(1,\abs{\xi_i}_p) =  - \frac1d \log\abs{a_0}_p. \]

Passons maintenant au terme~$h_\infty$.
D'apr\`es la formule de~\textsc{Jensen}, on a, pour tout $\alpha\in\C$,
\[ \int_0^{1} \log\abs{e^{2\mathrm i\pi t}-\alpha}_\infty\,dt
 = \log \max(1,\abs\alpha_\infty)
 = \begin{cases} 0 & \text{si $\abs\alpha_\infty\leq 1$;} \\
  \log\abs\alpha_\infty & \text{sinon.} \end{cases}
\]
Par cons\'equent, 
\[ \log\mathrm M(P) = \log \abs{a_0}_\infty+ \sum_{j=1}^d \log\max(1,\abs{\xi_j}_\infty)
 = \log\abs{a_0}_\infty+ d h_\infty(\xi) .\]

On a donc finalement
\[ d h(\xi) = - \sum_{p\leq \infty}  \log\abs{a_0}_p
 + \log\mathrm M(P) = \log\mathrm M(P) \]
compte tenu de la formule du produit (prop.~\ref{prop.formule.produit}).
\end{proof}

\section{Fonctorialit\'e}

Ce paragraphe explique le comportement de la hauteur sous l'effet
d'un morphisme.

\subsection{Exemples}

Commen\c{c}ons par des exemples tr\`es simples.

Soit $S\colon\P^k\times\P^m\ra \P^{km+k+m}$
le plongement de Segre, associant au couple $(x,y)$
de points de coordonn\'ees homog\`enes $[x_0:\cdots:x_k]$
et $[y_0:\cdots:y_m]$ le point de~$\P^{km+k+m}$ dont un
syst\`eme de coordonn\'ees homog\`enes
est la famille $(x_iy_j)$, index\'ee par l'ensemble des couples $(i,j)$
tels que $0\leq i\leq k$ et $0\leq j\leq m$,
disons ordonn\'e par l'ordre lexicographique.

Pour toute valeur absolue~$\abs{\cdot}$ sur un corps~$F$,
tout $(x_0,\dots,x_k)\in F^{k+1}$ et tout $(y_0,\dots,y_m)\in F^{m+1}$,
on a \'evidemment
\[ \max_{\substack{0\leq i\leq k\\ 0\leq j\leq m}} \abs{x_iy_j}
          = \max_{0\leq i\leq k} \abs{x_i} \max_{0\leq j\leq m} \abs{y_j}.
\]
La d\'efinition de la hauteur par la formule~\eqref{def.hauteur-p}
entra\^{\i}ne alors l'\'egalit\'e
\begin{equation}
 h(S(x,y))= h(x)+h(y), \quad \text{pour tout $x\in\P^k(\bar\Q)$
 et tout $y\in\P^m(\bar\Q)$.} 
\end{equation}

Le plongement de Veronese de degr\'e~$d$, $V_d\colon \P^k\ra\P^K$,
associe \`a un point~$x$ de coordonn\'ees homog\`enes $[x_0:\cdots:x_k]$
le point $V_d(x)$ dont un syst\`eme de coordonn\'ees homog\`enes 
est donn\'e par la famille des mon\^omes $x_0^{d_0}\dots x_n^{d_k}$, 
o\`u $(d_0,\dots,d_k)$
parcourt toutes les suites d'entiers naturels de somme~$d$.
Il y a $\binom{k+d}k$ telles suites, d'o\`u $K=\binom{k+d}k-1$.
Pour toute valeur absolue sur un corps~$F$ et
tout $(x_0,\ldots,x_k)\in F^{k+1}$,  on a 
\[ \abs{x_0^{d_0}\dots x_k^{d_k}}=\abs{x_0}^{d_0}\dots\abs{x_k}^{d_k}
 \leq \max(\abs{x_0},\dots,\abs{x_k})^{d}, \]
l'\'egalit\'e \'etant atteinte lorsque $d_r=d$ et $\abs{x_r}$ maximal.
On a ainsi
\begin{equation}
  h(V_d(x)) = d h(x) \quad\text{pour tout $x\in\P^k(\bar\Q)$.}
\end{equation}

Consid\'erons enfin la projection lin\'eaire
$p\colon \P^k \dashrightarrow \P^m $
qui applique un point $x$ de coordonn\'ees homog\`enes
$[x_0:\cdots:x_k]$ sur le point~$p(x)$
de coordonn\'ees homog\`enes $[x_0:\cdots:x_m]$, $m$ et~$k$
\'etant des entiers tels que $1\leq m < k$.
Son domaine de d\'efinition est exactement le compl\'ementaire du
sous-espace projectif~$E$
d\'efini par l'annulation des $m+1$ premi\`eres coordonn\'ees homog\`enes.
Pour $x\not\in E$, l'in\'egalit\'e \'evidente
\[ \max(\abs{x_0},\dots,\abs{x_m}) \leq \max(\abs{x_0},\dots,\abs{x_k})\]
entra\^{\i}ne l'in\'egalit\'e
\begin{equation}
 h(p(x)) \leq h(x) \quad\text{pour tout $x\in (\P^n\setminus E)(\bar\Q)$.}
\end{equation}
En revanche, notons qu'il n'y a pas d'in\'egalit\'e dans l'autre sens.
Par exemple, consid\'erons la projection~$p$ de~$\P^2$ dans~$\P^1$
donn\'ee par $p([x_0:x_1:x_2])=[x_0:x_1]$;
pour $x=[1:0:n]$,  avec $n\in\Z^*$, on a 
$h(x)=\log\abs n$ et $h(p(x))=h([1:0])=0$.

\subsection{Comportement par morphisme}

\begin{prop}\label{prop.fonctorialite}
Soit $f\colon \P^k\dashrightarrow\P^m$ une application rationnelle
d\'efinie par $m+1$ polyn\^omes 
$(f_0,\dots,f_m)$
en~$k+1$ variables, \`a coefficients dans~$\bar\Q$
homog\`enes de degr\'e~$d$
et sans facteur commun.
Notons $Z$ le lieu d'ind\'etermination de~$f$, 
lieu des z\'eros communs dans~$\P^k$ des polyn\^omes~$f_0,\dots,f_m$.

\begin{enumerate}
\item Il existe un nombre r\'eel~$c_1$
tel que pour tout $x\in(\P^k\setminus Z)(\bar\Q)$,
$h(f(x))\leq d h(x)+c_1$.
\item Pour toute sous-vari\'et\'e ferm\'ee~$X$ de~$\P^k$
qui ne rencontre pas~$Z$, il existe un nombre r\'eel~$c_X$
tel que $h(f(x))\geq d h(x)-c_X$ pour tout $x\in X(\bar\Q)$.
\end{enumerate}
\end{prop}
\begin{proof}
D\'emontrons d'abord \emph a).

Observons d'abord le lemme suivant: soit $F$ un corps
muni d'une valeur absolue et soit $\phi$ un polyn\^ome homog\`ene de degr\'e~$d$ 
\`a coefficients dans~$F$. Il existe
un nombre r\'eel~$C$ tel que
pour tout corps valu\'e~$K$ contenant~$F$ et dont la valuation prolonge
celle de~$F$
et toute famille $(x_0,\dots,x_k)$ d'\'el\'ements de~$K$,
\[ \abs{\phi(x_0,\dots,x_k)} \leq C \max(\abs{x_0},\dots,\abs{x_k})^d. \]
Cela est \'evident pour un mon\^ome et s'en d\'eduit, gr\^ace \`a l'in\'egalit\'e
triangulaire, par
r\'ecurrence sur le nombre de mon\^omes de~$\phi$.
Si la valeur absolue~$\abs\cdot$ est ultram\'etrique,
on peut en outre choisir pour~$C$ le maximum des valeurs absolues
des coefficients des mon\^omes qui constituent~$\phi$.

Soit maintenant $F$ un corps de nombres contenant les coefficients
des polyn\^omes~$f_i$. Quitte \`a multiplier les polyn\^omes~$f_i$
par un m\^eme entier non nul, on peut en outre supposer que les
coefficients des~$f_i$ sont des entiers alg\'ebriques. 
Par cons\'equent, pour tout nombre premier~$p$,
les valeurs absolues $p$-adiques de ces coefficients sont au plus~$1$
et l'on aura, pour tout corps de nombres~$K$ contenant~$F$, tout nombre
premier~$p$,
tout plongement $\sigma\colon K\hra\C_p$
et tout $x=(x_0,\dots,x_k)\in K^{k+1}$, l'in\'egalit\'e
\[ \max(\abs{\sigma(f_0(x))}_p, \dots,\abs{\sigma(f_m(x))}_p)
 \leq \max(\abs{\sigma(x_0)}_p,\dots,\abs{\sigma(x_k)}_p)^d. \]

Pour les valeurs absolues archim\'ediennes,
il existe de m\^eme, pour tout plongement $\sigma_0\colon F\hra\C$
un nombre r\'eel~$C_{\sigma_0}>0$ tel que l'on ait
\[ \max(\abs{\sigma(f_0(x))}_\infty, \dots,\abs{\sigma(f_m(x))}_\infty)
 \leq C_{\sigma_0} \max(\abs{\sigma(x_0)}_\infty,\dots,\abs{\sigma(x_k)}_\infty)^d
\]
pour tout corps de nombres~$K$ qui contient~$F$, tout plongement $\sigma\colon K\hra\C$ qui prolonge~$\sigma_0$ et tout \'el\'ement $(x_0,\dots,x_k)\in K^{k+1}$.

Soit $K$ un corps de nombres contenant~$F$ et soit $x$
un point de~$\P^k(K)$; choisissons-lui un syst\`eme
de coordonn\'ees homog\`enes $[x_0:\cdots:x_k]$ dans~$K$.
Supposons que $f$ soit d\'efini en~$x$, c'est-\`a-dire
que $f_0(x),\dots,f_m(x)$ ne soient pas tous nuls.
La d\'efinition de la hauteur implique alors la majoration
\begin{align*}
 h(f(x)) & = h([f_0(x):\cdots:f_m(x)]) \\
&= \frac1{[K:\Q]}\sum_{p\leq\infty}  \sum_{\sigma\colon K\hra\C_p} \log \max(\abs{\sigma(f_0(x))}_p, \dots,\abs{\sigma(f_m(x))}_p)
 \\
&\leq \frac1{[K:\Q]}\sum_{p < \infty} \sum_{K\hra\C_p} d \log \max(\abs{\sigma(x_0)}_p,\dots,\abs{\sigma(x_k)}_p) \\
&\qquad
{} + \frac1{[K:\Q]} \sum_{\sigma\colon K\hra\C}
          \left(\log C_{\sigma|_F} + d \log \max(\abs{\sigma(x_0)}_\infty,\dots,\abs{\sigma(x_k)}_\infty) \right)\\
& \leq  d h([x_0:\cdots:x_k]) + \max_{\sigma\colon F\hra\C} \log C_\sigma .
\end{align*}

L'assertion \emph a) ainsi d\'emontr\'ee, passons \`a la seconde
partie de la proposition.
Soit $(P_j)$ une famille de polyn\^omes homog\`enes \`a coefficients
dans~$\bar\Q$ d\'efinissant la sous-vari\'et\'e~$X$.
La sous-vari\'et\'e de~$\P^k$ d\'efinie par
les polyn\^omes~$P_j$ d'une part et~$f_0,\dots,f_m$
d'autre part est \'egale \`a $X\cap Z$, donc est vide.
D'apr\`es le th\'eor\`eme des z\'eros de Hilbert, dans sa version homog\`ene,
l'id\'eal $(P_j,f_i)$ engendr\'e par ces polyn\^omes
contient une puissance de l'id\'eal $(X_0,\dots,X_k)$.
Il existe ainsi un entier~$t$ et, pour tout $n\in\{0,\dots,k\}$
des polyn\^omes $G_{jn}$ et $H_{in}$ tels que
\[ X_n^t = \sum_j P_j G_{jn} + \sum_{i=0}^m f_i H_{in}. \]
Il est loisible de supposer les polyn\^omes $G_{jn}$ et $H_{in}$
homog\`enes, quitte \`a ne conserver que les mon\^omes de~$G_{jn}$
de degr\'e $t-\deg(P_j)$ et ceux de~$H_{in}$ de degr\'e~$t-d$.

Pour tout point $x=[x_0:\cdots:x_k]$ de~$X(\bar\Q)$,
on a alors
\[ x_n^t= \sum_j P_j(x) G_{jn}(x) + \sum_{i=0}^m f_i(x) H_{in}(x)
   = \sum_{i=0}^m f_i(x) H_{in}(x) \]
puisque, par hypoth\`ese, $P_j(x)=0$ pour tout~$j$.
Soit $N$ un nombre entier non nul tel que les coefficients
des polyn\^omes $NH_{in}$ soient entiers alg\'ebriques,
\'el\'ements d'un corps de nombres~$F$ contenant aussi les coefficients
des~$f_i$.

Par un argument similaire au \emph a), on en d\'eduit
que pour tout corps de nombres $K$ contenant~$F$,
tout nombre premier~$p$,
tout plongement~$\sigma$ de~$K$ dans~$\C_p$,
et tout point $[x_0:\cdots:x_k]\in X(F)$, on ait la majoration
\[ \abs{N \sigma(x_n)}_p^t \leq  \max(\abs{\sigma(f_0(x))}_p,\dots,\abs{\sigma(f_m(x))}_p) \max_{i,n}(\abs{\sigma(H_{in}(x))}_p). \]
Alors,
\[ \abs{N}_p \abs{\sigma(x_n)}_p^t \leq \max(\abs{\sigma(f_0(x))}_p,\dots,\abs{\sigma(f_m(x))}_p) \max(\abs{\sigma(x_0)},\dots,\abs{\sigma(x_n)})^{t-d}, \]
d'o\`u une majoration
\[ \max(\abs{\sigma(x_0)}_p,\dots,\abs{\sigma(x_n)}_p)^d
 \leq \abs{N}_p^{-1} \max(\abs{\sigma(f_0(x))}_p,\dots,\abs{\sigma(f_m(x))}_p).
\]

Pour les valeurs absolues archim\'ediennes, il existe
de m\^eme un nombre r\'eel~$C>0$ tel que 
\[ \max(\abs{\sigma(x_0)}_\infty,\dots,\abs{\sigma(x_n)}_\infty)^d
 \leq C \abs{N}_\infty^{-1} \max(\abs{\sigma(f_0(x))}_\infty,\dots,\abs{\sigma(f_m(x))}_\infty),\]
pour tout corps de nombres $K$ contenant~$F$,
tout point $x=[x_0:\cdots:x_k]\in X(K)$ et
tout plongement de $K$ dans~$\C$

Mettant bout \`a bout ces in\'egalit\'es, on obtient
\begin{multline*} d h([x_0:\cdots:x_k])
        \leq \log C + \sum_{p\leq\infty} \log\abs{N}_p^{-1} \\
{}  + \frac1{[K:\Q]} \sum_{p\leq\infty}\sum_{\sigma\colon K\hra\C_p}
 \log \max(\abs{\sigma(f_0(x))}_p,\dots,\abs{\sigma(f_m(x))}_p). \end{multline*}
Autrement dit,
$dh(x)\leq \log C+ h(f(x))$, ainsi qu'il fallait d\'emontrer.
\end{proof}

\subsection{Hauteurs, plongements, fibr\'es en droites}

Consid\'erons maintenant une vari\'et\'e projective~$X$, d\'efinie
sur~$\bar\Q$. \`A tout plongement $\phi$ de~$X$ dans un espace
projectif~$\P^k$ est associ\'ee une fonction hauteur sur~$X(\bar\Q)$,
d\'efinie par $h_\phi(x)=h_{\P^k}(\phi(x))$ pour $x\in X(\bar\Q)$,
o\`u $h_{\P^k}$ d\'esigne la hauteur sur~$\P^k(\bar\Q)$ construite pr\'ec\'edemment.
Cette d\'efinition ne requiert que le fait que $\phi$ soit une application
r\'eguli\`ere et s'\'etend donc \emph{verbatim} aux morphismes $\phi$ 
de~$X$ dans un espace projectif.
En fait, nous allons voir que $h_\phi$ ne d\'epend essentiellement que  
du fibr\'e en droites $\phi^*\mathscr O(1)$
sur~$X$ d\'eduit du fibr\'e tautologique $\mathscr O(1)$ sur l'espace
projectif.

\begin{lemm}
Soit $X$ une vari\'et\'e projective d\'efinie sur~$\bar\Q$.
Soit $\phi\colon X\ra\P^k$ et~$\psi\colon X\ra\P^m$ 
des morphismes de~$X$ dans des espaces projectifs.
Si $\phi^*\mathscr O(1)$ et $\psi^*\mathscr O(1)$ sont
isomorphes, la diff\'erence $h_\phi-h_\psi$ des hauteurs associ\'ees
\`a~$\phi$ et~$\psi$ est une fonction \emph{born\'ee} sur $X(\bar\Q)$.
\end{lemm}
\begin{proof}
Notons $\mathscr L$ le fibr\'e en droites $\phi^*\mathscr O(1)$;
par construction, il est engendr\'e par ses sections globales
et d\'efinit un morphisme $\alpha$ de~$X$ dans l'espace
projectif~$\P^s$ associ\'e \`a une base de l'espace vectoriel 
de ses sections globales
--- c'est le morphisme fourni par le syst\`eme lin\'eaire complet
associ\'e \`a~$\mathscr L$.
Il existe alors des applications rationnelles d\'efinies
par des polyn\^omes homog\`enes de degr\'e~$1$, $\phi'\colon\P^s\dashrightarrow \P^k$
et $\psi'\colon \P^s\dashrightarrow\P^m$
telles que $\phi=\phi'\circ\alpha$ et $\psi=\psi'\circ\alpha$.
En outre, le lieu d'ind\'etermination de~$\phi'$ et~$\psi'$
ne rencontre pas~$\alpha(X)$. (De mani\`ere \'equivalente,
les images r\'eciproques par~$\phi$
et~$\psi$ du syst\`eme lin\'eaire des hyperplans de~$\P^k$, resp.~$\P^m$,
sont, par d\'efinition de~$\mathscr L$,
des sous-syst\`emes lin\'eaires de celui associ\'e \`a~$\mathscr L$,
et sont sans point-base.)
D'apr\`es la prop.~\ref{prop.fonctorialite}, 
la fonction $h_{\P^s}-h_{\P^k}\circ\phi'$ est born\'ee sur $\alpha(X)(\bar\Q)$,
de m\^eme que la fonction $h_{\P^s}-h_{\P^m}\circ\psi'$.
Par cons\'equent, 
\[ h_\phi-h_\psi = h_{\P^k}\circ\phi'\circ\alpha
 - h_{\P^m}\circ\psi'\circ\alpha \]
est born\'ee sur $X(\bar\Q)$.
\end{proof}

Soit $\mathscr F$ l'espace vectoriel des fonctions
de~$X(\bar\Q)$ dans~$\R$ et soit $\mathscr F_b$
son sous-espace vectoriel constitu\'e des fonctions born\'ees.
Notons $\mathscr M(X)$ l'ensemble des morphismes de~$X$
dans un espace projectif et $\Pic(X)$ le groupe ab\'elien
des classes d'isomorphisme de fibr\'es en droites sur~$X$.

\`A $\phi\in\mathscr M(X)$, nous pouvons associer
d'une part la fonction $h_\phi$, ou plut\^ot sa classe $[h_\phi]$
modulo~$\mathscr F_b$, 
et d'autre part la classe d'isomorphisme $[\phi^*\mathscr O(1)]$ du
fibr\'e en droites $\phi^*\mathscr O(1)$ sur~$X$.
D'apr\`es le lemme ci-dessus, $[h_\phi]$ ne d\'epend
que de~$[\phi^*\mathscr O(1)]$.
Plus pr\'ecis\'ement, on a le th\'eor\`eme:
\begin{theo}
Il existe un unique homomorphisme de groupes~$\eta$
de~$\Pic(X)$ dans~$\mathscr F/\mathscr F_b$,
not\'e $\mathscr L\mapsto \eta_{\mathscr L}$,
qui, pour tout morphisme $\phi$ de~$X$ dans un espace projectif,
applique la classe d'isomorphisme du fibr\'e en droites $\phi^*\mathscr O(1)]$ 
sur la classe de la fonction $h_\phi$ modulo l'espace~$\mathscr F_b$
des fonctions born\'ees sur~$X(\bar\Q)$.
\end{theo}
\begin{proof}
Un tel homomorphisme est prescrit sur les classes
de fibr\'es en droites de la forme $\phi^*\mathscr O(1)$,
o\`u $\phi$ est un morphisme,  \emph{a fortiori}
sur les fibr\'es en droites tr\`es amples qui, par d\'efinition m\^eme,
sont de cette forme en prenant pour~$\phi$ un plongement.
Comme tout \'el\'ement de~$\Pic(X)$ est la diff\'erence de deux classes
de fibr\'es en droites tr\`es amples, il n'y a au plus qu'un
tel homomorphisme.

Notons $\Picplus(X)$ le sous-mono\"{\i}de de~$\Pic(X)$
form\'e des classes de diviseurs engendr\'es par leurs
sections globales. Ce sont exactement les images r\'eciproques
du faisceau~$\mathscr O(1)$ par un morphisme de~$X$
dans un espace projectif.
D'apr\`es le lemme ci-dessus, il existe donc une unique application
de~$\Picplus(X)$ dans~$\mathscr F/\mathscr F_b$ qui associe
\`a la classe de~$\phi^*\mathscr O(1)$ celle 
de la fonction
$h_\phi$, pour tout morphisme  $\phi$ de~$X$ dans un espace projectif.
Notons $D\mapsto \eta_D$ cette application.

Montrons qu'elle est additive.
Soit donc $\phi$ et $\psi$ des morphismes de~$X$ dans des espaces
projectifs $\P^k$ et $\P^m$ et consid\'erons la composition
\[ \alpha\colon X\xrightarrow{(\phi,\psi)} \P^k\times\P^m \xrightarrow {S} \P^{km+k+m}, \]
o\`u $S$ d\'esigne le plongement de Segre.
Le comportement du morphisme de Segre vis \`a vis des hauteurs
implique l'\'egalit\'e $h_\alpha=h_\phi+h_\psi$.
D'autre part, le fait qu'il soit d\'efini par des polyn\^omes
homog\`enes de bidegr\'e~$(1,1)$ implique
que $S^*\mathscr O(1)$ est isomorphe au produit tensoriel
externe des deux fibr\'es~$\mathscr O(1)$. Par suite,
$\alpha^*\mathscr O(1)$ est isomorphe au produit
tensoriel $\phi^*\mathscr O(1)\otimes \psi^*\mathscr O(1)$.
Si $D$ et~$E$ d\'esignent les classes de~$\phi^*\mathscr O(1)$
et $\psi^*\mathscr O(1)$, 
$\eta_{D+E}$ est donc la classe de~$h_\alpha=h_\phi+h_\psi$,
laquelle est la somme des classes $\eta_D$ et $\eta_E$.

L'existence d'un homomorphisme de groupes
$\eta\colon\Pic(V)\ra\mathscr F/\mathscr F_b$
v\'erifiant les propri\'et\'es
exig\'ees par le th\'eor\`eme r\'esulte maintenant
d'un argument \'el\'ementaire de diff\'erence.
\end{proof}

Soit $X$ une vari\'et\'e projective d\'efinie sur~$\bar\Q$ et 
soit $\mathscr L$ un fibr\'e en droites sur~$X$.
On appellera \emph{hauteur relative \`a~$\mathscr L$}
toute fonction de~$X(\bar\Q)$ dans~$\R$
dont la classe modulo les fonctions born\'ees
est \'egale \`a~$\eta_{\mathscr L}$.
Deux telles fonctions diff\`erent d'une fonction born\'ee.
Si $\mathscr L$ est ample (ou plus g\'en\'eralement
engendr\'e par ses sections globales), 
toute hauteur relative \`a~$\mathscr L$ est minor\'ee.

Une cons\'equence de la construction est la propri\'et\'e suivante:
\begin{prop}\label{prop.hof}
Soit $f\colon X\ra Y$ un morphisme de vari\'et\'es alg\'ebriques
d\'efinies sur~$\bar\Q$, soit $\mathscr M$ un fibr\'e
en droites sur~$Y$. Pour toute hauteur $h_{\mathscr M}$
relative \`a~$\mathscr M$ sur~$Y$, $h_{\mathscr M}\circ f$
est une  hauteur relative \`a~$f^*\mathscr M$ sur~$X$.
\end{prop}
\begin{proof}
Par un argument de diff\'erence,
il suffit de traiter le cas d'un fibr\'e tr\`es ample.
On peut donc supposer qu'il existe un plongement~$\phi$ 
de~$Y$ dans un espace projectif tel que $\mathscr M\simeq \phi^*\mathscr O(1)$.
Alors, $\psi=\phi\circ f$ est un morphisme
de~$X$ dans un espace projectif et l'on a 
$\psi^*\mathscr O(1)\simeq f^*\mathscr M$.
Par d\'efinition, $h_\psi=h_\phi\circ f$ est donc une hauteur
relative \`a $f^*\mathscr M$. Comme $h_{\mathscr M}$
est une hauteur relative \`a~$\mathscr M$, $h_{\mathscr M}-h_\phi$
est une fonction born\'ee, ce qui implique
que $h_{\mathscr M}\circ f-h_\psi$ est born\'ee sur $X(\bar\Q)$.
\end{proof}

\section{Finitude}

Le th\'eor\`eme principal de ce paragraphe
g\'en\'eralise la proposition~\ref{prop.finitude-Q}.

\begin{theo}\label{theo.finitude}
Pour tout entier~$d\geq 1$ et tout nombre
r\'eel~$B$, l'ensemble des points $x\in\P^k(\bar\Q)$
d\'efinis sur un corps de nombres de degr\'e au plus~$d$ 
et dont la hauteur est au plus~$B$ est fini.
\end{theo}

Quelques commentaires sur l'expression {\og d\'efinis
sur un corps de nombres de degr\'e au plus~$d$\fg}
et sur le corps de d\'efinition d'un point de l'espace
projectif.
Soit $x\in\P^k(\bar\Q)$ et soit $[x_0:\cdots:x_k]$
un syst\`eme de coordonn\'ees homog\`enes de~$x$, dans~$\bar\Q$.
Dire que $x$ est d\'efini sur un corps
de nombres~$K$ signifie que $x\in\P^k(K)$, autrement dit que
$x$ admet un syst\`eme $[\xi_0:\cdots:\xi_k]$ de coordonn\'ees homog\`enes dans~$K$.
De la proportionalit\'e des deux syst\`emes r\'esulte que
$x$ est d\'efini sur~$K$ si et seulement si $x_i/x_j\in K$
pour tout couple $(i,j)$ d'\'el\'ements de~$\{0,\dots,k\}$
tels que $x_j\neq 0$.
Dit autrement, si, disons $x_0\neq 0$, 
le corps $\Q(x_1/x_0,\dots,x_k/x_0)$ est le plus petit
corps de nombres sur lequel $x$ soit d\'efini. On l'appelle
le \emph{corps de d\'efinition} de~$x$ et on le note~$\Q(x)$.

La th\'eorie de Galois fournit une autre fa\c{c}on de voir ce corps.
Le groupe $\Gal(\bar\Q/\Q)$ des automorphismes de~$\bar\Q$
agit sur $\P^k(\bar\Q)$ via, rappelons-le, sur son action sur les coordonn\'ees
homog\`enes : 
si $\sigma\in\Gal(\bar\Q/\Q)$ et $x=[x_0:\cdots:x_k]\in\P^n(\bar\Q)$,
$\sigma(x)=[\sigma(x_0):\cdots:\sigma(x_k)]$.
Le stabilisateur du point~$x$ est exactement
le sous-groupe $\Gal(\bar\Q/\Q(x))$ de~$\Gal(\bar\Q/\Q)$.
Notons $d$ le degr\'e du corps~$\Q(x)$
et soit $\sigma_1,\dots,\sigma_d$ une famille de repr\'esentants
de $\Gal(\bar\Q/\Q)$ modulo $\Gal(\bar\Q/\Q(x))$.
(Leurs restrictions au corps~$\Q(x)$  sont exactement
les $d$ homomorphismes de corps de~$\Q(x)$ dans~$\bar\Q$.)
L'orbite de~$x$ sous $\Gal(\bar\Q/\Q)$ est alors form\'ee
des $d$ points $\sigma_1(x),\dots,\sigma_d(x)$,
que l'on appelle les conjugu\'es de~$x$.

\begin{proof}
Pour tout couple $(i,j)$ d'\'el\'ements de~$\{0,\dots,k\}$,
notons $p_{ij}$ la projection de~$\P^k$
sur~$\P^1$ donn\'ee par $[x_0:\cdots:x_k]\mapsto [x_i:x_j]$;
elle est d\'efinie hors du sous-espace projectif $P_{ij}$ de codimension~$2$
d\'efini par l'annulation des coordonn\'ees homog\`enes~$x_i$ et~$x_j$.

Si $x\in (\P^n\setminus P_{ij})(\bar\Q)$, on a d\'emontr\'e
que $h(p_{ij}(x))\leq h(x)$. En outre, $p_{ij}(x)$ est
d\'efini sur $\Q(x)$, ainsi qu'il r\'esulte de la description
du corps de d\'efinition d'un point de l'espace projectif
rappel\'ee ci-dessus.

Supposons \'etabli le th\'eor\`eme lorsque $n=1$; alors,
l'ensemble des projections $p_{ij}(x)$, pour $x\in\P^k(\bar\Q)$
d\'efini sur un corps de degr\'e au plus~$d$ et de hauteur au plus~$B$,
est fini. Autrement dit, les quotients $x_i/x_j$
ne peuvent prendre qu'un nombre fini  de valeurs possibles,
d'o\`u la finitude annonc\'ee.

Il reste \`a montrer le cas $n=1$.
Introduisons un morphisme~$\theta$ de~$(\P^1)^d$ dans~$\P^d$
donn\'ee par 
\[  \theta ([x^{(1)}_0:x^{(1)}_1], \dots,[x^{(d)}_0:x^{(d)}_1])
 = [z_0:\cdots:z_d], \]
o\`u les~$z_j$ sont d\'efinis par la relation
\[ \prod_{i=1}^d (x^{(i)}_0 + T x^{(i)}_1) = \sum_{j=0}^d z_j T^j, \]
o\`u $T$ est une ind\'etermin\'ee.
En d'autres termes, on a, pour $j\in\{0,\dots,d\}$,
\[ z_j=\sum_{\substack{\eps\colon \{1,\dots,d\}\ra\{0,1\}\\ \eps_1+\dots+\eps_d=j}} x^{(1)}_{\eps_1} \dots x^{(d)}_{\eps_d}, \]
versions homog\`enes des fonctions sym\'etriques \'el\'ementaires
puisque si $x^{(i)}_0\neq 0$ pour tout~$i$, alors $z_0\neq 0$ et
\[ \frac{z_j}{z_0} = \sum_{i_1<\dots<i_j}
         \xi^{(i_1)} \dots \xi^{(i_j)}, \]
o\`u l'on a pos\'e $\xi^{(i)}=x^{(i)}_1/x^{(i)}_0$.

Le fibr\'e en droites $\theta^*\mathscr O(1)$
est isomorphe au produit tensoriel externe des
fibr\'es $\mathscr O(1)$ sur chacune des copies de~$\P^1$,
car $\theta$ est d\'efinie par des formes
de multidegr\'e~$(1,\dots,1)$.
Il en r\'esulte qu'\`a une fonction born\'ee pr\`es,
\[ \sum_{i=1}^d h(x^{(i)}) \approx h(\theta(x^{(1)},\dots,x^{(d)}))
 \]
pour tout $(x^{(1)},\dots,x^{(d)})\in\P^1(\bar\Q)^d$.

Soit maintenant $x\in \P^1(\bar\Q)$ dont le corps
de d\'efinition $\Q(x)$ est de degr\'e~$d$.
Notons $x^{(1)},\dots,x^{(d)}$ les \'el\'ements de son
orbite sous $\Gal(\bar\Q/\Q)$ et
appliquons la relation
au $d$-uplet $[x]=(x^{(1)},\dots,x^{(d)})$.
Comme la hauteur est invariante sous $\Gal(\bar\Q/\Q)$
(prop.~\ref{prop.h-galois})
il en r\'esulte l'exitence d'un nombre r\'eel~$c$
tel que
\[ dh(x) -c \leq  h(\theta([x]))\leq dh(x)+c , \]
pour $x\in\P^1(\bar\Q)$.

Le point important est que $\theta([x])$ est, par construction
m\^eme, invariant sous $\Gal(\bar\Q/\Q)$.
C'est donc un point de~$\P^d(\Q)$ de hauteur
au plus $dh(x)\leq dB+c$.
(De fait, lorsque $x_0\neq 0$,
$\theta([x])$ n'est autre que la collection
des coefficients du polyn\^ome minimal de~$x_1/x_0$.)

Comme l'ensemble des points de~$\P^d(\Q)$ de hauteur au plus~$dB+c$
est fini (proposition~\ref{prop.finitude-Q}),
l'ensemble des~$\theta([x])$,
pour $x\in\P^1(\bar\Q)$ de hauteur au plus~$B$
et d\'efinis sur une extension de degr\'e~$d$ de~$\Q$
est fini.
L'application $\theta$ n'est pas injective mais ses
fibres ont cardinal au plus~$d!$. En effet, la connaissance
de $(z_0,\dots,z_d)$ d\'etermine celle de $(\xi^{(1)},
\dots,\xi^{(d)})$ \`a l'ordre pr\`es :
ce sont les oppos\'es des inverses des racines du polyn\^ome
$\sum z_j T^j$.
Plus pr\'ecis\'ement, la connaissance de~$\theta([x])$,
c'est-\`a-dire, si $x_0\neq0$, du polyn\^ome minimal de~$x_1/x_0$, 
d\'etermine $x\in\P^1(\bar\Q)$ \`a un choix parmi~$d$ pr\`es:
celui d'une racine de ce polyn\^ome de degr\'e~$d$.

Par cons\'equent, l'ensemble des points~$x\in\P^1(\bar\Q)$
de  hauteur au plus~$B$
et d\'efinis sur une extension de degr\'e~$d$ de~$\Q$
est fini, ce qu'il fallait d\'emontrer.
\end{proof}

\begin{coro}\label{coro.finitude}
Soit $X$ une vari\'et\'e projective sur~$\bar\Q$
et soit $\mathscr L$ un fibr\'e en droites ample sur~$X$.
Notons $h_{\mathscr L}$ une hauteur relative \`a~$\mathscr L$
sur~$X$. Pour tout entier~$d$ et tout nombre r\'eel~$B$,
l'ensemble des points $x\in X(\bar\Q)$
d\'efinis sur une extension de degr\'e au plus~$d$
et de hauteur au plus~$B$ est fini.
\end{coro}
\begin{proof}
Par d\'efinition, il existe un entier~$m\geq 1$
tel que $\mathscr L^{\otimes m}$ soit tr\`es ample,
c'est-\`a-dire isomorphe au fibr\'e en droites $\phi^*\mathscr O(1)$,
o\`u $\phi$ est un plongement de~$X$ dans un espace projectif~$\P^k$.
Par construction de la machine des hauteurs,
$m h_{\mathscr L}-h_\phi$ est une fonction born\'ee
sur~$X(\bar\Q)$. Comme $\phi$ est injectif,
on voit que l'assertion r\'esulte 
de l'\'enonc\'e de finitude sur $\P^k$.
\end{proof}

\begin{rema}
Le corps $k(T)$ des fractions rationnelles en une variable \`a coefficients
dans un corps~$k$, et ses extensions finies, les corps
des fonctions de courbes alg\'ebriques projectives r\'eguli\`eres sur
une extension finie de~$k$,
jouissent de propri\'et\'es alg\'ebriques tr\`es semblables
\`a celles du corps~$\Q$, ou des corps de nombres.
On peut en particulier y d\'efinir une notion de hauteur;
par exemple, si $U_0,\ldots,U_n$ sont des polyn\^omes de~$k[T]$,
deux \`a deux sans facteur commun, la hauteur du point
de l'espace projectif~$\P^n(k(T))$
de coordonn\'ees homog\`enes $[U_0:\cdots:U_n]$ 
est d\'efinie comme le maximum des degr\'es des~$U_i$.
Si le corps~$k$ est infini, on observera que le th\'eor\`eme
de finitude ci-dessus n'est plus vrai en g\'en\'eral, pas plus que les
cons\'equences que nous en avons tir\'ees, 
par exemple la prop.~\ref{prop.hauteur-normalisee-Q}.
(On peut par exemple prendre~$k$ alg\'ebriquement clos
et consid\'erer un endomorphisme de~$\P^n$ \`a coefficients dans~$k$
comme un endomorphisme, {\og constant\fg}, 
de l'espace projectif sur le corps~$k(T)$.)
Dans la suite de ce texte, je me contenterai d'indiquer
quelques r\'ef\'erences bibliographiques \`a des extensions ou contre-exemples
des th\'eor\`emes valables sur les corps de nombres.
\end{rema}

\section{Hauteurs locales et fonctions de Green}
\label{sec.locales}

La formule~\eqref{def.hauteur-p}
qui d\'efinit  la hauteur d'un point de l'espace
projectif \`a coefficients dans un corps de nombres~$K$ 
est une somme index\'ee sur les
diff\'erents plongements de~$K$ dans les corps~$\C_p$,
o\`u $p$ parcourt l'ensemble des nombres premier et~$\infty$.
Cependant, chacun de ces termes n'est pas une fonction
sur l'espace projectif car ils d\'ependent
du choix des coordonn\'ees homog\`enes ; seule leur somme
n'en d\'epend plus, en vertu de la formule du produit.

La th\'eorie des fonctions de Green permet d'exprimer la hauteur
d'un point comme somme de termes locaux ({\og \emph{hauteurs locales}\fg})
bien d\'efinis.

\subsection{D\'efinition}

Soit $K$ un corps valu\'e complet et alg\'ebriquement clos,
soit $X$ une vari\'et\'e projective d\'efinie sur~$K$
et soit $D$ un diviseur de Cartier effectif sur~$X$.
(Rappelons qu'il s'agit d'un sous-sch\'ema ferm\'e qui 
est localement d\'efini par une \'equation non-diviseur de z\'ero.)

On appelle \emph{fonction de Green}
relativement \`a~$D$ sur~$X$ toute fonction continue
$\lambda_D \colon (X\setminus D)(K)\ra \R$ qui v\'erifie
les propri\'et\'es suivantes:
\begin{itemize}
\item
si $D$ est tr\`es ample, 
il existe un plongement de~$X$ dans un espace projectif~$\P^n$
tel que $D=X\cap H_0$, o\`u $H_0=\{x_0=0\}$, et tel que la fonction
donn\'ee par
\[ x\mapsto \lambda_D(x)  + \log \frac{\abs{x_0}}{\max(\abs{x_0},\dots,\abs{x_n})}
\]
s'\'etende (de mani\`ere unique car $X\setminus D(K)$ est dense
dans~$X(K)$) 
en une fonction continue et born\'ee sur~$X(K)$.
\item
si $D=E-F$ est la diff\'erence de deux diviseurs tr\`es amples, 
il existe des fonctions de Green~$\lambda_E$ et~$\lambda_F$ 
relativement \`a~$E$ et~$F$
comme ci-dessus telles que  $\lambda_D=\lambda_E-\lambda_F$
sur $(X\setminus (E\cup F))(K)$.
\end{itemize}

\begin{lemm}
L'ensemble des fonctions de Green relativement \`a un diviseur de
Cartier~$D$ est un espace affine sous l'espace vectoriel
des fonctions continues et born\'ees sur~$X(K)$.
\end{lemm}
\begin{proof}
Soit $D$ et $E$ des diviseurs tr\`es amples
et soit $\lambda_D$, $\lambda_E$ des fonctions de Green 
d\'efinies \`a l'aide de plongements $\phi$ et~$\psi$.
On v\'erifie que la composition de~$(\phi,\psi)$
et d'un plongement de Segre~$S$ d\'efinit la fonction~$\lambda_D+\lambda_E$,
\`a une fonction continue born\'ee pr\`es.

Par un argument \'el\'ementaire,
il suffit alors de d\'emontrer que la diff\'erence de deux hauteurs
locales associ\'ees \`a un diviseur tr\`es ample s'\'etend en une fonction
continue born\'ee sur $X(K)$.
Soit $\phi\colon X\ra\P^n$ et $\psi\colon X\ra\P^m$
des plongements tels que $D=\phi^*H_0=\psi^*H_0$,
o\`u $H_0$ est l'hyperplan d'\'equation $x_0=0$ dans~$\P^n$,
resp. d'\'equation $y_0=0$ dans~$\P^m$.
En particulier, ces deux plongements sont associ\'es \`a un plongement~$\alpha
\colon X\ra\P^s$
d\'efini par le syst\`eme lin\'eaire complet associ\'e \`a~$D$
compos\'es avec des projections lin\'eaires $\phi'$
et~$\psi'$
dont les centres ne rencontrent pas~$\alpha(X)$.
Par hypoth\`ese, 
les formes lin\'eaires $(\phi'_0,\dots,\phi'_n)$ d\'efinissant~$\phi'$,
resp. $[\psi'_0:\cdots:\psi'_m]$ d\'efinissant~$\psi'$
ne s'annulent pas simultan\'ement sur~$\alpha(X)$.
Comme les formes $\phi'_0$ et $\psi'_0$  sur~$\P^s$ d\'efinissent
toutes deux $\alpha(D)$, leur quotient est une
fonction r\'eguli\`ere sur~$\alpha(X)$, donc constante
sur chacune des composantes connexes de~$\alpha(X)$. Par cons\'equent,
l'application de $\alpha(X)$ dans~$\R$ d\'efinie par
\[ \mu\colon y=[y_0:\cdots:y_s]\mapsto \log \frac{\max(\abs{\phi'_0(y)},\dots,\abs{\phi'_n(y)})}{\max(\abs{\psi'_0(y)},\dots,\abs{\psi'_m(y)})} \]
est bien d\'efinie --- num\'erateur et d\'enominateur
sont homog\`enes --- et continue sur~$\alpha(X)$.
Si $K$ est localement compact, en l'occurrence si $K=\C$,
elle est alors born\'ee.
Dans le cas g\'en\'eral, il faut une fois de plus utiliser le th\'eor\`eme des z\'eros
de Hilbert.

L'existence d'une \emph{majoration}
\[ \max(\abs{\phi'_0(y)},\dots,\abs{\phi'_n(y)}) \leq c \max(\abs{y_0},\dots,\abs{y_s}) \]
est \'evidente.
Soit $(P_j)$ une famille de polyn\^omes homog\`enes d\'efinissant $\alpha(X)$.
La famille $(P_j,\psi'_i)$ n'ayant pas de z\'ero commun, il existe
des polyn\^omes $Q_{jk}$ et $H_{ik}$, et un entier~$t$, tels que
\[ Y_k^t = \sum_k  Q_{jk} P_j + \sum_{i=0}^m \psi'_i H_{ik}. \]
On peut de plus supposer ces polyn\^omes homog\`enes.
Puisque $P_j(y)=0$ pour tout~$j$,
il existe un nombre r\'eel~$c$ tel que
\[ \abs{y_k}^t \leq c \max(\abs{\psi'_0(y)},\dots,\abs{\psi'_m(y)})
 \max_{i,k} \abs{H_{ik}(y)}. \]
De plus, comme les polyn\^omes $H_{ik}$ sont de degr\'e~$t-1$,
on a une majoration $\abs{H_{ik}(y)}\ll \max(\abs{y_0},\dots,\abs{y_s})^{t-1}$.
Finalement, on en d\'eduit une \emph{minoration}
\[ \max(\abs{\psi'_0(y)},\dots,\abs{\psi'_m(y)}) \geq c' \max(\abs{y_0},\dots,\abs{y_s}). \]
Le fait que $\mu$ soit born\'ee est alors \'evident.
\end{proof}

\begin{prop}
Soit $X$ une vari\'et\'e projective sur~$K$, soit $D$ et~$E$ des diviseurs
de Cartier sur~$X$ et soit $\lambda_D$, $\lambda_{E}$
des fonctions de Green pour~$D$ et~$E$ respectivement.
Alors, les fonctions $\lambda_D+\lambda_E$ et $\lambda_D-\lambda_E$
sont les restrictions \`a $X\setminus(D\cup E)(K)$ de
fonctions de Green pour~$D+E$ et~$D-E$ respectivement.
\end{prop}
\begin{proof} 
Compte tenu de la d\'efinition, il suffit de traiter le cas de~$D+E$
sous l'hypoth\`ese que $D$ et~$E$ sont tr\`es amples, associ\'es
\`a des plongements $\phi\colon X\hra\P^k$ et $\psi\colon X\hra\P^m$
tels que $D$ est d\'efini par $x_0=0$ dans~$\phi(X)$
et $E$ est d\'efini par $y_0=0$ dans~$\psi(X)$.
Consid\'erons la composition~$\alpha$ de~$(\phi,\psi)$ et du plongement de Segre 
$S\colon \P^k\times\P^m\hra \P^{km+k+m}$,
donn\'e par $S([x_0:\cdots:x_k],[y_0:\cdots:y_m])=[x_0y_0:\cdots:x_ky_m]$.
Dans $\alpha(X)$, l'\'equation $z_0=0$ d\'efinit le diviseur de Cartier
$D+E$. L'assertion r\'esulte alors de l'\'egalit\'e
\[ \log \frac{\abs{x_0y_0}}{\max(\abs{x_iy_j})}
= \log \frac{\abs{x_0}}{\max(\abs{x_i})} + \log \frac{\abs{y_0}}{\max(\abs{y_j})} ,\]
valable pour tout couple de points $(x,y)\in\P^k(K)\times\P^m(K)$
tels que $x_0\neq 0$ et $y_0\neq 0$.
\end{proof}

\begin{prop}
Soit $f\colon X\ra Y$ un morphisme entre vari\'et\'es
projectives sur~$K$, soit $D$ un diviseur de Cartier sur~$X$,
soit $E$ un diviseur de Cartier sur~$Y$. On suppose que 
$D=f^*E$. Si $\lambda_E$ est une fonction de Green relativement \`a~$E$ sur~$Y$,
alors $\lambda_E\circ f$ est une fonction de Green relativement \`a~$D$
sur~$X$.
\end{prop}
\begin{proof}
Il suffit de traiter le cas o\`u $E$ est tr\`es ample, associ\'e
\`a un plongement~$\psi$ de~$Y$ dans~$\P^m$
et tel que $E$ soit d\'efini par l'\'equation $y_0=0$ dans~$\psi(Y)$.
Soit $D'$ un diviseur de Cartier tr\`es ample sur~$X$,
associ\'e \`a un plongement $\phi$ de~$X$ dans~$\P^k$
de sorte que $D'$ soit d\'efini par l'\'equation $x_0=0$ dans~$\phi(X)$.
Introduisons alors la composition 
$\alpha$ de $(\phi,\psi\circ f)\colon X\ra \P^k\times\P^m$
et du plongement de Segre vers $\P^{km+k+m}$.
C'est un plongement et l'\'equation $z_0=0$ d\'efinit $D'+f^*E=D'+D$
dans~$\alpha(X)$. Pour $x\in X$ de coordonn\'ees
homog\`enes~$[x_0:\cdots:x_k]$ et d'images
$y=f(x)=[y_0:\cdots:y_m]$ dans~$Y$ et $z=\alpha(x)=[z_0:\cdots:z_{km+k+m}]$
dans~$\P^{km+k+m}$, on a alors 
\[ \log\frac{\abs{z_0}}{\max(\abs{z_i})}
= \log \frac{\abs{x_0}}{\max(\abs{x_i})} + \log \frac{\abs{y_0}}{\max(\abs{y_j})}.
\]
Si $\lambda_{D'}$ est une fonction de Green pour~$D'$,
il en r\'esulte que $\lambda_{D'}+\lambda_E\circ f$ est une fonction
de Green pour $D'+f^*E$; par cons\'equent, $\lambda_E\circ f$
est une fonction de Green pour~$E$.\footnote{%
Voir l'exercice~\ref{exer.implicite}
pour un point de d\'etail n\'ecessaire \`a une preuve compl\`ete.}
\end{proof}

\subsection{Cas du corps~$\C$}

Le fibr\'e en droites $\mathscr O(D)$ a pour sections 
les fonctions m\'eromorphes~$f$ ayant au plus~$D$ comme p\^ole,
c'est-\`a-dire telles que $\div(f)+D$ soit effectif.
Si $D$ est effectif,
la fonction r\'eguli\`ere constante~$1$ d\'efinit en particulier
une section globale de ce fibr\'e en droites que l'on
note~$\mathbf 1_D$ et dont le diviseur (en tant
que section de~$\mathscr O(D)$) n'est autre que~$D$.

Supposons que le corps valu\'e soit le corps des nombres complexes.
Soit $X$ une vari\'et\'e projective complexe, soit $D$ un diviseur
de Cartier de~$X$ et soit $\lambda_D\colon(X\setminus D)(\C)\ra\R$
une fonction continue. Pour que $\lambda_D$ soit une hauteur
locale relativement \`a~$D$, il faut et il suffit
qu'il existe un recouvrement ouvert fini $(U_1,\dots,U_k)$
tel que $D$ soit d\'efini par une \'equation $f_i$ sur~$U_i$
et que la fonction $\lambda_D+\log\abs{f_i}$
s'\'etende en une fonction continue
sur~$U_i$.
Il revient ainsi au m\^eme d'exiger qu'il existe une \emph{m\'etrique hermitienne
continue} sur le fibr\'e en droites $\mathscr O(D)$
telle que $\log\norm{\mathbf 1_D}=\lambda_{D}$.

Notons par exemple 
que la m\'etrique hermitienne de r\'ef\'erence que nous avons choisie
sur le fibr\'e $\mathscr O(1)$ de l'espace projectif~$\P^k$
est donn\'ee par la formule
\[ \norm{a_0X_0+\dots+a_kX_k} ([x_0:\cdots:x_k])
   = \frac{\abs{a_0 x_0+\dots+a_k x_k}}{\max(\abs{x_0},\dots,\abs{x_k})}. \]
Ce n'est pas tout \`a fait la m\'etrique de Fubini-Study,
laquelle est donn\'ee par 
\[ \norm{a_0X_0+\dots+a_kX_k} ([x_0:\cdots:x_k])
   = \frac{\abs{a_0 x_0+\dots+a_k x_k}}{(\abs{x_0}^2+\dots+\abs{x_k}^2)^{1/2}}. \]
On voit que l'on peut ainsi d\'efinir la notion
de fonction de Green $\mathscr C^\infty$, parall\`element \`a
celle de m\'etrique hermitienne $\mathscr C^\infty$: par d\'efinition,
une fonction de Green est dite~$\mathscr C^\infty$
si c'est le logarithme de la norme de la section~$\mathbf 1_D$
pour une m\'etrique hermitienne $\mathscr C^\infty$.

La formule classique $\ddc \log \abs{z}^{-2}+\delta_0=0$
en une variable et, plus g\'en\'eralement, la formule
de Poincar\'e-Lelong $\ddc \log\abs{f}^{-2}+\delta_{\div(f)}=0$
entra\^{\i}nent que $\ddc  g_D+\delta_D$ est une forme diff\'erentielle
de type~$(1,1)$, lisse,
pourvu que $g_D$ soit une fonction de Green $\mathscr C^\infty$.
On a not\'e $\delta_D$ le courant d'int\'egration sur~$D$,
d\'efini, au choix, par int\'egration des formes sur la partie
lisse de~$D$ ou par r\'esolution des singularit\'es. C'est
un courant positif ferm\'e sur~$X(\C)$, de bidegr\'e~$(1,1)$.

On retrouve alors la d\'efinition standard en g\'eom\'etrie
d'Arakelov telle que pos\'ee par~\cite{gillet-s90}.

\begin{defi}
Soit $X$ une vari\'et\'e projective complexe (lisse)
et soit $Z$ une sous-vari\'et\'e int\`egre de~$X$ de codimension~$p$.
On appelle courant de Green pour~$Z$ tout courant $g_Z$
sur~$X(\C)$ tel que 
$ \ddc  g_Z + \delta_Z $ soit une forme lisse de type~$(p,p)$.
\end{defi}

Revenons aux fonctions de Green associ\'ees \`a un diviseur~$D$.
Soit $g_D$ une telle fonction de Green, supposons-la lisse,
ou au moins de classe~$\mathscr C^2$,
de sorte qu'est d\'efinie la forme diff\'erentielle
$\omega_D=\ddc  g_D+\delta_D$.

Alors, pour tout entier~$k\geq 1$,
$\bigwedge^k\omega_D$ est une forme de type~$(k,k)$
sur~$X(\C)$. Prenons en particulier $k=\dim X$;
on associe classiquement \`a cette forme une
mesure sur~$X(\C)$: en coordonn\'ees locales holomorphes $z_1=x_1+iy_1,\ldots,z_k=x_k+iy_k$, 
si $\bigwedge^k\omega_D=f(z) \mathrm dz_1\wedge \cdots \wedge \mathrm dz_k
\wedge d\bar z_1\wedge\cdots \wedge\mathrm d\bar z_k$, la mesure
correspondante est
$\abs{f(z)} 2^k \mathrm dx_1 \mathrm dy_1\cdots\mathrm dx_k\mathrm dy_k$.
(Observer que $dz_j\wedge d\bar z_j=-2i dx_j\wedge dy_j$.)
Que cela soit bien d\'efini r\'esulte de la formule du changement de variables
dans les int\'egrales multiples.

On v\'erifie par une nouvelle application de la formule de Poincar\'e-Lelong
que cette forme~$\omega_D$ et cette mesure~$\bigwedge^k\omega_D$
ne d\'ependent du couple~$(D,g_D)$ que par l'interm\'ediaire
du fibr\'e hermitien qu'il d\'efinit. Elles co\"{\i}ncident d'ailleurs
avec la forme de Chern de ce fibr\'e hermitien 
et sa puissance ext\'erieure maximale.

\medskip

Dans le cas d'un corps $p$-adique, la situation est plus d\'elicate
pour un certain nombre de raisons:
\begin{itemize}
\item $\C_p$ n'est pas localement compact;
\item l'op\'erateur aux d\'eriv\'ees partielles $\ddc $ n'existe pas, pas
plus d'ailleurs que les courants ;
\item les mesures naturelles n'existent pas.
\end{itemize}
La th\'eorie de~\cite{zhang95b} fournit n\'eanmoins une bonne notion
de m\'etrique $p$-adique.
On r\'esout (simultan\'ement) ces trois probl\`emes \`a l'aide
de la th\'eorie des espaces analytiques introduits par~\cite{berkovich1990},
cf. \cite{gubler1998,gubler2003} et~\cite{chambert-loir2006}.

\subsection{D\'ecomposition de la hauteur en termes locaux}

Soit $X$ une vari\'et\'e projective sur un corps
de nombres~$K$ et soit $D$ un diviseur de Cartier sur~$X$.
Il s'agit d'\'ecrire la hauteur d'un point de~$X(\bar\Q)$
qui n'appartient pas \`a~$D$, relativement au fibr\'e $\mathscr O(D)$,
comme une somme de termes locaux index\'ee par l'ensemble~$M_K$
des places de~$K$.

Si $v$ est une place de~$K$, nous noterons $\C_v$
le corps valu\'e complet alg\'ebriquement clos $\C$ ou~$\C_p$
correspondant et $\eps_v$ le nombre de plongements de~$K$
dans~$\C_v$ qui induisent cette valeur absolue.
L'adh\'erence~$K_v$ de~$K$ dans~$\C_v$ est un corps valu\'e complet
(\'egal \`a~$\R$, $\Q_p$ pour $p$ un nombre premier, ou une
extension finie d'un de ces corps).
Nous noterons $\bar{K_v}$ la cl\^oture alg\'ebrique
de~$K_v$ dans~$\C_v$.

Soit $v$ une place de~$K$.
Soit $g_v$ une fonction de Green sur $X(\C_v)$ relativement \`a~$D$
(nous dirons aussi que $g_v$  est une fonction de Green $v$-adique).
On d\'efinit une \emph{hauteur locale} relativement \`a~$D$ 
de la fa\c{c}on suivante.
Soit $P\in (X\setminus D)(\bar {K_v})$; soit  $n$ son degr\'e sur~$K$
et soit $P_1,\dots,P_n$ ses conjugu\'es dans~$X(\C_v)$.
On pose
\[ h_v (P) = \frac1n \sum_{i=1}^n g_v(P_i). \]

Si $D$ est effectif et tr\`es ample,
nous dirons qu'une famille $(g_v)$ de fonctions de Green $v$-adiques
relativement \`a~$D$, index\'ee par l'ensemble~$M_K$ des places
de~$K$ est \emph{\'el\'ementairement admissible}
si pour presque toute place~$v$, $g_v$ est d\'efinie
par un m\^eme plongement~$f=[f_0:\cdots:f_s]$ (d\'efini sur~$K$) 
de~$X$ dans un espace projectif~$\P^s$
dont $D$ est la section hyperplane $\{x_0=0\}$: pour toute place~$v$,
sauf pour un nombre fini d'entre elles, on a ainsi
\[ g_v (x) = -\log \frac{ \abs{f_0(x)}_v }{ \max(\abs{f_0(x)}_v,\ldots,\abs{f_s(x)}_v)}. \]
Dans le cas g\'en\'eral, $D$ est la diff\'erence~$E-F$ 
de deux diviseurs effectifs tr\`es amples et nous dirons qu'une famille
$(g_v)$ est \emph{admissible} si l'on a $g_v=g_{E,v} -g_{F,v}$
pour tout~$v$, o\`u $(g_{E,v})_v$ et $(g_{F,v})_v$
sont des familles \'el\'ementairement admissibles de fonctions  de Green pour~$E$
et~$F$ respectivement.

Le lemme suivant,
dont la d\'emonstration est laiss\'ee au lecteur,
montre que l'ensemble des familles admissibles de fonctions de Green
est stable par les op\'erations standard de la g\'eom\'etrie alg\'ebrique.
\begin{lemm}\label{fonctorialite.admissible}
Soit $X$ une vari\'et\'e alg\'ebrique projective sur  un corps de nombres~$K$.

\begin{enumerate}
\item Soit $(g_v)$ et $(g'_v)$ des familles de fonctions
de Green (\'el\'ementairement) admissibles relativement \`a des diviseurs~$D$ et~$D'$.
La famille $(g_v+g'_v)$ est  une fonction de Green (\'el\'ementairement)
admissible
relativement \`a~$D+D'$.

\item Soit $X'$ une vari\'et\'e alg\'ebrique projective, int\`egre,
d\'efinie sur~$K$, soit $f\colon X'\ra X$ un morphisme et soit $D$ un diviseur
sur~$X$ tel que $f(X')\not\subset D$.
Si $(g_v)$ est une famille de fonctions de Green
(\'el\'ementairement) admissible relativement \`a~$D$, $(g_v\circ f)$ est
une famille de fonctions de Green (\'el\'ementairement) admissible relativement
au diviseur $f^*D$.
\end{enumerate}
\end{lemm}

\begin{prop}\label{prop.decomposition}
Soit $D$ un diviseur de Cartier sur~$X$, soit
$h_D$ une hauteur pour~$D$ et soit
$(g_v)$ une famille admissible de fonctions de Green pour
le diviseur~$D$. Alors, la s\'erie
$ h(x)=\sum_{v\in M_K} \eps_v h_v(x)$  
est une somme finie pour tout $x\in  (X\setminus D)(\bar\Q)$;
de plus,  $h-h_D$ est born\'ee sur $(X\setminus D)(\bar\Q)$.
\end{prop}
\begin{proof}
Par lin\'earit\'e, on se ram\`ene au cas o\`u le diviseur~$D$
est tr\`es ample et o\`u $(g_v)$ est une famille
\'el\'ementairement admissible de fonctions de Green.
Soit $\phi\colon X\hra\P^k$ un plongement de~$X$ dans un espace
projectif dont~$D$ est la section hyperplane~$\{X_0=0\}$
et d\'efinissant presque toutes les fonctions de Green.
Pour toute place~$v$ et tout point~$x\in (X\setminus D)(\C_v)$,
le point $\phi(x)$ a des coordonn\'ees homog\`enes $[x_0:\dots:x_k]$
avec $x_0\neq 0$; posons alors
\[ g^0_v(x) = g_v(x) + \log\frac{\abs{x_0}_v}{\max(\abs{x_0}_v,\dots,\abs{x_k}_v}. \]
Par d\'efinition, pour toute place~$v$, la fonction $g^0_v$ est born\'ee,
et est identiquement nulle pour presque toute place~$v$.

Soit alors $K'$ un corps de nombres contenant~$K$
et soit $x\in (X\setminus D)(K')$,
de conjugu\'es $x^{(1)},\dots,x^{(n)}$ sur~$K$. 
Si $\phi(x)$ a pour coordonn\'ees homog\`enes $[x_0:\dots:x_k]$,
on a donc $x_0\neq 0$ et
\begin{align*}
h_\phi(x) & = h_{\P^k}( \phi(x) )  
   = \frac1{[K':\Q]} \sum_{p\leq\infty}
  \sum_{\sigma\colon K'\hra\C_p} \log\max\left(\abs{\sigma(x_0)},\dots,
  \abs{\sigma(x_k)}_p \right) \\
 & = - \frac1{[K':\Q]} \sum_{p\leq\infty}
  \sum_{\sigma\colon K'\hra\C_p} \log\frac{\abs{\sigma(x_0)}}{\max\left(\abs{\sigma(x_0)},\dots, \abs{\sigma(x_k)}_p \right)} \\
& = - \frac1{n[K:\Q]}\sum_{i=1}^n \sum_{v\in M_K} \eps_v (g^0_v(x^{(i)})-g_v(x^{(i)})\\
&= -\frac1{n[K:\Q]}\sum_{i=1}^n \sum_{v\in M_K}  \eps_vg^0_v(x^{(i)})
+ \frac1{n[K:\Q]}\sum_{v\in M_K}  \eps_v\sum_{i=1}^n g_v(x^{(i)}) \\
&= -\frac1{n[K:\Q]}\sum_{i=1}^n \sum_{v\in M_K}  \eps_vg^0_v(x^{(i)})
+ \frac1{[K:\Q]}\sum_{v\in M_K}  \eps_v\hat h_v(x) .
\end{align*}
La proposition r\'esulte alors de ce que la premi\`ere somme est born\'ee
ind\'ependamment de~$x$.
\end{proof}

\subsection{La hauteur d\'etermine les fonctions de Green}

Soit $X$ une vari\'et\'e alg\'ebrique projective sur un corps de nombres,
soit $D$ un diviseur de Cartier sur~$X$ et soit $h_D$
une hauteur pour~$D$.
Il est naturel de se demander dans quelle mesure~$h_D$
d\'etermine~$D$. Suivant qu'on se donne~$h_D$
exactement, ou \`a~$\mathrm O(1)$ pr\`es, la r\'eponse est fournie par
le r\'esultat suivant.

\begin{theo}
Soit $X$ une vari\'et\'e projective lisse sur un corps de nombres~$K$;
soit $D$ un diviseur de Cartier sur~$X$ et soit $h_D$
une hauteur pour~$D$.
\begin{enumerate}\item
Supposons que $h_D$ soit born\'ee ; alors la classe de~$D$ est de torsion 
dans le groupe de Picard de~$X$ :
il existe un entier~$n$ et une fonction rationnelle~$f$ sur~$X$
telle que $nD=\div(f)$.
\item
Supposons que $h_D$ poss\`ede une d\'ecomposition en somme de
termes locaux, donn\'es par une famille admissible~$(g_v)$
de fonctions de Green pour le 
diviseur~$D$. Si $h_D$ est constante, chacune des fonctions~$g_v$
est constante.
\end{enumerate}
\end{theo}

Ce th\'eor\`eme
est d\^u \`a A.~\textsc{N\'eron} pour la premi\`ere partie, voir~\cite{serre1997},
\S 2.9 et~3.11, et \`a~\cite{agboola-pappas2000} pour la seconde.
(Dans cet article, il est d'ailleurs observ\'e  
qu'il suffit de supposer~$X$ normale dans l'\'enonc\'e du th\'eor\`eme).
La seconde partie est tout particuli\`erement int\'eressante dans les
contextes
o\`u l'on dispose de familles admissibles canoniques de fonctions
de Green, en particulier celui des syst\`emes dynamiques
(voir~\cite{kawaguchi-silverman2007a}).

\section{Exercices}

\begin{exer}
Soit $\xi$ un nombre alg\'ebrique, 
soit $d$ son degr\'e et soit $P=a_0X^d+\cdots + a_d$ le polyn\^ome minimal.
on note $\mathrm H(P)=\max(\abs{a_0},\ldots,\abs{a_d})$.
On rappelle que la hauteur~$h(\xi)$ de~$\xi$ est par
d\'efinition celle du point de coordonn\'ees homog\`enes~$[1:\xi]$
de~$\P^1$. 

Soit $P=a_0X^d+\cdots+a_d$ un polyn\^ome \`a coefficients complexes
de degr\'e~$d$.
On note $\mathrm H(P)=\max(\abs{a_0},\ldots,\abs{a_d})$.

a) Montrer que l'on a l'in\'egalit\'e
\[ 2^{-d} \mathrm H(P) \leq \mathrm M(P) \leq \sqrt{d+1} \mathrm H(P). \]

b) En d\'eduire que pour deux polyn\^omes~$P_1$ et~$P_2$, \`a coefficients
complexes, on a 
\begin{gather}
 \mathrm H(P_1)\mathrm H(P_2) \leq 2^d\sqrt{d+1} \mathrm H(P_1P_2) \\
 \mathrm H(P_1P_2) \leq 2^d \big(1+\frac d2\big) \mathrm H(P_1)\mathrm H(P_2),
\end{gather}
o\`u $d=\deg(P_1P_2)$.

c) Si $P$ est le polyn\^ome minimal d'un nombre alg\'ebrique~$\xi$,
montrer que la hauteur~$h(\xi)$ est encadr\'ee comme suit:
\[ \frac1d \log H(P) - \log 2 \leq h(\xi) \leq \frac1d \log H(P) + \frac1{2d}\log(d+1). \]
\end{exer}

\begin{exer}
Soit $f\in\bar\Q(t)$ 
une fraction rationnelle non constante, \'ecrite sous la forme~$P/Q$
d'un quotient de polyn\^omes $P$ et~$Q\in \bar\Q[t]$, premiers entre eux;
on pose $d=\max(\deg P,\deg Q)$.
Montrer que l'on a
\[ \lim_{h(\xi)\ra\infty} \frac{h(f(\xi))}{h(\xi)}= d . \]
\end{exer}

\begin{exer}\label{exer.lehmer1}
Soit $\xi$ un entier alg\'ebrique de degr\'e~$d$; on note
$\xi_1,\ldots,\xi_d$ ses conjugu\'es dans~$\C$
et, pour $n\in\N$, $S_n=\sum_{j=1}^d\xi_j^n$
la $n$-i\`eme somme de Newton. On pose aussi
$\mu(\xi)=\max(\abs{\xi_1},\ldots,\abs{\xi_d})$ 

a) Montrer que pour tout entier~$n$, $S_{n}$ est un entier
relatif qui v\'erifie $\abs{S_n}\leq d\mu(\xi)^n$.
(Utiliser le th\'eor\`eme sur les fonctions sym\'etriques.)

b) Soit $p$ un nombre entier tel que $S_n=S_{np}$ pour
$1\leq n\leq d$; montrer que $\xi$ est une racine de l'unit\'e
(ou $\xi=0$).

c) Montrer que pour tout nombre premier~$p$, $S_{np}$
et $S_n^p$ sont congrus \`a~$S_n$ modulo~$p$.
(Observer que les coefficients du polyn\^ome
sym\'etrique $X_1^p+\cdots+X_d^p-(X_1+\cdots+X_d)^p$
sont multiples de~$p$.)

d) On suppose que $\xi\neq 0$ et que $\xi$ n'est pas une racine
de l'unit\'e; on va montrer que $\mu(\xi)\geq e^{1/(4ed^2)}$.
Supposons par l'absurde que l'ingalit\'e inverse soit vraie
et soit $p$ un nombre preier tel que $2ed<p<4ed$ (il en existe
d'apr\`es le th\'eor\`eme de Tch\'ebitcheff --- postulat de \textsc{Bertrand}).
Montrer que $\abs{S_{np}-S_n}<p$ pour $1\leq n\leq d$,
puis que $S_{np}=S_n$ pour $1\leq n\leq d$.
En d\'eduire que $\xi$ est une racine de l'unit\'e.

e) Sous la m\^eme hypoth\`ese que d), montrer que 
$h(\xi)\geq 1/(4ed^3)$.
Pour d'autres r\'esultats dans la m\^eme veine,
voir l'exercice~\ref{exer.lehmer2} du chapitre~\ref{chap.sysdyn},
ainsi par exemple que le livre de~\cite{waldschmidt2000}
o\`u est d\'emontr\'e le meilleur r\'esultat connu, d\^u \`a~\cite{dobrowolski1979},
et repris dans la prop.~\ref{prop.w-dobrowolski} du chapitre~\ref{chap.sysdyn}.
\end{exer}

\begin{exer}
Soit $p\colon\P^2\dashrightarrow\P^1$ la projection lin\'eaire donn\'ee
par $p([x_0:x_1:x_2])=[x_0:x_1]$, d'unique point
d'ind\'etermination~$Q=[0:0:1]$.
Soit $X$ une courbe de~$\P^2$, d'\'equation homog\`ene $f(x_0,x_1,x_2)=0$.
On suppose que $X$ ne contient pas le point~$Q$.
D\'eterminez (en fonctions des coefficients de~$f$)
un nombre r\'eel~$c_X$ tel que $\abs{h(p(x))-h(x)}\leq c_X$
pour tout $x\in X(\bar\Q)$.
\end{exer}

\begin{exer}
Soit $X$ une vari\'et\'e alg\'ebrique projective int\`egre  d\'efinie sur
un corps de nombres~$F$
et soit $k$ sa dimension. 

a) D\'emontrer qu'il existe un morphisme fini $f\colon X\ra P^k$.
(Plonger $X$ dans un espace projectif~$\P^n$ et v\'erifier
qu'une projection g\'en\'erique convient.)

b) Soit $\mathscr L$ un fibr\'e en droites ample sur~$X$.
D\'emontrer qu'il existe un entier $n\geq 1$ tel que $\mathscr L^n\otimes f^*\mathscr O(-1)$ soit ample. En d\'eduire que pour toute fonction
hauteur sur~$X$ associ\'ee \`a~$\mathscr L$,
il existe des nombres r\'eels $a>0$ et $b$
tels que $h_{\mathscr L}(x)\geq a h(f(x))-b$
pour tout $x\in\bar\Q$.

c) D\'emontrer que pour tout $x\in X(\bar\Q)$,
le degr\'e de l'extension $F(x)/F(f(x))$ est major\'e par le degr\'e de~$f$.

d) Observer que la hauteur n'est pas major\'ee sur $\P^k(F)$
(il suffit de traiter le cas o\`u $F=\Q$).
\`A l'aide du th\'eor\`eme de finitude (cor.~\ref{coro.finitude}),
en d\'eduire que la hauteur~$h_{\mathscr L}$
n'est pas major\'ee sur~$X(\bar\Q)$.

e) D\'emontrer qu'il existe un nombre r\'eel~$c$ tel que l'ensemble
des points de $X(\bar\Q)$  de hauteurs~$\leq c$ soit dense dans~$X$
pour la topologie de Zariski. (Consid\'erer des points~$x\in  X(\bar\Q)$
tels que les coordonn\'ees homog\`enes de~$f(x)$ soient des racines de l'unit\'e.)
\end{exer}

\begin{exer}
Soit $f\colon\P^2\dashrightarrow\P^1$ une application rationnelle
donn\'ee par trois polyn\^omes $(f_0,f_1,f_2)$ de m\^eme degr\'e~$d$
et sans facteur commun.
Son lieu d'ind\'etermination~$Z$ est un ensemble fini de points.
Si $X$ est une courbe de~$\P^2$, suppos\'ee lisse ou au moins lisse en tout point
de~$X\cap Z$, la restriction de~$f$ \`a~$X\setminus (X\cap Z)$
s'\'etend en un unique morphisme~$\tilde f$ de~$X$ dans~$\P^1$.

a) Le degr\'e~$\tilde d$ de~$\tilde f$ est inf\'erieur ou \'egal \`a~$d$,
en fait \'egal \`a $d-\Card(X\cap Z)$ (cardinal compt\'e avec multiplicit\'es).
Pour que $\tilde d=d$, il faut et il suffit que $X$ ne rencontre pas~$Z$.

b) Il existe un nombre r\'eel~$c_X$ tel que pour tout $x\in X(\bar\Q)$,
$\abs{\tilde d h(x)-h(\tilde f(x))}\leq c_X$.
\end{exer}
\begin{exer}
Soit $X$ une vari\'et\'e projective et soit $D$ un diviseur
de Cartier effectif sur~$X$. Soit $h_D$ une hauteur
relative au fibr\'e en droites $\mathscr O(D)$ sur~$X$.
Il existe un nombre r\'eel~$c$ tel que $h_D(x)\geq c$
pour tout $x\in X(\bar\Q)$ tel que $x\not\in D$.
\end{exer}

\begin{exer}
Soit $X$ une vari\'et\'e projective et soit $\mathscr L$
un fibr\'e en droites sur~$X$. Soit $B$ le lieu des z\'eros
communs de toutes les sections des puissances de~$\mathscr L$.
(Dire que $B=\emptyset$ signifie donc qu'une puissance de~$\mathscr L$
est engendr\'ee par ses sections globales.)
Soit $h_{\mathscr L}$ une hauteur relative \`a~$\mathscr L$ sur~$X$.
Il existe un nombre r\'eel~$c$ tel que $h_{\mathscr L}(x)\geq c$
pour tout $x\in X(\bar\Q)$ qui n'appartient pas \`a~$B$.
\end{exer}

\begin{exer}
Soit $\mathscr L$ un fibr\'e en droites sur~$X$
tel qu'il existe un fibr\'e en droites ample~$\mathscr M$
de sorte que $\mathscr L\otimes\mathscr M^{-1}$
soit effectif (en d'autres termes, $\mathscr L$ est \emph{gros}).
Il existe un ferm\'e de Zariski strict~$Z$ de~$X$
hors duquel la propri\'et\'e de finitude pour la hauteur $h_{\mathscr L}$
est v\'erifi\'ee.
\end{exer}

\begin{exer}\label{exer.implicite}
Le r\'esultat suivant \'etait implicite dans la d\'emonstration
des \'enonc\'es de base sur les fonctions de Green.

Soit $X$ une vari\'et\'e projective sur un corps~$K$, valu\'e et
alg\'ebriquement clos et soit $D$ un diviseur de Cartier sur~$X$.
Soit $\lambda$ une fonction d'un ouvert $U$ de~$X\setminus D(K)$
dans~$\R$; on suppose que $U$ est dense dans~$X(K)$.

On suppose que pour tout point~$x\in X(K)$, il existe un diviseur
de Cartier tr\`es ample~$E$ ne contenant pas~$x$ et
une fonction de Green~$\lambda_E$ pour~$E$
tels que la fonction $\lambda+\lambda_E$ s'\'etende (n\'ecessairement
uniquement, car $U$ est dense) en une fonction
de Green pour~$D+E$. Alors $\lambda$ s'\'etend uniquement
en une fonction de Green pour~$D$.
\end{exer}

%% file: sysdyn.tex
\chapter{Systèmes dynamiques d'origine arithmétique}
\label{chap.sysdyn}

\section{Systèmes dynamiques polarisés}
\label{sec.sysdyn}

\subsection{La définition et quelques exemples}

Par définition, un \emph{système dynamique polarisé}
est la donnée d'une variété projective~$X$ (disons
intègre, mais pas forcément lisse),
d'un endomorphisme~$f$ de~$X$
et d'un fibré en droites ample~$\mathscr L$ sur~$X$
tel que $f^*\mathscr L$ soit isomorphe à une puissance~$\mathscr L^d$
de~$\mathscr L$, où $d$ est un entier \emph{supérieur ou égal à~$2$}.

L'entier~$d$, que \cite{zhang2006} propose d'appeler le \emph{poids}
de~$(X,f,\mathscr L)$,
est relié au \emph{degré de~$f$} par la formule
$\deg(f)=d^{\dim X}$. On a en effet
\[ c_1(f^*\mathscr L)^{\dim X} = d^{\dim X} c_1(\mathscr L)^{\dim X}
 = \deg(f) c_1(\mathscr L)^{\dim X}, \]
d'où l'assertion en divisant par $c_1(\mathscr L)^{\dim X}$
qui n'est pas nul puisque $\mathscr L$ est ample.

Comme l'a noté~\cite{serre60b}, l'action d'un tel endomorphisme
sur la cohomologie de~$X$ obéit à {\og l'analogue kählérien
des conjectures de Weil\fg}.\footnote{%
La conjecture de \textsc{Weil} en question est l'hypothèse 
de \textsc{Riemann} pour les variétés algébriques sur les corps finis :
elle concerne le cas où $X$ est une variété algébrique projective
lisse sur un corps fini de cardinal~$q$ et $f$ est
l'endomorphisme de Frobenius donné par l'élévation des coordonnées
à la puissance~$q$. Son poids est~$q$.}
Supposant~$X$ lisse et de dimension~$k$, l'endomorphisme~$f^*$
du groupe de cohomologie singulière~$H^j(X,\C)$  est diagonalisable
et ses valeurs propres sont toutes de valeur absolue archimédienne~$d^{j/2}$.
En particulier, tous les \emph{degrés dynamiques} de~$f$
(définis comme les rayons spectraux de~$f^*$ agissant 
sur la cohomologie singulière, voir les articles de~\textsc{Cantat}
et~\textsc{Guedj}) sont strictement dominés par le dernier, égal à~$d^k$.
L'étude des systèmes dynamiques polarisés apparaît ainsi 
comme un cas particulier des études plus spécifiquement dynamiques
exposées dans ce volume.

\medskip

Nous avons déjà donné au paragraphe~\ref{sec.hauteur-Q}/\ref{subsec.irrat-prep}
l'exemple de l'espace projectif~$\P^k$
et d'un endomorphisme~$f$ défini par une famille $(f_0,\dots,f_k)$
de polynômes homogènes de degré~$d$ sans zéro commun
autre que $(0,\dots,0)$.
On a en effet $f^*\mathscr O(1)\simeq\mathscr O(d)$.
Les sous-variétés de~$\P^k$ qui sont stables par~$f$
fournissent de même un système dynamique polarisé.
Par un résultat de~\cite{fakhruddin2003}, c'est en fait le
cas général:

\begin{prop}[Fakhruddin] \label{theo.fakhruddin}
Soit $(X,f,\mathscr L)$ un système dynamique polarisé défini
sur un corps infini~$C$.\footnote{Cette hypothèse, reprise de~\cite{fakhruddin2003},
n'est probablement pas nécessaire pour le présent énoncé.}
Il existe un plongement~$\iota$ de~$X$
dans un espace projectif~$\P^k$ tel que $\iota^*\mathscr O(1)$
soit isomorphe à une puissance de~$\mathscr L$,
et un endomorphisme~$F$ de degré~$d$ de~$\P^k$
tels que $F\circ\iota=\iota\circ f$.
\end{prop}
(Dans ce qui suit, remplacer $\mathscr L$ par une puissance
sera souvent inoffensif.)
\begin{proof}
Quitte à remplacer~$\mathscr L$ par une de ses puissances~$\mathscr L^d$,
on peut supposer que $\mathscr L$ est très ample
et que, pour tout entier~$m>0$, 
d'une part $H^i(X,\mathscr L^m)=0$ pour tout entier~$i>0$,
et d'autre part 
l'homomorphisme de multiplication
\[ H^0(X,\mathscr L^m)\otimes H^0(X,\mathscr L)\ra H^0(X,\mathscr L^{m+1})\]
est surjectif.
Notons alors $\iota\colon X\hra\P^k$ le plongement de~$X$
dans un espace projectif défini par~$\mathscr L$; 
par construction, on a $\iota^*\mathscr O_{\P^k}(1)=\mathscr L$
et, si $m \geq 1$, l'homomorphisme naturel de $H^0(X,L)^{\otimes m}$
dans $H^0(X, L^m)$ est surjectif.
En outre, par un raisonnement élémentaire
de régularité de Castelnuovo--Mumford,
le faisceau d'idéaux de~$X$ dans~$\P^k$ est de régularité au plus~$2$;
en particulier, l'idéal homogène de~$X$ dans~$\P^k$
est engendré par des polynômes de degré~$d$
(on a supposé $d\geq 2$).

Pour $j\in\{0,1,\ldots,\dim X\}$,
choisissons par récurrence une section~$s_j$ dans~$H^0(X,\mathscr L)$
qui n'est identiquement nulle sur aucune composante
irréductible du lieu d'annulation
commun de~$s_0,\ldots,s_{j-1}$. Alors, toute composante
irréductible du lieu d'annulation commun de~$s_0,\ldots,s_{j}$
est de codimension~$>j$ dans~$X$. En particulier, $s_0,\ldots,s_{\dim X}$
n'ont pas de zéro commun.
On peut en outre supposer ces sections linéairement indépendantes.
Complétons-les alors en une base $(s_0,\ldots,s_k)$ de~$H^0(X,\mathscr L)$.

Soit $j\in\{0,\ldots,k\}$; alors, $f^*s_j$ est une section
de $H^0(X,f^*\mathscr L)=H^0(X,\mathscr L^d)$ et,
par l'hypothèse faite sur~$\mathscr L$, il existe
un polynôme~$F_j\in C[X_0,\ldots,X_k]$, homogène de degré~$d$
tel que $f^*s_j=F_j(s_0,\ldots,s_k)$.
Les polynômes $F_0,\ldots,F_k$ définissent une \emph{application rationnelle}
$F\colon [X_0:\cdots:X_k]\mapsto [F_0:\cdots:F_k]$
de~$\P^k$ dans lui-même, définie là où les~$F_j$ ne s'annulent
pas simultanément.
Nous allons démontrer qu'il est possible de modifier les~$F_j$
de sorte que $F$ soit définie partout.

Plus précisément, nous allons démontrer qu'il existe
des polynômes $F_0,\ldots,F_k$ dans~$C[X_0,\ldots,X_k]$,
homogènes de degré~$d$, vérifiant $f^*s_j=F_j(s_0,\ldots,s_k)$ et tels que 
pour tout entier $s\in\{\dim X,\ldots,k\}$ et toute composante
irréductible~$Z$ de $\mathbf V(F_0,\ldots,F_j)\subset\P^k$,
$\codim(Z,\P^k)> j$. (On note $\mathbf V(\cdots)$ le lieu
de~$\P^k$ défini par une famille de polynômes homogènes.)

Observons que l'on peut conserver le choix précédent 
pour $F_0,\ldots,F_{\dim X}$. En effet, si $Z$ est une composante
irréductible de~$\mathbf V(F_0,\ldots,F_{\dim X})$,
$Z\cap X=\emptyset$ (un point $x$ de~$X\cap Z$ vérifierait
$f^*s_j(x)=0$ pour $0\leq j\leq \dim X$
et $f(x)$ serait un zéro commun de $s_0,\ldots,s_{\dim X}$).
En particulier, 
$\codim(Z,\P^k)>\dim X$ car
deux sous-variétés de~$\P^k$ de dimensions {\og complémentaires\fg} ont un point
d'intersection.

Supposons maintenant $j>\dim X$ et $F_0,\ldots,F_{j-1}$ construits.
Notons $\mathscr P_d$ l'espace vectoriel des polynômes de~$C[X_0,\ldots,X_k]$
qui sont homogènes de degré~$d$
et $\mathscr I_d$ le sous-espace vectoriel formé des polynômes
qui sont identiquement nuls sur~$X$. Pour $G\in\mathscr P_d$,
la condition $f^*s_j=G(s_0,\ldots,s_k)$ définit un sous-espace
affine~$A$ de~$\mathscr P_d$, d'espace vectoriel associé~$\mathscr I_d$.
Toujours pour $G\in\mathscr P_d$, et pour $Z$ une composante
irréductible de~$\mathbf V(F_0,\ldots,F_{j-1})$, la
condition que $G$ soit identiquement nul sur~$Z$
définit un sous-espace vectoriel~$V_Z$ de~$\mathscr P_d$.
Il s'agit de trouver un élément de~$A$ qui n'appartienne à
aucun des sous-espaces~$V_Z$. Comme le corps~$C$ est infini,
c'est possible dès lors qu'aucun des sous-espaces~$V_Z$ ne contient~$A$.

En initialisant la récurrence, on a démontré 
que ces composantes~$Z$ sont disjointes de~$X$.
Or l'hypothèse que l'idéal homogène de~$X$ dans~$\P^k$ est
engendré par des polynômes de degré~$d$ implique précisément
qu'il existe, pour toute composante~$Z$ comme ci-dessus,
un polynôme $H\in\mathscr I_d$  qui n'est pas identiquement nul
sur~$Z$. Autrement dit, $\mathscr I_d$ n'est pas contenu
dans~$V_Z$ et \emph{a fortiori}, le sous-espace affine~$A$
dirigé par~$\mathscr I_d$ n'est pas contenu dans~$V_Z$.

Cela conclut par récurrence la construction des polynômes~$F_0,\ldots,F_k$.
En particulier,
les composantes irréductibles de~$\mathbf V(F_0,\ldots,F_k)$
sont de codimension~$>k$ dans~$\P^k$ : cela entraîne
que $\mathbf V(F_0,\ldots,F_k)=\emptyset$
et conclut, du même coup, la démonstration de
la proposition.
\end{proof}

\paragraph{Systèmes dynamiques abéliens}
Les variétés abéliennes fournissent des exemples fondamentaux
de systèmes dynamiques polarisés.
Considérons une variété abélienne~$X$ sur un corps~$F$,
c'est-à-dire une variété projective munie d'une structure
de groupe algébrique.  Pour tout entier~$n$, la multiplication par~$n$,
notée $[n]$, définit un endomorphisme de~$X$.
Le théorème du cube, voir par exemple~\cite{mumford74},
entraîne que pour tout fibré en droites
symétrique~$\mathscr L$, le fibré en droites $[n]^*\mathscr L$ 
est isomorphe à~$\mathscr L^{n^2}$.

Les points prépériodiques d'un tel système dynamique
sont les points de torsion de~$X$. L'intérêt 
de cette remarque est double: si elle suggère 
d'employer des méthodes de systèmes dynamiques 
pour aborder des questions arithmétiques ou géométriques
concernant les points de torsion d'une variété
abélienne, elle permet aussi d'envisager
la théorie des systèmes dynamiques polarisés
comme une \emph{extension} de la théorie arithmético-géométrique
des variétés abéliennes.

Plus généralement, on appellera \emph{système dynamique abélien}
un système dynamique de la forme~$(X,f)$, où $X$
est une variété abélienne et $f\colon X\ra X$ un 
endomorphisme de la variété algébrique~$X$,
c'est-à-dire la composition d'un endomorphisme~$\phi$
de~$X$ comme variété abélienne et d'une translation par
un point~$x_0$ de~$X$.
Lorsque $\phi-\id$ est surjectif, ce qui sera le cas
si $X$ est simple et $\phi\neq\id$ (ou, plus généralement
si pour toute sous-variété abélienne~$Y\neq 0$ de~$X$,
$\alpha|_Y\neq\id_Y$),  alors le système dynamique
$(X,f)$ est conjugué au système~$(X,\phi)$.

Il convient d'observer
qu'un système dynamique abélien n'est en général pas polarisé,
par exemple lorsque $X$ est un produit $X_1\times X_2$ et 
que $f=f_1\times f_2$ est donné par la multiplication
par deux entiers~$n_1$ et~$n_2$ distincts sur chacun des facteurs.

\paragraph{Systèmes dynamiques toriques}
Soit $d$ un entier tel que $d\geq 2$. On définit
un système dynamique sur l'espace projectif~$\P^k$
en posant $f([x_0:\cdots:x_k])=(x_0^d:\cdots:x_k^d)$.
C'est un système dynamique polarisé car $f^*\mathscr O(1)=\mathscr O(d)$;
il est de poids~$d$. Soit $G$ l'ouvert de~$\P^k$
où aucune coordonnée homogène ne s'annule ;
si l'on fixe la coordonnée homogène~$x_0$ égale à~$1$,
on voit que $G$ est isomorphe au tore algébrique $(\gm)^k$,
où $\gm=\A^1\setminus\{0\}$ est le groupe multiplicatif.

De la sorte, $\P^k$ apparaît comme une \emph{compactification}
de~$G$, compactification qui n'est pas du tout arbitraire
car la multiplication de~$G$, à savoir l'application
$m\colon G\times G\ra G$ qui définit la structure
de groupe sur $(\gm^k)$, s'étend en une action de~$G$
sur~$\P^k$, donnée par
$ (u,x)\mapsto [x_0:u_1x_1:\cdots:u_k x_k]$
si $u=(u_1,\dots,u_k)\in G$ et $x=[x_0:\cdots:x_k]\in \P^k$.

Plus généralement, on appelle \emph{variété torique}
une variété~$X$,  disons projective et lisse, contenant un
tore algébrique~$G$ comme ouvert dense, de sorte
que la multiplication de~$G$ s'étende en un morphisme
de~$G\times X$ dans~$X$.
Le complémentaire de~$G$ dans~$X$ est alors un diviseur~$D$,
d'ailleurs un diviseur canonique\footnote{C'est-à-dire
le diviseur d'une~$k$-forme différentielle méromorphe,
où $k$ est la dimension de~$X$, supposée lisse.} de~$X$.
Pour tout entier~$d\geq 2$, l'endomorphisme
$u\mapsto u^d$ de~$G$ s'étend en un morphisme $f\colon X\ra X$
pour lequel $f^*\mathscr O_X(D)\simeq \mathscr O_X(dD)$.
On obtient de la sorte un système dynamique polarisé
de poids~$d$, si ce n'est que $\mathscr O_X(D)$
n'est pas forcément ample --- il appartient néanmoins
à l'intérieur du cône effectif de~$X$.\footnote{%
De toutes façons, la théorie des variétés
toriques montre que l'image inverse par~$f$
de tout diviseur~$E$  est linéairement équivalente
à $dE$; les constructions ci-dessous ne dépendent
pas substantiellement du choix d'un diviseur ample~$E$.}
De tels systèmes dynamiques seront appelés \emph{toriques}.

Observons que pour $u\in\C$,
la suite $(u^{d^n})_n$ ne prend qu'au plus une fois chaque
valeur si $u$ n'est pas une racine de l'unité,
et ne prend qu'un nombre fini de valeurs sinon.
Autrement dit, les points prépériodiques de ces
systèmes dynamiques qui sont contenus dans~$G$
s'identifient aux $k$-uplets $(u_1,\dots,u_k)$ de racines
de l'unité.

Lorsque $X=\P^k$, observons d'ores et déjà
une remarquable propriété de la hauteur naturelle
vis à vis de ces endomorphismes:
\begin{lemm}
Pour $x=[x_0:\cdots:x_k]\in\P^k(\bar\Q)$ et $d\in\N$,
on a
\[ h([x_0^d:\cdots:x_k^d])= d\, h([x_0:\cdots:x_k]). \]
\end{lemm}
\begin{proof}
Cela résulte immédiatement de la définition
de la hauteur, compte-tenu du fait que pour tout corps valué $K$
et toute famille $(x_0,\dots,x_k)$ d'éléments de~$K$,
\[  \max(\abs{x_0}^d,\dots,\abs{x_k}^d)=\max(\abs{x_0},\dots,\abs{x_k})^d.\]
\end{proof}
C'est un premier exemple, d'ailleurs fondamental,
de hauteur \emph{normalisée} par rapport à un système dynamique.

Les systèmes dynamiques toriques ou abéliens
ne sont pas les plus intéressants
du strict point de vue  de la théorie des systèmes dynamiques;
en revanche, les questions arithmétiques qu'ils suscitent sont souvent
fondamentales.

\bigskip

\paragraph{\'Eléments de classification}
Soit $(X,f,\mathscr L)$ un système dynamique polarisé.
Supposons que $X$ soit lisse et géométriquement connexe.
D'après le lemme~4.1 de~\cite{fakhruddin2003},
voir aussi~\cite{cantat2003},
la dimension de Kodaira d'une telle variété~$X$
est négative ou nulle. 
Supposons que $\kod(X)=0$.
Alors, une puissance du fibré canonique de~$X$
est trivial,
et en caractéristique~$0$,
on peut démontrer 
que $(X,f)$ est déduit d'un système dynamique abélien~$(X',f')$
par passage au quotient par l'action d'un groupe fini
agissant sans point fixe sur~$X'$.
La démonstration utilise  un théorème
de~\cite{beauville1983}, 
voir aussi~\cite{bogomolov1974a,bogomolov1974b}, 
reposant sur la   solution de~\cite{yau1978} à la conjecture de~\textsc{Calabi},
à savoir l'existence d'une métrique kählérienne Ricci plate sur~$X$.

La classification en dimension~$2$ 
est essentiellement due 
à \cite{nakayama2002} (voir aussi~\cite{fujimoto2002},
ainsi que \cite{fujimoto-n2005}
pour l'analogue non kählérien) et
fait l'objet d'une proposition
(prop.~2.3.1) de~\cite{zhang2006}, article de synthèse sur le thème
de cette conférence et dont je recommande la lecture.
Les surfaces qui portent un système dynamique polarisé sont:
\begin{itemize}
\item les surfaces abéliennes ;
\item les surfaces hyperelliptiques possédant une revêtement
étale par le  produit de deux courbes elliptiques ;
\item les surfaces toriques ;
\item les surfaces réglées sur une courbe elliptique
associées, soit à un fibré de rang~$2$ de la 
forme~$\mathscr O\oplus\mathscr M$,
où $\mathscr M$  soit, ou bien de torsion, ou bien de degré non nul,
soit à un fibré indécomposable de degré impair.
\end{itemize}

Citons enfin un résultat de~\cite{beauville2001},
reposant sur des idées de~\cite{amerik-r-v1999},
selon lequel une hypersurface lisse de degré au moins~$3$
de l'espace projectif de dimension au moins~$3$ 
n'admet aucun endomorphisme de degré~$>1$.

Pour plus de détails, je renvoie à l'article de Serge~\textsc{Cantat}
dans ce volume.

\subsection{Hauteur normalisée}

Soit $(X,f,\mathscr L)$ un système dynamique polarisé
défini sur~$\bar\Q$. Notons $d$ son poids;
c'est l'unique entier~$\geq 2$
tel que $f^*\mathscr L\simeq\mathscr L^{d}$.

\`A la suite de~\cite{call-s93},
eux-mêmes inspirés par la normalisation
de~\cite{neron1965} et \textsc{Tate} (voir~\cite{lang1964})
dans le cas où $X$ est une variété abélienne,
le but de ce paragraphe est de définir,
une \emph{hauteur normalisée}
relative au fibré en droites~$\mathscr L$. 

Les résultats qui suivent sont des généralisations
directes des propositions correspondants du~\S\ref{sec.hauteur-Q}.
Leur démonstration est identique.
\begin{prop}\label{prop.tate}
\emph{(Rappelons que $d\geq 2$.)}
Il existe une unique hauteur relative à~$\mathscr L$,
$\hat h_{\mathscr L}\colon X(\bar\Q)\ra\R$
telle que $\hat h_{\mathscr L}(f(x))=d\hat h_{\mathscr L}(x)$ 
pour tout $x\in X(\bar\Q)$.
\end{prop}
\begin{proof}
Notons $E$ l'espace affine des hauteurs relatives à~$\mathscr L$
sur $X(\bar\Q)$ ; son espace vectoriel directeur est
l'espace~$\mathscr F_b$ des fonctions bornées de~$X(\bar\Q)$ dans~$\R$.
Munissons $\mathscr F_b$ de la norme uniforme et 
l'espace affine~$E$ de la distance induite.
C'est un espace métrique complet.

L'application $T\colon \phi\mapsto \frac1d \phi\circ f$
est linéaire et applique~$E$ dans lui-même.
En effet, si $h$ est une hauteur relative à~$\mathscr L$,
$h\circ f$ est une hauteur relative à~$f^*\mathscr L$
(prop.~\ref{prop.hof}) donc $h\circ f - d h$ est bornée.
Par suite, $T(f)=\frac1d h\circ f$ est une hauteur relative à~$\mathscr L$
sur~$X$.

Cette application~$T$ est contractante, de constante de Lipschitz 
au plus~$1/d<1$.
Elle possède donc un unique point fixe dans~$E$.
\end{proof}

La fonction~$\hat h_{\mathscr L}$ dont
la proposition précédente affirme l'existence et l'unicité
est appelée \emph{hauteur normalisée}, ou hauteur canonique.
On a aussi la formule
\begin{equation}
\label{eq.normalisee}
\hat h_{\mathscr L}(x) = \lim_{n\ra\infty} \frac1{d^n} h(f^n(x)) ,
 \end{equation}
pour toute hauteur~$h$ relative à~$\mathscr L$ sur~$X$.
La hauteur normalisée vérifie les propriétés suivantes:

\begin{prop}\label{prop.hauteur-normalisee-Qbarre}
\begin{enumerate}
\item On a $\hat h_{\mathscr L}(x)\geq 0$ pour tout $x\in X(\bar\Q)$;
\item un point $x\in X(\bar\Q)$ vérifie $\hat h_{\mathscr L}(x)=0$
si et seulement s'il est prépériodique ;
\item pour tout nombre entier~$D$
et tout nombre réel~$B$, l'ensemble de points $x\in X(\bar\Q)$
de degré au plus~$D$ et tels que $\hat h_{\mathscr L}(x)\leq B$ est fini.
\end{enumerate}
\end{prop}
\begin{proof}
La propriété~\emph c) résulte immédiatement de ce
que $\hat h_{\mathscr L}$ est une hauteur relative à un fibré
en droites ample sur~$X$ et du corollaire~\ref{coro.finitude}.
Comme les hauteurs relatives à un fibré en droites ample
sont minorées, 
la propriété~\emph a) est manifeste sur la formule~\eqref{eq.normalisee}
ci-dessus. Comme au \S\ref{sec.hauteur-Q},
elle se déduit aussi, ainsi que la propriété~\emph b),
de l'assertion de finitude.

Soit en effet $x\in X (\bar\Q)$. On a $\hat h(f^n(x))=d^n \hat h(x)$.
Si $x$ est prépériodique, il existe des entiers~$n\geq 0$ et~$p\geq 1$
tels que $f^n(x)=f^{n+p}(x)$. Par suite,
$d^n \hat h_{\mathscr L}(x)=d^{n+p}\hat h_{\mathscr L}(x)$, d'où $\hat h_{\mathscr L}(x)$ car $d\geq 2$.
Inversement, si $\hat h_{\mathscr L}(x)\leq 0$, les termes de la suite $(f^n(x))$ 
forment un ensemble de points de hauteur normalisée au plus~$B$,
tous définis sur le corps~$\Q(x)$; un tel ensemble est fini d'après~\emph c). 
Il existe donc des entiers~$n\geq 0$
et~$p\geq 1$ tels que $f^n(x)=f^{n+p}(x)$.
Autrement dit, $x$ est prépériodique et $\hat h_{\mathscr L}(x)=0$.
\end{proof}

Voici, d'après~\cite{northcott1950},
une conséquence remarquable de la proposition précédente.
\begin{coro}\label{theo.northcott}
Soit $(X,f,\mathscr L)$ un système dynamique polarisé
défini sur un corps de nombres~$K$.
Pour tout entier~$D$, les points de~$X(\bar\Q)$
qui sont prépériodiques et dont le degré est au plus~$D$
forment un ensemble fini.
\end{coro}
Il convient de remarquer ici
que, sous les hypothèses du théorème,
l'ensemble des points de~$X(\bar\Q)$ qui sont périodiques
est dense dans~$X$ pour la topologie de Zariski
(proposition~\ref{theo.per.dense}). Il y en est \emph{a fortiori}
de même de l'ensemble des points prépériodiques.

Lorsque, de plus, $X$ est une variété abélienne,
la hauteur normalisée obtenue est appelée \emph{hauteur de N\'eron--Tate}.\label{hNT}
La compatibilité d'une telle hauteur avec la loi de groupe
de~$X$ est assez remarquable : en effet, cette hauteur $\hat h_{\mathscr L}$
est une fonction de degré~$\leq 2$ sur le groupe~$X(\bar\Q)$,
au sens où l'identité suivante est vérifiée:
\begin{equation}
\label{eq.hNT}
  \hat h_{\mathscr L}(x+y+z) - \hat h_{\mathscr L}(y+z)
-  \hat h_{\mathscr L}(z+x) -  \hat h_{\mathscr L}(x+y) 
+ \hat h_{\mathscr L}(x)
+ \hat h_{\mathscr L}(y)
+ \hat h_{\mathscr L}(z)
-\hat h_{\mathscr L}(0) = 0 \end{equation}
pour tout triplet $(x,y,z)$ de points de~$X(\bar\Q)$.
Une telle fonction est la somme d'une  forme quadratique $q_{\mathscr L}$,
d'une forme linéaire $\ell_{\mathscr L}$ et d'une constante
(exercice~\ref{exo.hNT}).
Dans le cas où les fibrés en droites $\mathscr L$ 
et $[-1]^*\mathscr L$ sont isomorphes, on dit
que $\mathscr L$ est symétrique et l'on a $\ell_{\mathscr L}=0$;
dans le cas antisymétrique où $\mathscr L$ et $[-1]^*\mathscr L$
sont inverses l'un de l'autre, on a $q_{\mathscr L}=0$.

D'autre part, lorsque $X=\P^1$ et $f$ est l'endomorphisme
$[x:y]\mapsto [x^d:y^d]$ pour $d$ un entier~$\geq 2$,
la hauteur normalisée n'est autre que la hauteur
naturelle sur~$\P^1$. Les points prépériodiques
pour ce système dynamique sont $[0:1]$, $[1:0]$
et les points~$[1:\xi]$ où $\xi\in\C$ est une racine
de l'unité.
Modulo l'identification
entre hauteur d'un point $[1:\xi]$ 
et mesure de Mahler~$\mathrm M(P)$ du polynôme minimal~$P$ de~$\xi$,
on en déduit le théorème suivant, dont
la première partie est essentiellement due à \textsc{Kronecker.}
Voir aussi l'exercice~\ref{exer.lehmer1} pour une version
effective.
\begin{coro}
Soit $P$ un polynôme irréductible à coefficient entiers.

\begin{enumerate}
\item Si $\mathrm M(P)=0$, les racines de~$P$ sont
des racines de l'unité ou~$0$.
\item Lorsque $n\ra\infty$, le degré du corps engendré
sur~$\Q$ par une racine primitive $n$-ième de l'unité
tend vers l'infini.
\end{enumerate}
\end{coro}
En fait, toutes les racines primitives $n$-ièmes
de l'unité sont conjuguées sur~$\Q$
(\textsc{Gauß}) puisque le polynôme cyclotomique~$\Phi_n$
est irréductible. Comme l'indicatrice d'\textsc{Euler}
$\phi(n)$ tend vers l'infini avec~$n$,
cela redonne la seconde assertion.
La première assertion peut également
être précisée en disant qu'un polynôme irréductible $P\in\Z[X]$
tel que $\mathrm M(P)=0$ est, au signe près, ou bien égal à~$X$,
ou bien égal à un polynôme cyclotomique.

\subsection{Normalisation des hauteurs locales}

On a étudié au paragraphe~\ref{sec.locales}
les décompositions de la hauteur d'un point en somme
de termes locaux, indexés par les places de~$K$.
Dans le cas d'un système dynamique polarisé,
il est naturel de se demander si la hauteur normalisée
est justiciable d'une telle décomposition  dont chaque
terme serait plus ou moins canonique.
Dans le cas des variétés abéliennes, la normalisation
obtenue remonte à~\cite{neron1965}.

Revenons donc à la théorie des hauteurs locales
en nous plaçant sur un corps~$K$, supposé valué complet et algébriquement clos.
Soit $X$ une variété projective sur~$K$,
$f\colon X\ra K$ un endomorphisme de~$X$
et soit $D$ un diviseur de Cartier sur~$X$
tel que $f^*D$ soit linéairement équivalent à~$dD$,
où $d$ est un entier~$\geq 2$.
Autrement dit, dans le cas où $D$ est ample,
$(X,f,\mathscr O(D))$ est un système dynamique polarisé.

Il y a deux façons pour définir une fonction de Green
canonique relativement à~$D$.
La première, due à~\cite{call-s93},
procède d'un raisonnement au niveau des fonctions ;
la seconde, due à~\cite{zhang95b}, construit
des métriques canoniques.
Je mélange ici les deux points de vue, 
les espaces affines des fonctions de Green pour un diviseur~$D$
n'étant que l'image par $\log\abs{\cdot}$ du torseur
des métriques hermitiennes continues sur $\mathscr O(D)$.

\begin{prop}
Soit $\alpha\in K(X)$ tel que $f^*D=dD+\div(\alpha)$.
Il existe une unique fonction de Green $\hat g_D$
relativement à~$D$ telle que
\[ \hat g_D (f(x)) = d \hat g_D(x) - \log \abs{\alpha(x)} \]
pour tout $x\in X(K)$ qui n'appartient ni à~$D$
ni à~$f^*D$.
\end{prop}
\begin{proof}
L'ensemble des fonctions de Green pour~$D$ est un espace affine
de direction l'espace vectoriel des fonctions continues bornées
sur~$X(K)$. La norme sup de la différence de deux fonctions
de Green définit alors une distance sur cet espace qui 
en fait un espace métrique complet.
Si $g$ est une fonction de Green pour~$D$,
$T(g) = \frac1d (g\circ f+\log\abs\alpha)$ en est une autre.
L'application $g\mapsto T(g)$ est contractante,
de constante de Lipschitz au plus~$1/d<1$.
Elle admet donc un unique point fixe dans l'ensemble
des fonctions de Green pour~$D$.
\end{proof}

\begin{exem}
Supposons que~$D$ soit la section hyperplane~$X_0=0$
de l'espace projectif $X=\P^k$.
D'après la structure des endomorphismes de l'espace projectif,
il existe des polynômes~$(F_0,\ldots,F_k)$ de~$K[X_0,\dots,X_k]$,
homogènes de degré~$d$
et sans zéro commun non trivial dans~$\bar K$, tels que
$f([x_0:\cdots:x_k])=[F_0(x):\cdots:F_k(x)]$, pour tout
point~$[x_0:\cdots:x_k]$ de~$\P^k$.
Notons $F=(F_0,\ldots,F_k)$ l'endomorphisme de l'espace affine
de dimension~$k+1$ qui relève~$f$. Pour $n\geq 0$,
notons $F^{n}=(F_0^{n},\ldots,F_k^{n})$ le $n$-ième itéré de~$f$;
il relève~$f^n$. En outre, les polynômes $F_j^n$ (pour $0\leq j\leq j$)
sont de degré~$d^n$ et sans zéro commun non trivial.

Comme $D$ est le diviseur des zéros de la {\og fonction\fg}~$x_0$,
$f^*D$ est celui de~$F_0$ ; notant $\alpha$ la fraction
rationnelle~$F_0(x)/x_0^d$, on a $f^* D = dD+\div(\alpha)$.
La fonction $g=-\log \left(\abs{x_0}/\max(\abs{x_0},\ldots,\abs{x_k})\right)$
est une fonction de Green pour~$D$; toute autre
fonction de Green en diffère d'une fonction continue.
La démonstration itérative du théorème du point fixe 
montre que la fonction de Green canonique pour~$g$
est la limite des fonctions de Green
\begin{align*}
 g_n(x) & = - \frac1{d^n} g\circ f^n(x) + \frac1{d^{n}} \log\abs{\alpha\circ f^{n-1}}+\cdots + \frac 1d \log\abs{\alpha(x) } \\
& = - \frac1{d^n} \log \frac{\abs{F_0^n(x)}}{\max(\abs{F_0^n(x)},\ldots,\abs{F_k^n(x)})}
+ \sum_{m=1}^n \frac1{d^m} \log \frac{ \abs{F_0^{m}(x)}}{\abs{F_0^{m-1}(x)}^d} \\
& = \frac1{d^n} \log \max\big(\abs{F_0^n(x)},\ldots,\abs{F_k^n(x)}\big)
 - \log \abs{x_0}. 
\end{align*}
Pour tout $x=(x_0,\ldots,x_k)$ non nul, posons
\[ G_n(x) = \frac1{d^n} \log \max(\abs{F_0^n(x)},\ldots,\abs{F_k^n(x)}). \]
Les calculs qui précèdent entraînent que $G_n$ converge
vers une fonction~$G$ définie sur le complémentaire de l'origine dans l'espace
affine et vérifiant les propriétés suivantes :
\[ G(\lambda x) = \log \abs\lambda + G(x) ,
\qquad
  G(F(x))= d G(x).
\]
On l'appelle la \emph{fonction de Green homogène} ;
elle est reliée à la fonction de Green canonique pour le diviseur~$D$
par la relation $g_D(x)=G(x)-\log\abs{x_0}$.

Supposons $K=\C$. Alors, $G$ est plurisousharmonique
dans~$\C^{k+1}\setminus\{0\}$
et est pluriharmonique sur l'image réciproque de l'ensemble
de Fatou de~$f$; voir~\cite{sibony1999}, \S1.6.
Lorsque $K$ est un corps ultramétrique, 
\cite{kawaguchi-silverman2007b} démontrent
que $G$ est localement constante sur l'image réciproque
de l'ensemble de Fatou de~$f$.
\end{exem}

Le théorème suivant affirme que ces fonctions de Green normalisées
forment une famille admissible, donc donnent lieu comme anoncé
à une décomposition de la hauteur normalisée d'un point en
somme de termes locaux.

\begin{theo}
Soit $(X,f,\mathscr L)$ un système dynamique polarisé défini
sur un corps de nombres~$K$.
Soit $D$ un diviseur tel que $\mathscr L\simeq\mathscr O(D)$
et soit $\alpha\in K(X)$ tel que $f^*D =dD+\div(\alpha)$.
Pour toute place~$v$ de~$K$, notons
$\hat g_v$ la fonction de Green sur le corps~$\C_v$
pour le diviseur~$D$, normalisée
relativement à~$\alpha$, ainsi que $\hat h_v$ la hauteur locale associée. 

Alors, la famille $(\hat g_v)$ est admissible et,
pour tout $x\in (X\setminus D)(\bar\Q)$, on a
\[ \hat h_{\mathscr L}(x) = \sum_{v\in M_K} \eps_v \hat h_v(x). \]
\end{theo}
(Comme au paragraphe~\ref{sec.locales}, si $v$ est une place de~$K$, 
on note $\eps_v$ le nombre de plongements 
du corps~$K$ dans~$\C_v$ qui induisent la valeur absolue correspondant
à~$v$.)

\begin{proof}
La famille~$(\hat g_v)$
de fonctions de Green normalisées est obtenue en itérant
à partir d'une famille admissible $(g^0_v)$ fixée
 l'opérateur~$T\colon (g_v)\mapsto (\frac1d(g_v\circ f+\log\abs\alpha_v)$.
D'après le lemme~\ref{fonctorialite.admissible},
cet opérateur applique une famille admissible de fonctions
de Green sur une autre famille admissible.
Plus précisément, d'après cette proposition, cet opérateur 
fixe presque toutes les composantes~$g_v$.
Autrement dit, on a $\hat g_v=g^0_v$ pour presque toute place~$v$.
Cela démontre précisément que la famille~$(\hat g_v)$ est admissible.

Notons $h'$ la somme des hauteurs locales associées ; 
d'après la proposition~\ref{prop.decomposition},
c'est la restriction au complémentaire de~$D$ d'une hauteur
pour ce diviseur.
Pour montrer que c'est bien la hauteur normalisée, il suffit
donc de vérifier l'équation  fonctionnelle qui la caractérise.

Soit $v$ une place de~$K$.
Pour tout point $x\in X(\C_v)$ tel que $x\not\in D$ et $f(x)\not\in D$,
on a $\hat g_v(f(x))=d \hat g_v(x)-\log\abs{\alpha(x)}_v$.
Par suite, si un point $P\in X(\bar K)$ 
est tel que $P\not\in D$ et $f(P)\not\in D$, 
on a les égalités
\[ \hat h_v(f(P))=d\hat h_v(P)-\frac1{[K(P):K]} \log \abs{N_{K(P)/K}(\alpha(x))}_v.\]
En sommant ces égalités et en utilisant la formule du produit,
on en déduit $h'(f(P))=d h'(P)$. Pour traiter le cas général,
on considère un diviseur~$E$ linéairement équivalent à~$D$
qui ne contienne ni~$P$ ni~$f(P)$
auquel on applique l'analyse précédente.
\end{proof}

\subsection{Hauteurs locales sur~$\C$ et mesures canoniques}
\label{subsec.mesures-canoniques}

Conservons les notations du paragraphe précédent
en supposant de plus que $K=\C$ et que $D$ est un diviseur ample.

Il possède alors
une fonction de Green de classe~$\mathscr C^\infty$,
$g_D$, dont la forme
associée $\omega_D=\ddc  g_D+\delta_D$ est une forme de Kähler.
D'après le théorème de \textsc{Wirtinger},
la mesure $(\omega_D)^k$ est positive, de masse
totale égale à l'auto-intersection $(D)^k$ de~$D$
(si $D$ est une section hyperplane, ce n'est autre que le degré de~$X$,
c'est-à-dire le nombre de points d'intersection avec~$X$ d'une
droite assez générale de~$\P^k$.).

La démonstration du théorème du point fixe de Picard
fait intervenir les itérés $T^n(g_D)$ de~$g_D$
sous l'opérateur~$T$ et montre leur convergence
vers la fonction de Green~$\hat g_D$,
unique solution de $T(\hat g_D)=\hat g_D$.
Notons que l'on a 
\[ \ddc  T(g_D)  +\delta_D = \frac1d \ddc (g_D\circ f) +\delta_D + \frac 1d \delta_{\div(\alpha)}  = \frac1d f^*\omega _D .\]
Autrement dit, la forme associée à~$T(g_D)$ est encore de Kähler.
Un argument standard de courants positifs de masse bornée
entraîne que $\ddc \hat g_D+\delta_D$ est un courant positif
fermé sur~$X$, qu'on note $\hat\omega_D$.
C'est la limite des formes positives de type~$(1,1)$,
$\frac1{d^n} (f^n)^*\omega_D$.
En outre, la suite des mesures $(d^{-n} (f^n)^*\omega_D^k)_n$
converge vers une mesure, notée $(\hat\omega_D)^k$ sur
$X(\C)$. 
Cette dernière écriture $(\hat\omega_D)^k$ 
est rendue licite par la théorie des produits de courants
positifs fermés localement donnés par le $\ddc $ d'une fonction psh
continue, telle que décrite dans~\cite{demailly93}.

On appelle \emph{mesure canonique} associée
au système dynamique polarisé~$(X,f,\mathscr L)$ la
mesure de probabilité $(\hat\omega_D)^k/(D)^k$.
Dans le cas des variétés abéliennes, cette mesure canonique est la
mesure de Haar normalisée.
Dans le cas d'un système dynamique torique, par
exemple~$\P^k$ muni de l'endomorphisme donné
par l'élévation des coordonnées homogènes à une même
puissance~$d\geq 2$, il s'agit
de même de la mesure de Haar normalisée du sous-groupe
compact maximal~$(\mathbf S^1)^k$ de~$(\C^*)^k$.

Pour plus de détails et des exemples
plus intéressants du point de vue ergodique,
je renvoie aux articles de~\textsc{Cantat}
et~\textsc{Guedj} dans ce volume.

Une propriété topologique de ces mesures aura des conséquences
importantes plus bas : elles ne chargent pas les sous-ensembles
algébriques stricts
(et plus généralement les sous-ensembles pluripolaires),
\emph{cf.} l'article de~\textsc{Guedj}, théorème~3.1. En particulier,
leur support est dense pour la topologie  de Zariski.
Ce dernier fait est facile à constater dans les deux exemples
(abéliens et toriques) explicités ci-dessus.

\subsection{Extensions}

La construction de hauteurs normalisées peut être
étendue de plusieurs manières.

\medskip

1) Par additivité des hauteurs relativement à des fibrés
en droites, on peut définir une hauteur relativement
à un élément de $\Pic(X)\otimes_\Z\R$. L'intérêt est
que cet ensemble peut contenir de nouvelles classes vérifiant
des propriétés permettant l'application des idées
évoquées ci-dessus.

C'est ainsi que sur certaines surfaces K3, \cite{silverman1991},
\cite{call-s93},
et, à leur suite, \cite{billard1997},
construisent une hauteur normalisée pour l'action d'automorphismes.
Leur construction a été généralisée dans~\cite{kawaguchi2008}
au cas des automorphismes d'une surface projective~$X$
dont le premier degré dynamique~$\lambda$ est~$>1$
(c'est-à-dire, d'après~\cite{gromov2003}\footnote{Bien
que publié en~2003, cet article fondamental date de~1977.} 
et~\cite{yomdin1987},
dont l'entropie topologique est strictement positive).
Dans tous les cas, il s'agit d'établir l'existence 
de deux diviseurs à coefficients rationnels $D_+$
et~$D_-$ vérifiant $f^*(D_+)\sim \lambda D_+$
et $f^*(D_-)\sim \lambda^{-1} D_-$. 
Suivant~\cite{cantat2001},
cela se démontre en considérant l'action de~$f$
sur les espaces vectoriels $H^{1,1}(X,\R)$ et $\Pic(X)_\R$, en
notant que cette action préserve $\Pic(X)$,
la forme d'intersection 
et le cône des diviseurs nef (l'adhérence du cône ample).
On démontre (voir~\cite{mcmullen2002}) que les
valeurs propres sont des nombres de Salem:
$\lambda$ est un nombre algébrique, $\lambda^{-1}$
en est un conjugué et tous les autres conjugués
sont de module~$1$.

Ils en déduisent des théorèmes de finitude pour les points
prépériodiques de degré donné sur~$X$.
Dans le cas général considéré par~\cite{kawaguchi2008},
il s'agit seulement d'une hauteur relative à un fibré en droites
\emph{big} et~\emph{nef};
il apparaît ainsi une sous-variété exceptionnelle,
les points prépériodiques de laquelle on ne peut rien dire.
\footnote{\'Etant donné un endomorphisme~$f$ de~$\P^2$ de degré~$d\geq 2$,
éclatons un point fixe de~$f$ en lequel la différentielle de~$f$ 
est l'identité.
On obtient une surface~$X'$ sur laquelle~$f$ s'étend en un endomorphisme
qui laisse invariant point par point le diviseur exceptionnel.}

\medskip

2) En général, il n'y a pas de hauteur normalisée pour l'action
d'une application rationnelle, pas plus que
de théorème de finitude pour les points prépériodiques,
cf. l'exercice~\ref{ex-involution}.
Un certain nombre d'auteurs, \cite{silverman1994}, \cite{denis1995b},
\cite{marcello2003},  \cite{kawaguchi2006}, ont étudié tout particulièrement
le cas des automorphsimes polynomiaux de l'espace affine.
(Ceux-ci définissent des automorphismes birationnels
de l'espace projectif auxquels la théorie précédente
ne s'étend pas \emph{a priori}.)
Considérons un automorphisme~$f$ du plan affine~$\A^2$.
Il est donné par deux polynômes ; notons $d(f)$ le maximum
de leurs degrés. Le $n$-ième itéré $f^n$ de~$f$
possède de même  un degré et l'on définit le \emph{degré dynamique}
$\delta(f)$ de~$f$ comme la limite de $d(f^n)^{1/n}$.
Cette limite existe et est un entier naturel
qui vérifie $1\leq \delta(f)\leq d(f)$
(cf.~\cite{kawaguchi2006}, prop.~3.2).

Le cas $\delta(f)=d(f)$ correspond aux \emph{automorphismes réguliers}
pour lesquels les lieux d'indétermination dans~$\P^2$
de~$f$ et de~$f^{-1}$ ne se rencontrent pas.
Lorsque $\delta(f)\geq 2$, \cite{kawaguchi2006} montre
l'existence d'une fonction $\hat h$ sur~$\A^2(\bar\Q)$
vérifiant les propriétés suivantes, où $h$ désigne la restriction à~$\A^2$
de la hauteur sur~$\P^2$.
\begin{enumerate}
\item il existe des constantes $a>1$ et $b>0$ telles que pour tout $x\in\A^2(\bar\Q)$,
 $\frac 1a h(x)-b \leq \hat h(x) \leq a h(x)+b$ ;
\item  pour tout $x\in\A^2(\bar\Q)$, $\hat h(f(x))+\hat h(f^{-1}(x))
= \big(\delta+\frac1\delta) \hat h(x)$.
\end{enumerate}
La première propriété implique que $\A^2(\bar\Q)$ 
n'a qu'un nombre fini de points de degré borné et de \^hauteur bornée.
La seconde propriété entraîne que les points prépériodiques
sont exactement les points de \^hauteur nulle.

\medskip

3) Plus généralement, le contexte des correspondances
peut donner lieu à des variations intéressantes.

\section{Quelques conjectures}
\label{sec.conjectures}

Soit $(X,f,\mathscr L)$ un système dynamique polarisé
défini sur un corps de nombres~$K$. Soit $d\geq 2$
l'entier tel que $f^*\mathscr L\simeq\mathscr L^{\otimes d}$. 
Soit $\hat h_{\mathscr L}$ la hauteur normalisée 
sur $X(\bar\Q)$ associée à~$(X,f,\mathscr L)$.

Je décris dans ce paragraphe quelques unes des conjectures 
concernant l'arithmétique des systèmes dynamiques polynomiaux.
La plupart de ces énoncés ont fait l'objet d'une étude
approfondie depuis les années 1970, au moins dans le cas
des systèmes dynamiques abéliens.

\subsection{Rareté des points prépériodiques dans une sous-variété}

\begin{prop}\label{theo.per.dense}
L'ensemble des points de~$X$ qui sont périodiques pour~$f$ est
dense dans~$X$ pour la topologie de Zariski.
\end{prop}
\begin{proof}
Lorsque $X$ est une variété complexe lisse,
on peut en donner une démonstration 
en utilisant les résultats de théorie ergodique
démontrés dans les notes de~\textsc{Guedj} de ce volume.\footnote
{Comme nous l'avons rappelé au début de ce chapitre,
les différents degrés dynamiques d'un système dynamique polarisé 
sont dominés par le dernier,
\emph{cf.}~\cite{serre60b}.}

Soit $V$ l'adhérence de l'ensemble des points périodiques de~$f$
pour la topologie de Zariski, c'est-à-dire le plus petit
ensemble algébrique de~$X$ qui contient les points périodiques.
L'ensemble $V(\C)$ est fermé dans~$X(\C)$ et contient
les points périodiques ; il contient donc l'adhérence
de ces points pour la topologie usuelle. 
D'après le théorème~3.3 de cet article,
les points périodiques (répulsifs) de~$f$ s'équidistribuent
selon la mesure canonique~$\mu_f$ sur~$X(\C)$ associée à~$f$.
\emph{A fortiori}, $V(\C)$ contient le support de la mesure~$\mu_f$.
Comme cette mesure ne charge pas les parties fermées 
strictes pour la topologie de Zariski (\emph{loc. cit.}, théorème~3.1),
$V(\C)=X(\C)$ et $V=X$.

\cite{fakhruddin2003}, faisant usage 
de résultats fondamentaux de~\cite{hrushovski2004},
en donne une démonstration algébrique par réduction
au cas où le corps de base est la clôture algébrique
d'un corps fini. Il s'agit de montrer
que le complémentaire d'une hypersurface~$Y$ 
de~$X$ contient un point prépériodique, voire périodique.
Voici le principe de la démonstration.

Considérons un {\og modèle\fg} du système dynamique
polarisé~$(X,f,\mathscr L)$ 
et de~$Y$ sur un anneau~$A$ qui est une~$\Z$-algèbre
de type fini. Une façon de procéder est 
d'utiliser la prop.~\ref{theo.fakhruddin},
c'est-à-dire de considérer~$X$ comme une sous-variété
invariante par un système dynamique d'un espace
projectif~$\P^N$ 
et de définir~$A$ comme l'anneau engendré par les
coefficients, d'une part des polynômes $(f_0,\ldots,f_N)$ qui définissent
ce système dynamique et d'autre part 
d'un système d'équations des variétés~$X$ et~$Y$. 
Il convient d'adjoindre à cet anneau l'inverse
du résultant des polynômes homogènes~$f_0,\ldots,f_N$, pour
que ces polynômes définissent bien un endomorphisme
de l'espace projectif~$\P^N_A$ sur~$A$.
On peut supposer que~$\mathbf Y$ est la trace
d'une section hyperplane de~$\P^N_A$, quitte à agrandir~$Y$
de sorte qu'il est défini par une forme linéaire,
et à ajoindre à~$A$ l'inverse d'un coefficient non nul
de cette forme.

On obtient de la sorte un système dynamique polarisé
$(\mathbf X,\mathbf f)$ sur~$A$
ainsi qu'une section hyperplane~$\mathbf Y$ de~$\mathbf X$.

Soit $\mathfrak m$ un idéal maximal de~$A$ 
et soit $\kappa $ son corps résiduel.
D'après le théorème des zéros de \textsc{Hilbert},
$\kappa$ est un corps fini.
Notons $X_\kappa$, $f_\kappa$, $Y_\kappa$ les objets sur le corps~$\kappa$
déduits de~$\mathbf X$, $\mathbf f$ et~$\mathbf Y$
par réduction  modulo~$\mathfrak m$.
Si $x$ est un point de~$X_\kappa(\bar\kappa)$, défini sur
le corps fini $\kappa(x)$, tous les itérés de~$x$
sont aussi définis sur~$\kappa(x)$. Par conséquent, l'orbite
de~$x$ est finie et $x$ est prépériodique.
Il existe en particulier des points prépériodiques de~$X_\kappa(\bar\kappa)$
qui n'appartiennent pas à~$Y_\kappa$.
(C'est ici qu'intervient éventuellement le théorème de~\cite{hrushovski2004}:
il affirme que cet ensemble contient des points \emph{périodiques}
pourvu que le cardinal de~$\kappa$ soit choisi assez grand.)
Soit $x$ un tel point ; supposons $f^n(x)=f^{n+p}(x)$
pour $p>n\geq 0$, et soit $Z$ le sous-schéma de~$\mathbf X$
défini par la coïncidence de~$f^n$ et~$f^{n+p}$.

Nous allons démontrer que \emph{chaque composante irréductible de~$Z$
se surjecte sur~$\Spec A$.} (Intuitivement, cela signifie
que les ensembles de coïncidence se comportent suffisamment
bien dans la {\og famille de systèmes dynamiques\fg} paramétrée
par~$\Spec A$.) Admettons pour l'instant ce fait.
Il existe alors un point~$\xi$ dans~$Z$
dont l'image est le point générique de~$\Spec A$ et dont
$x$ est une spécialisation. Un tel point n'appartient pas
à~$\mathbf Y$ et définit un point prépériodique de~$X$,
voire périodique si l'on peut prendre $n=0$, d'où la proposition.

Pour démontrer le résultat admis, commençons par observer
que $Z$ est fini sur~$\Spec A$.
En effet, sur $Z$, les fibrés en droites 
$(f^{n+p})*\mathscr O(1)$ et~$(f^{n})^*\mathscr O(1)$
sont isomorphes, donc $\mathscr O(d^{n+p})\simeq \mathscr O(d^{n})$,
ce qui entraîne que le fibré en droites $\mathscr O(d^n(d^p-1))|_Z$
est trivial. Comme il est relativement très ample sur~$A$, $Z$ est affine
sur~$A$, c'est donc un schéma fini sur~$\Spec A$.
Mais $Z$, égal à l'intersection de la diagonale de~$\mathbf X\times_A\mathbf X$
et de l'image de~$\mathbf X$ par le couple~$(f^n,f^{n+p})$, 
est l'intersection de deux sous-schémas de dimension~$\dim X+\dim A$
dans un schéma de dimension $2\dim X+\dim A$.
Par conséquent, ses composantes irréductibles sont de dimension
au moins~$\dim A$ (appliquer le théorème~3, p.~110, de~\cite{serre2000}
après avoir plongé le tout dans un schéma régulier sur~$A$).
Comme le morphisme $Z\ra\Spec A$ est fini, elles sont de dimension
exactement $\dim A$. D'après le théorème
de constructibilité de~\textsc{Chevalley},
l'image d'une telle composante irréductible 
domine~$\Spec A$ ; par propreté des morphismes finis,
l'image d'une telle composante irréductible est  fermée dans~$\Spec A$,
donc égale à~$\Spec A$.
\end{proof}

Plus généralement, soit $Y$ une sous-variété de~$X$
qui est prépériodique, c'est-à-dire telle que la suite
de sous-variétés $(f^n(Y))$ n'ait qu'un nombre fini de termes
distincts. Soit $n$ et $p$ des entiers, avec $p>0$,
tels que $f^n(Y)=f^{n+p}(Y)$. Alors la sous-variété~$Y_n=f^n(Y)$
est invariante par $f^p$, et $(Y_n,f^p,\mathscr L|_{Y_n})$
est un système dynamique polarisé, de poids~$d^k$.
L'ensemble des points périodiques de~$f$ contenus dans~$Y_n$
est donc dense dans~$Y_n$ pour la topologie de Zariski.
Comme $f^n\colon Y\ra Y_n$ est fini et surjectif, l'ensemble des points
périodiques de~$f$ contenus dans~$Y$ est aussi dense
dans~$Y$ pour la topologie de Zariski.

\medskip

Il était tentant de faire la conjecture inverse,
d'autant plus que de nombreux résultats non-triviaux
concernant les variétés abéliennes et les tores la rendaient plausible.
Compte-tenu de contre-exemples récents, nous l'énonçons sous forme
interrogative.
\begin{conj}\label{conj.manin-mumford}
Soit $(X,f,\mathscr L)$ un système dynamique polarisé sur
un corps algébriquement clos de caractéristique zéro.
Soit $Y$ une sous-variété irréductible de~$X$
et soit $Y_0$ l'adhérence, pour la topologie de Zariski,
de l'ensemble des points prépériodiques de~$X$ qui appartiennent à~$Y$.
Est-il vrai que les composantes irréductibles de~$Y_0$ sont
des sous-variétés prépériodiques ?
\end{conj}
\removelastskip
En d'autres termes, si $Y$ n'est pas elle-même prépériodique,
est-il vrai que ses points prépériodiques 
sont contenus dans une réunion finie de sous-variétés
strictes de~$Y$.

Remarquons aussi qu'il s'agit d'une conjecture géométrique.
Toutefois, des arguments de spécialisation standard
permettent de supposer que ce système dynamique est défini
sur un corps de nombres.

L'ensemble des cas connus est mince, essentiellement
les systèmes toriques et abéliens, mais à chaque
fois la démonstration des théorèmes fut spectaculaire.
Nous allons décrire ces exemples importants ci-dessous,
puis nous expliquerons les contre-exemples
proposés par \cite{ghioca-tucker2009}.

\paragraph{Systèmes dynamiques abéliens}
Le cas où $X$ est un système dynamique abélien,
$f$ étant l'endomorphisme
de multiplication par un entier~$\geq 2$,
a été conjecturé par~\textsc{Manin} et~\textsc{Mumford},
semble-t-il motivés par l'ex-conjecture de~\textsc{Mordell}.
Elle a été démontrée pour la première fois par~\cite{raynaud83,raynaud83c}
(démonstration de nature arithmétique, $p$-adique).
Dans ce cas, les variétés prépériodiques sont plutôt appelées
{\og sous-variétés de torsion\fg} : il s'agit en effet
des translatées des sous-variétés abéliennes de~$X$ par un point de torsion,
\emph{cf.} par exemple~\cite{hindry1988}, lemme~10
(et, dans un cas voisin, le lemme~\ref{lemm.ssvartor}  ci-dessous).
Citons aussi  la solution donnée par~\cite{hindry1988},
suivant une stratégie de S.~\textsc{Lang}, stratégie qui requiert
des énoncés arithmétiques difficiles sur
l'image des représentations galoisiennes associée à la
variété abélienne~$X$, dus à~\cite{serre-85-86} 
et qu'il avait exposés dans un cours au Collège de France.
(Essentiellement, si $X$ est définie sur un corps
de nombres~$K$, il faut savoir que l'action du groupe
de Galois~$\Gal(\bar\Q/K)$ sur les points de $n$-torsion de~$X$
est assez grosse; \cite{hindry1988} utilise par exemple
qu'il existe~$c>0$ de sorte que lorsque $n$ tend vers l'infini, 
un point d'ordre~$n$ a au moins~$n^c$ conjugués distincts.
Voir aussi la démonstration de la proposition~\ref{prop.ihara}.)
Les démonstrations plus récentes de~\cite{pink-roessler2002,pink-roessler2004}
utilisent aussi cet ingrédient, mais pas  celle de~\cite{roessler2005}.
Ces trois derniers articles ont été inspirés 
par la démonstration de~\cite{hrushovski2001}  dans le
cas des corps de fonctions ;  la 
théorie des modèles y était un outil important.

\paragraph{Systèmes dynamiques toriques}
Le cas des systèmes dynamiques toriques est assez similaire
(certaines des références ci-dessus traitent d'ailleurs
simultanément les deux cas).
Commençons par énumérer les sous-variétés invariantes
dans le cas où $X=\P^k$ et $f([x_0:\cdots:x_k])=[x_0^d:\cdots:x_k^d]$.
Le cas général s'y ramène via la géométrie des variétés toriques.
Rappelons que le tore~$\gm^k$ 
agit sur~$X$ par 
$(u_1,\dots,u_k)\cdot [x_0:\cdots:x_k]=[x_0:u_1x_1:\cdots:u_kx_k]$.

\begin{lemm}\label{lemm.ssvartor}
Soit $V$ une sous-variété irréductible de~$\P^k$,
rencontrant~$\gm^k$,
telle qu'il existe un point $\alpha\in\gm^k$
telle que $f(V)\subset \alpha\cdot V$.
Alors $V$ est un translaté d'un sous-tore de~$\gm^k$,
par un point de torsion dans le cas particulier où~$\alpha=1$.
\end{lemm}
\begin{proof}[Démonstration, d'après~\cite{pink-roessler2002}]
Soit $G$ le stabilisateur de~$V$; c'est le sous-groupe de~$\gm^k$
formé des $u=(u_1,\ldots,u_k)\in\gm^k$ tels
que $u\cdot x\in V$ pour tout $x\in V$.

Supposons d'abord que $G=\{1\}$. Alors,
les $d^k$ sous-variétés $u\cdot V$, où $u$ parcourt les points de~$\gm^k$
tels que $u^d=\alpha^{-1}$, sont disjointes et contenues dans $f^{-1}(V)$.
Comme le degré de $f$ est~$d^k$, ces sous-variétés décrivent exactement
les composantes irréductibles de~$f^{-1}(V)$.
Le cycle $f^*(V)$ est somme des cycles~$u\cdot V$;
si $\deg$ désigne le degré d'une sous-variété de l'espace projectif,
on a alors $\deg(f^{-1}(V))=d^k\deg(V)$. D'autre part,
comme $f\colon\P^k\ra\P^k$ est de poids~$d$, le degré du cycle~$f^*(V)$
est égal à~$d^{\dim V} \deg(D)$. On a donc $\dim V=k$
et $V=\P^k$. Comme le stabilisateur de~$V$ est trivial,
$k=0$ et $V=\{1\}$ est un point de torsion.

Dans le cas général, $\gm^k/G$ est un tore $\gm^{k'}$.
Par construction,
l'image de~$(V\cap\gm^k)/G$ dans ce tore~$\gm^{k'}$
a un stabilisateur trivial et est stable par l'élévation à la puissance~$d$.
Son adhérence~$V'$ dans~$\P^{k'}$ vérifie les hypothèses
du lemme. Par suite, $V/G$ est un point, et $V$
est un translaté de~$G$, c'est-à-dire $V=\beta \cdot G$.

Alors, $f(V)=\beta^d\cdot G=\alpha\beta\cdot G$.
Si $\alpha=1$, $\beta^{d-1}\in G$. Par {\og complète
réductibilité des tores\fg}, on peut alors
écrire $\gm^k=G\cdot G'$, où $G'$ est un sous-tore;
écrivons $\beta=\gamma\gamma'$  avec $\gamma\in G$
et $\gamma'\in G'$. On a $V=\gamma' G$
et l'égalité $\beta^{d-1}=\gamma^{d-1}(\gamma')^{d-1}$
entraîne que $\gamma'$ est un point de torsion.
\end{proof}

Comme on l'a mentionné dans le cas des systèmes dynamiques
abéliens, une des approches possibles utilise
des renseignements de nature galoisienne. Ceux-ci
sont notablement pour les tores, car ils réduisent essentiellement
à l'irréductibilité des polynômes cyclotomiques.
Pour en donner une idée,
nous exposons maintenant la démonstration,
due à~\textsc{Ihara}, \textsc{Serre} et~\textsc{Tate},
du cas d'une courbe
dans~$\P^2$. Notre exposition reprend celle de~\cite{lang1983}, p.~160,
ainsi que la présentation d'\cite{hindry1988}.

\begin{prop}\label{prop.ihara}
Soit $V$ une courbe irréductible du plan projectif complexe
qui contient une infinité de points de torsion de~$\gm^2$.
Alors $V$ est l'adhérence dans~$\P^2$ d'un translaté
d'un sous-tore de~$\gm^2$ par un point de torsion de~$\gm^2$.
\end{prop}
De manière plus explicite,
$V\cap\gm^2$ possède alors une équation de la forme $X^aY^b=u$,
où $(a,b)$ est un couple d'entiers relatifs premiers entre eux
et $u$ est une racine de l'unité. 
\begin{proof}
Si $V$ a une équation~$F=0$ dans~$\P^2$ et $\sigma$ est
un automorphisme de~$\C$, on note $V^\sigma$
la courbe de~$\P^2$ d'équation $F^\sigma=0$
obtenue en appliquant $\sigma$ aux coefficients de~$F$.
Si $x=[x_0:x_1:x_2]\in V$, 
le point~$\sigma(x)=[\sigma(x_0):\sigma(x_1):\sigma(x_2)]$
appartient à~$V^\sigma$.

Soit $K_1$ le sous-corps de~$\C$ engendré par les racines de l'unité;
montrons que $V$ est définie sur~$K_1$.
Soit $F\in\C[X_0,X_1,X_2]$ une équation homogène de~$V$ 
dont un des coefficients est égal à~$1$;
montrons que tous les autres coefficients de~$F$ appartiennent à~$K_1$.

Soit $\sigma$ un automorphisme de~$\C$ qui fixe~$K_1$.
Pour tout point de torsion $u=[1:u_1:u_2]$ appartenant à~$V$,
$\sigma(u)=[1:\sigma(u_1):\sigma(u_2)]$ appartient
à $V^\sigma$; comme $\sigma(u)=u$, \mbox{$u\in V^\sigma\cap V$}.
Par hypothèse, $V$ et~$V^\sigma$ ont une infinité de points
d'intersection. D'après le théorème de Bézout,
$V^\sigma=V$. Les formes $F$ et~$F^\sigma$ sont alors proportionnelles,
donc égales car un des coefficients de~$F$ est égal à~$1$.
Par suite, $F\in K_1[X_0,X_1,X_2]$.

Il existe donc une racine de l'unité~$\alpha$
tel que les coefficients de~$F$ appartiennent à~$\Q(\alpha)$ ;
posons $K=\Q(\alpha)$ et $m=[\Q(\alpha):\Q]$.
Pour la fin de la démonstration, nous allons supposer par l'absurde
que $V$ n'est pas un translaté d'un sous-tore 
et majorer l'ordre d'un point de torsion appartenant à~$V$.

Soit $u$ un point de torsion de~$V$ et soit~$n$ son ordre.
Soit $\xi$ une racine de l'unité d'ordre~$n$; comme $[\Q(\xi):\Q]=\phi(n)$,
$\xi$ possède au moins~$\phi(n)/m$ conjugués sur~$K$,
donc $u$ possède au moins $\phi(n)/m$ conjugués sur~$K$
qui sont autant de points de torsion d'ordre~$n$ situés sur~$V$.

D'après le théorème des nombres premiers, il existe un nombre réel~$c_1>0$
(indépendant de~$n$)
et un nombre premier~$p$ ne divisant pas~$n$ tel
que $p\leq c_1\log n$. 
Soit $p$ un tel nombre premier.
Il existe un automorphisme~$\sigma_1$
de~$\C$ tel que $\sigma_1(\xi)=\xi^p$; alors $\sigma=\sigma_1^{m}$
est un automorphisme de~$\C$ tel que $\sigma(\xi)=\xi$
et $\sigma(u)=u^d$, où $d=p^m$. Par suite, $\sigma\in\Aut(\C/K)$.

Notons $V_d$ l'image réciproque de~$V$ par
l'élévation des coordonnées à la puissance~$d$;
on a $u\in V_d$ : si $F\in K[X_0,X_1,X_2]$
est une forme qui définit~$V$ et $u=[1:u_1:u_2]$,
\[ F(u^d)=F(1,u_1^d,u_2^d)=F(1,\sigma(u_1),\sigma(u_2))
=F^\sigma(1,\sigma(u_1),\sigma(u_2))=\sigma(F(u))=0, \]
donc $u\in V_d\cap V$.
Cela vaut aussi des~$\phi(n)/m$ conjugués qu'on en avait déduit.
Comme $V$ est supposé ne pas être une sous-variété
de torsion, $V$ n'est pas une composante irréductible
de~$V_d$. D'après le théorème de Bézout, on a donc
\[ \frac{\phi(n)}m \leq \Card(V_d\cap V) \leq (\deg V_d)(\deg V)
  \leq d^2 (\deg V)^2. \]
Par suite, 
\[ \phi(n)\leq md^2(\deg V)^2\leq m c_1^2 (\deg V)^2 (\log n)^{2m} = c_2(\log n)^{2m}.\]
Comme $\phi(n)\gg n^{1/2}$ lorsque $n$ tend vers l'infini,
cette inégalité implique que l'ordre d'un point de torsion situé
sur~$V$ est borné.
Il n'y a donc qu'un nombre fini de tels points de torsion.
\end{proof}

\paragraph{Contre-exemples}
Comme je l'ai dit plus haut, la conjecture~\ref{conj.manin-mumford}
n'est pas vraie en toute généralité. En voici
un contre-exemple, dû à \cite{ghioca-tucker2009}. 
Prenons pour $X$ le carré~$E\times E$ d'une courbe elliptique~$E$
possédant de la multiplication complexe. Pour fixer les idées,
supposons que $E$ soit la courbe d'invariant~1728,
décrite analytiquement comme le quotient~$\C/\Z[i]$,
ou par l'équation $y^2=x^3-x$.
L'intérêt de cette courbe est de disposer de plus d'endomorphismes
que les simples multiplications par un entier; précisément,
l'anneau de ses endomorphismes est l'anneau~$\Z[i]$ des entiers de
Gauß. (La {\og multiplication par~$i$\fg} est
donnée par $(x,y)\mapsto (-x,iy)$.) Si $a\in\Z[i]$ n'est
pas nul, l'endomorphisme~$[a]$ de~$E$ est fini
et son degré est égal à~$\abs a^2$ (la norme de~$a$).
En outre, si $\mathscr L=\mathscr O(0)$ est le fibré ample
sur~$E$ correspondant au diviseur~$0$,
$[a]^*\mathscr L$ est isomorphe à~$\mathscr L^{\abs a^2}$.
(Lorsque $\abs a^2$ est pair, 
il faut éventuellement remplacer~$\mathscr L$ par son carré;
nous négligeons ce point.)

Pour tout couple~$(a_1,a_2)$ d'entiers de Gauß, non nuls,
un endomorphisme de la surface~$X$ est l'application~$f=([a_1],[a_2])$
qui agit comme la multiplication par~$a_1$ sur la première
composante, et par~$a_2$ sur la seconde.
Munissons~$X$ du fibré en droites ample
$\mathscr L=\mathscr O_X(0\times E+E\times 0)$.
On a $f^*\mathscr L=\mathscr O_X(\abs a_1^2 0\times E
+ \abs a_2^2 E\times 0)$. Pour que $(X,f,\mathscr L)$ 
soit un système dynamique polarisé, il suffit donc que $\abs a_1^2=\abs a_2^2$.
Prenons par exemple $a_1=3+4i$ et $a_2=5$; on a en effet $3^2+4^2=5^2$.

Avec ces choix, on a le contre-exemple suivant
à la conjecture~\ref{conj.manin-mumford}:
\begin{prop}
La diagonale de~$X$ contient une infinité
de points prépériodiques
pour le système dynamique polarisé~$(X,f,\mathscr L)$,
mais n'est pas elle-même une sous-variété prépériodique.
\end{prop}
\begin{proof}
Les systèmes dynamiques $(E,[a])$ sont très semblables
et leurs points prépériodiques sont les mêmes: les points
de torsion de~$E$. En particulier, la diagonale~$Y$
de~$X=E\times E$, qui contient une infinité de points
torsion de la surface abélienne~$X$, contient une infinité
de points prépériodiques pour~$(X,f)$. Puisque cette diagonale
est de dimension~$1$, les points prépériodiques y sont alors
denses pour la topologie de Zariski.

Si la conjecture~\ref{conj.manin-mumford} est vraie, cette diagonale~$Y$
doit être une sous-variété prépériodique.
Démontrons que ce n'est pas le cas.
Pour tout entier~$p\geq 1$,
$f^p(Y)=([a_1^p],[a_2^p])(Y)$ n'est égal à~$Y$
que si $a_1^p=a_2^p$. En effet, la différentielle
de~$[a_1]$ en l'origine est la multiplication par~$a_1$
dans l'espace tangent~$T_0E$, donc l'espace tangent à~$f^p(Y)$
est la droite de vecteur directeur~$(a_1^p,a_2^p)$
dans~$T_0X=T_0E\times T_0E\simeq\C^2$.
Plus généralement, si $n$ et~$p$ sont des entiers~$\geq 1$,
l'égalité $f^{n+p}(Y)=f^n(Y)$
impose que les vecteurs~$(a_1^{n+p},a_2^{n+p})$
et $(a_1^n,a_2^n)$ soient colinéaires,
ce qui n'arrive que si $(a_1/a_2)^p=1$.

L'affirmation résulte ainsi de ce que 
le nombre complexe~$a_1/a_2=(3+4i)/5$ n'est pas une racine
de l'unité. (Dans~$\Z[i]$, les racines de l'unité sont~$\pm 1$
et~$\pm i$; on peut aussi utiliser le fait que
le polynôme minimal de~$a_1/a_2$, égal à~$5X^2-6X+5$,
n'est pas unitaire, donc $a_1/a_2$ n'est même pas un entier algébrique.)
\end{proof}

On peut bien sûr construire d'autres exemples dans la même veine,
voir par exemple~\cite{pazuki2009} pour des produits
de surfaces abéliennes.
Ces exemples peuvent paraître artificiels : ce ne sont,
après tout, que des systèmes dynamiques produits.  Ils n'en obligent
pas moins à imaginer une nouvelle formulation 
de la conjecture~\ref{conj.manin-mumford}, d'autant plus 
qu'ils infirment aussitôt les conjectures~\ref{conj.bogomolov2}
et~\ref{conj.equidis} qui vont suivre, puisqu'elles
renforcent la conjecture~\ref{conj.manin-mumford}.

L'idée qui semble
prévaloir est que la source de contre-exemples dénichée par
\textsc{Ghioca} et~\textsc{Tucker} est effectivement la seule obstruction.
Dans un travail en cours, \textsc{Ghioca}, \textsc{Tucker}
et~\textsc{Zhang} vérifient que  c'est le cas pour
les endomorphismes de~$\P^1\times\P^1$ qui sont de la forme
$(f_1,f_2)$; s'ils violent la conjecture~\ref{conj.manin-mumford},
ce sont des endomorphismes de Lattès, quotients
d'un exemple comme celui que nous avons expliqué.

\subsection{Minoration de la hauteur d'un point qui n'est pas prépériodique}

Les points prépériodiques sont de hauteur nulle, et inversement.
Mais que peut valoir la hauteur d'un point qui n'est pas
prépériodique ? D'après le théorème de finitude de Northcott,
elle est strictement positive, et peut être arbitrairement
grande. Elle peut aussi être arbitrairement petite.
Soit $x\in X(\bar\Q)$ un point qui n'est pas prépériodique;
définissons une suite de points~$(x_n)$
en posant $x_0=x$ et où, pour tout $n$,
$x_{n+1}\in X(\bar\Q)$ est un antécédent de~$x_n$ par~$f$.
Cette suite est bien définie car l'endomorphisme~$f$
est fini et surjectif.
Pour tout~$n$, on a $\hat h_{\mathscr L}(f(x_n))
 = d h_{\mathscr L}(x_n)$, d'où
\[ \hat h_{\mathscr L}(x_n) = \frac1{d^n} \hat h_{\mathscr L}(x). \]

Cependant, en même temps que la hauteur de~$x_n$
décroît, le corps sur lequel ce point est défini
est susceptible de croître. En effet, comme $f$
est de degré~$d^{\dim X}$, il faut, pour déterminer~$x_{n+1}$,
résoudre une équation de degré~$d^{\dim X}$.
Cette équation n'est peut-être pas irréductible,
autrement dit le degré de l'extension $[K(x_{n+1}):K(x_n)]$
n'est pas forcément égal à~$d^{\dim X}$.
C'est par exemple le cas si $X$ est un produit~$X'\times X''$,
$f=(f',f'')$ et $\mathscr L=\mathscr L'\boxtimes\mathscr L''$,
et que la première coordonnée~$x'$ du point~$x$ est un point fixe
de~$f'$. On peut alors choisir $x_n$ de la forme~$(x'_0,x''_n)$ 
et le degré de l'extension~$[K(x_{n+1}):K(x_n)]$ est au plus
égal à~$d^{\dim X''}$.

On espère néanmoins qu'il est toujours au moins égal à~$d$,
du moins en moyenne.
\begin{conj}\label{conj.lehmer}
Existe-t-il un nombre réel~$c>0$ tel que pour tout point
$x\in X(\bar\Q)$ qui n'est pas prépériodique,
\[ \hat h_{\mathscr L}(x) \geq \frac c{[K(x):K]}. \]
\end{conj}

La première apparition de cette conjecture concerne 
la hauteur naturelle sur~$\P^1$ sous la forme d'un
\emph{problème} dans~\cite{lehmer1933}: Pour $\eps>0$, trouver
un polynôme unitaire $P$ à coefficients entiers dont la mesure de Mahler
vérifie $1<\exp(\deg(P)\mathrm M(P)) <1+\eps$;
il ajoute : {\og Whether or not the problem has a solution for $\eps<0{,}176$,
we do not know.\fg} Et nous ne savons toujours pas !
Le meilleur résultat connu dans ce cas est dû à~\cite{dobrowolski1979};
citons-en une version un peu affaiblie:
\begin{prop}\label{prop.w-dobrowolski}
Pour tout~$\eps>0$, 
il existe un nombre réel~$c>0$ tel que pour tout nombre algébrique~$\xi$
qui n'est ni nul ni une racine de l'unité, on ait
\[ h(\xi) \geq \frac c{[\Q(\xi):\Q]^{1+\eps}}. \]
\end{prop}

Dans le cas de l'espace projectif~$\P^k$
et de l'endomorphisme $f\colon [x_0:\cdots:x_k]\mapsto
[x_0^d:\cdots:x_k^d]$, pour lequel la hauteur naturelle
est une hauteur canonique, \cite{amoroso-david2003}
ont étendu les méthodes de~\cite{dobrowolski1979}.
Par ailleurs, \cite{david-hindry2000} ont traité
le cas des variétés abéliennes à multiplication  complexe,
généralisant un théorème de~\cite{laurent1983}
qui concernait les courbes elliptiques à multiplication complexe.
Ces résultats impliquent immédiatement des théorèmes analogues
pour leurs quotients, comme les endomorphismes
de Lattès ou de Tchébychev de~$\P^1$.

Tant~\cite{dobrowolski1979}
que les extensions  qui s'en inspirent
utilisent de manière cruciale que pour beaucoup de nombres
premiers~$p$, on peut trouver 
un endomorphisme
du système dynamique qui est {\og congru modulo~$p$\fg}
à l'application polynomiale donnée par l'élévation des coordonnées
à la puissance~$p$ (ou à un de ses itérés). 
Dans le cas du problème de \textsc{Lehmer} classique,
cet endomorphisme n'est autre que l'application 
$[x_0:x_1]\mapsto [x_0^p:x_1^p]$, cette congruence
apparaissant par le biais du petit théorème de~\textsc{Fermat}.

En revanche, cette méthode ne peut pas s'étendre à 
un système dynamique général, même sur $\P^1$.
En effet,  si $f$ et $g$ sont deux fractions rationnelles
de degrés~$\geq 2$ qui commutent, 
et telles qu'aucun itéré de~$f$ n'est égal à aucun itéré de~$g$,
alors $(\P^1,f)$ est un endomorphisme déjà traité,
c'est-à-dire l'élévation à la puissance~$d$, un endomorphisme
de Tchébychev, ou un endomorphisme de Lattès.
C'est, à la suite de premiers
travaux de~\textsc{Fatou} et~\textsc{Julia},
le théorème de~\cite{ritt1923} et \cite{eremenko1989},
cf. le~\S7 des notes de~\cite{milnor2006}.

Dans le cas d'un système dynamique abélien général, 
défini sur un corps de nombres~$K$,
\cite{masser1984}
a démontré  une inégalité plus faible, du type: si $x\in X(\bar\Q)$ 
n'est pas prépériodique,
\[ \hat h_{\mathscr L} (x) \geq \frac c{[K(x):\Q]^{\kappa}}, \]
où $c$ et~$\kappa$ sont des constantes explicites.
De telles estimations, pourtant bien plus faibles
que celle prédite par la conjecture~\ref{conj.lehmer},
ne semblent pas connues pour un système dynamique polarisé arbitraire,
même en dimension~$1$.

\subsection{Dénombrement des points prépériodiques}

Supposons encore que le système dynamique~$(X,f,\mathscr L)$
soit défini sur un corps de nombres~$K$.
Les points prépériodiques de~$X$ qui sont définis sur~$K$
sont en nombre fini, d'après le cor.~\ref{theo.northcott}.
Est-il possible de contrôler effectivement ce nombre de points?
Précisément :

\begin{conj}\label{conj.morton-silverman}
Est-il possible de borner le nombre de points $K$-rationnels
de~$X$ qui sont prépériodiques en termes uniquement
de~$[K:\Q]$, de~$\dim X$, de~$d$ et de~$c_1(\mathscr L)^{\dim X}$ ?
\end{conj}

D'après le théorème de plongement de \textsc{Fakhruddin}
(prop.~\ref{theo.fakhruddin}), il suffit en fait
de démontrer  le cas où $X$ est un espace projectif
et $\mathscr L=\mathscr O(1)$. C'est d'ailleurs dans ce cas
particulier que \cite{morton-silverman1994} avaient
énoncé cette conjecture.  

Commençons par donner quelques cas particuliers
qui l'ont motivée.

\paragraph{Systèmes dynamiques abéliens}
Le cas des systèmes dynamiques toriques
étant un exercice facile (exercice~\ref{exo.morton-silv}),
supposons que $X$ est une variété  abélienne principalement polarisée
et $f$ la multiplication par~$2$ dans~$X$.
\emph{Les points prépériodiques pour~$f$ sont alors les points
de torsion de~$X$, c'est-à-dire ceux dont un multiple
non nul est nul.} En effet, $x\in X(\bar\Q)$
est prépériodique, c'est-à-dire si $f^n(x)=f^{n+p}(x)$,
avec $n\geq 0$ et $p>0$, on a $2^n x=2^{n+p}x$,
d'où $2^n(2^p-1)x=0$ et $x$ est de torsion.
Inversement, si $x$ est annulé par un entier~$N>0$,
écrivons $N=2^nm$, où $m$ est impair. Alors, $2$ est inversible
dans l'anneau fini~$\Z/m\Z$, et il existe  un entier~$p>0$
tel que $2^p\equiv 1\pmod m$. Alors, $2^{n+p}x=2^n x$
et $x$ est prépériodique.

Dire que $X$ est principalement polarisée par~$\mathscr L$
signifie exactement que $c_1(\mathscr L)^{\dim X}=(\dim X)!$.
Si de plus $\mathscr L$ est symétrique, ce qu'on peut supposer,
alors $f^*\mathscr L\simeq\mathscr L^{\otimes 4}$,
donc $(X,f,\mathscr L)$ est de poids~$4$.
La conjecture revient donc à savoir si le nombre
de points de torsion de~$X$ qui sont $K$-rationnels
peut être majoré par une constante qui ne dépende que de~$k$
et du degré~$[K:\Q]$.

La question est déjà remarquablement difficile !
Lorsque $k=1$, c'est-à-dire pour les courbes elliptiques,
elle a été résolue par~\cite{mazur1977} lorsque $K=\Q$
et en général, après quelques résultats intermédiaires, par~\cite{merel96} ; 
pour $k\geq 2$, elle est encore ouverte.

De manière analogue aux 
résultats partiels de~\cite{flexor-o1990}
en direction  du théorème de~\cite{merel96},
il y a des résultats partiels, citons notamment
\cite{morton-silverman1994,call-goldstine1997,benedetto2007},
où la constante~$c$ obtenue
dépend du corps~$K$, et non seulement de son degré~$[K:\Q]$.
Leur démonstration fait intervenir des rudiments
de dynamique $p$-adique.

\subsection{Discrétion des points d'une sous-variété qui ne sont pas prépériodiques}

La conjecture suivante fait intervenir la hauteur normalisée.
On suppose donc que le système dynamique polarisé $(X,f,\mathscr L)$
est défini sur un corps de nombres.

Si $Y$ est une sous-variété prépériodique de~$X$, on a vu
que $Y$ contient des points prépériodiques, et même qu'ils
sont denses pour la topologie de Zariski. La hauteur
normalisée prend donc la valeur~$0$ sur~$Y(\bar\Q)$.
Même si l'on exclut les points prépériodiques, elle
prend des valeurs arbitrairement petites car on
peut toujours considérer la pré-orbite
d'un point, qui est formée de points de hauteurs
tendant vers~$0$. Une question naturelle,
posée par~\cite{bogomolov80b} lorsque $X$ est la jacobienne
d'une courbe de genre au moins~$2$ et $Y$ cette courbe,
plongée de façon standard, consiste à se demander
si ce phénomène est la seule raison pour laquelle
la hauteur normalisée peut prendre des valeurs
arbitrairement petites sur~$Y$.

%

\begin{conj}\label{conj.bogomolov2}
Soit $Y$ une sous-variété de~$X$.
L'ensemble des sous-variétés prépériodiques de~$(X,f)$
contenues dans~$Y$ n'a qu'un nombre fini d'éléments maximaux;
notons $Y^*$ leur complémentaire dans~$Y$.
De plus, $\inf\limits_{x\in Y^*(\bar\Q)} \hat h_{\mathscr L}(x)>0$.
\end{conj}

%
 
Cette conjecture renforce la conjecture~\ref{conj.manin-mumford};
les contre-exemples à cette dernière entraînent
qu'elle n'est pas vraie en général. Il est à espérer
que l'on parviendra à la reformuler convenablement.
Remarquablement, les cas particuliers dans lesquels
l'une ou l'autre des conjectures~\ref{conj.manin-mumford}
et~\ref{conj.bogomolov2} sont prouvées
sont précisément les mêmes.

Le cas des variétés toriques fut
démontré d'abord par~\cite{zhang95} par des techniques
de géométrie d'Arakelov combinées au résultat d'\textsc{Ihara},
\textsc{Serre} et~\textsc{Tate} expliqué plus haut.
Il fut reprouvé par~\cite{bilu97} comme corollaire d'un
théorème d'équidistribution. Nous expliquerons sa
démonstration au~\S\ref{sec.bilu} du chapitre~\ref{chap.equip1}.

Le cas d'un système dynamique abélien est un théorème
de~\cite{zhang98}, après que \cite{ullmo98} eut
traité la conjecture de~\cite{bogomolov80b},
c'est-à-dire le cas où $Y$ est une courbe algébrique plongée
naturellement dans sa jacobienne~$X$. 
La preuve utilise un théorème d'équidistribution 
tel que le théorème~\ref{theo.equidis}
ainsi qu'un joli argument de théorie de la mesure,
voir le paragraphe~\ref{sec.equidis}/\ref{subsec.ullmo-zhang}.
Un peu avant, \cite{philippon95} avait traité le cas 
d'une sous-variété d'un produit de courbes elliptiques.

La conjecture~\ref{conj.bogomolov2}
est aussi vraie pour les produits de tels systèmes
dynamiques d'après~\cite{chambert-loir2000b}
({\og variétés semi-abéliennes isotriviales\fg}).
Voir aussi~\cite{david-p98,david-philippon1999,david-p2000,david-p2002,amoroso-david2003,amoroso-david2004,amoroso-david2006}
pour une
nouvelle démonstration des cas précédents, ainsi que
celui des variétés semi-abéliennes générales qui ne rentre
pas tout à fait dans le contexte des systèmes dynamiques polarisés.
Notons que ces derniers auteurs démontrent un énoncé \emph{effectif}:
pourvu que $Y$ ne soit pas translaté d'une sous-variété semi-abélienne,
ils fournissent une valeur explicite pour~$c$, 
ne dépendant que de la géométrie de~$X$ et~$Y$.

Elle est aussi valable pour les quotients de tels systèmes
dynamiques (Lattès et autres).
Enfin, \cite{mimar1997} a traité quelques cas de systèmes
dynamiques à variables séparées dans~$\P^1\times\P^1$,
c'est-à-dire de la forme $(x,y)\mapsto (f(x),g(y))$,
où $f$ et~$g$ sont deux fractions rationnelles de même degré.

\subsection{\'Equidistribution vers la mesure canonique des suites de points 
dont la hauteur normalisée tend vers~$0$}

Soit $K$ un corps de nombres sur lequel le 
système dynamique polarisé $(X,f,\mathscr L)$
est défini.
Tout point $x\in X(\bar\Q)$ définit
une mesure de probabilité discrète sur~$X(\C)$,
notée~$\delta_x$ : c'est la moyenne des masses
de Dirac aux $[K(x):K]$ conjugués de~$x$.

\begin{conj}\label{conj.equidis}
Soit $(x_n)$ une suite de points de $X(\bar\Q)$
vérifiant les deux hypothèses suivantes:
\begin{enumerate}\def\theenumi{\roman{enumi}}\def\labelenumi{(\theenumi)}
\item pour toute sous-variété prépériodique $Y\subsetneq X$,
il n'existe qu'un nombre fini d'entiers~$n$ tels que $x_n\in Y$ ;
\item la suite $(\hat h_{\mathscr L}(x_n))$ des hauteurs normalisées
des~$x_n$ tend vers~$0$.
\end{enumerate}
Est-il vrai que la suite $(\delta_{x_n})$ converge
vers la mesure canonique $\hat\mu_{\mathscr L}$ sur~$X(\C)$ ?
\end{conj}

Faisons quelques commentaires au travers de deux exemples.

\begin{enumerate}
\item
Supposons par exemple que $X=\P^1$ et que $f(x)=x^2$;
alors les points prépériodiques sont $0$, $\infty$ et les racines
de l'unité,
la hauteur normalisée~$\hat h_{\mathscr L}$ n'est autre que la hauteur usuelle
et la mesure canonique~$\hat\mu_{\mathscr L}$  est la mesure d'intégration sur
le cercle unité.
La condition $\hat h_{\mathscr L}(x_n)\ra 0$ est en particulier vérifiée
si l'on prend pour $x_n$ une racine de l'unité, toutes distinctes
entre elles.
Si $x_n$ est d'ordre~$d_n$, c'est une racine du polynôme
cyclotomique $\Phi_{d_n}$, dont on sait qu'il est irréductible
de degré~$\phi(d_n)$ (nombre d'entiers entre~$1$ et~$d_n$
qui sont premiers à~$d_n$). Les conjugués  de~$x_n$
sont ainsi les autres racines primitives $d_n$-ièmes de l'unité.
Si par exemple $d_n$ est un nombre premier, les conjugués 
de~$x_n$ sont les $\exp(2\mathrm ik\pi/d_n)$, pour $1\leq k\leq d_n-1$.
Lorsque $d_n$ tend vers l'infini (toujours en étant un nombre premier)
la mesure~$\delta_{x_n}$ correspond à des sommes de Riemann 
sur le cercle (à l'oubli négligeable
près d'une masse de Dirac en~$1$ divisée par~$d_n$)
et converge ainsi vers la mesure de Haar normalisée
sur~$\mathrm S^1$.

Inversement, la conjecture~\ref{conj.equidis} appliquée
à cet exemple indique que les racines de l'unité
sont de grand degré, sorte de version quantitative
du théorème affirmant l'irréductibilité des polynômes
cyclotomiques.

\item
Prenant toujours des points périodiques sur $X=\P^1$, $f$ étant alors
arbitraire, cette conjecture est un analogue arithmétique
d'un énoncé d'équidistribution classique en théorie
des systèmes  dynamiques dû à~\cite{brolin1965} et~\cite{lyubich1983},
où l'on remplace la mesure~$\delta_{x_n}$
par la moyenne des masses de Dirac aux points périodiques
répulsifs d'ordre~$d_n$ tendant vers l'infini.
Comme pour les racines de l'unité, la conjecture~\ref{conj.equidis}
implique que ces points périodiques répulsifs se répartissent
en relativement peu d'orbites galoisiennes, toutes
proches de la mesure d'équilibre.

\item
En théorie des systèmes dynamiques, on peut aussi
considérer, au lieu de~$\delta_{x_n}$, le nuage des
itérés inverses $f^{-n}(x_0)$ d'un point initial~$x_0$.
De fait, si $f(x_{n+1})=x_n$ pour tout~$n$,
on a $\hat h_{\mathscr L}(x_n)=\hat h_{\mathscr L}(x_0)/d^n$,
donc la condition~(ii) de la conjecture~\ref{conj.equidis}
est satisfaite.
Là encore, cette conjecture tend à affirmer que
le nuage $f^{-n}(x_0)$ se répartit en relativement peu
d'orbites galoisiennes, toutes proches de la mesure
d'équilibre.

\item
Incidemment, il ne semble pas facile de produire
une suite $(x_n)$ de points vérifiant la condition~$\hat h_{\mathscr L}(x_n)\ra 0$ sans combiner les deux approches précédentes --- points périodiques
et préimages itérées.
\end{enumerate}

\begin{prop}
La conjecture~\ref{conj.equidis} implique la conjecture~\ref{conj.bogomolov2}.
\end{prop}
\begin{proof}
Soit $Y$ une sous-variété irréductible de~$X$  qui n'est pas prépériodique.

Commençons par montrer qu'il n'existe dans~$Y$ qu'un nombre
fini de sous-variétés prépériodiques maximales.
Dans le cas contraire, on pourrait en effet construire
une suite~$(Y_n)$ de sous-variétés irréductibles
prépériodiques  de~$X$ contenues dans~$Y$, telle que pour tout~$n$,
$Y_n$  ne soit pas contenue dans la réunion de~$Y_1,\ldots,Y_{n-1}$
et telle que toute sous-variété prépériodique de~$(X,f)$ qui est
contenue dans~$Y$ soit contenue dans l'une des~$Y_n$.

Alors $Y_n\cap (Y_1\cup\cdots\cup Y_{n-1})$ est un fermé de Zariski
strict de~$Y_n$ ; 
il existe donc un point~$y_n\in Y_n(\bar\Q)$
qui est prépériodique mais qui n'appartient pas 
à~$Y_m$ si $m<n$.  On a $\hat h_{\mathscr L}(y_n)=0$.
Supposant la conjecture~\ref{conj.equidis} vérifiée,
la suite de mesures discrète~$(\delta_{y_n})$ converge
vers la mesure canonique~$\hat\mu_{\mathscr L}$ sur~$X(\C)$.
Mais comme $y_n\in Y$ pour tout~$n$, les supports
des mesures~$(\delta_{y_n})$ sont contenus dans la partie
fermée~$Y(\C)$ de~$X(\C)$.
Par suite, le support de leur mesure limite est lui-aussi
contenu dans~$Y(\C)$. La contradiction vient de ce que la mesure~$\hat\mu_{\mathscr L}$ ne charge pas les sous-ensembles
algébriques stricts (théorème~3.1 du texte de~\textsc{Guedj},
voir aussi plus haut, 
paragraphe~\ref{sec.sysdyn}/\ref{subsec.mesures-canoniques}).

Cela démontre que la réunion~$Y_0$ des sous-variétés prépériodiques
de~$(X,f)$ qui sont contenues dans~$Y$ est une partie fermée
de~$Y$. On a $Y_0\neq Y$ puisque $Y$ n'est pas prépériodique.
Supposons que $\hat h_{\mathscr L}(x)$ prenne des valeurs
arbitrairement petites lorsque~$x\in Y(\bar\Q)$ mais $x\not\in Y_0$.
Soit $(y_n)$ une suite de points de~$Y(\bar\Q)$ n'appartenant
pas à~$Y_0$ telle que $\hat h_{\mathscr L}(y_n)\ra 0$.
Comme précédemment, la contradiction apparaît en considérant 
la suite de mesures~$(\delta(y_n))$ supportées par~$Y(\C)$:
cette suite converge en effet vers
la mesure canonique~$\hat\mu_{\mathscr L}$, donc
son support est contenu dans~$Y(\C)$ bien
qu'il soit Zariski-dense dans~$X(\C)$.
\end{proof}

\section{Un théorème d'équidistribution 
et ses applications}
\label{sec.equidis}

La géométrie d'Arakelov a  fourni une voie
d'attaque très efficace 
de la conjecture d'équidistribution (conjecture~\ref{conj.equidis}).
Le point de départ en est le théorème d'équidistribution
ci-dessous, dû aux travaux successifs de~\cite{szpiro-u-z97},
\cite{autissier2001b} puis pour finir~\cite{yuan2008}.
(Voir aussi~\cite{bilu97,rumely99,chambert-loir2000b}
pour des démonstrations alternatives lorsque $X=\P^1$.)

\begin{theo}[Yuan]\label{theo.equidis}
Soit $(x_n)$ une suite de points de $X(\bar\Q)$
vérifiant les deux hypothèses suivantes:
\begin{enumerate}\def\theenumi{\roman{enumi}}\def\labelenumi{(\theenumi)}
\item pour toute sous-variété $Y\subsetneq X$,
il n'existe qu'un nombre fini d'entiers~$n$ tels que $x_n\in Y$ ;
\item la suite $(\hat h_{\mathscr L}(x_n))$ des hauteurs normalisées
des~$x_n$ tend vers~$0$.
\end{enumerate}
Alors, la suite $(\delta_{x_n})$ converge
vers la mesure canonique $\hat\mu_{\mathscr L}$ sur~$X(\C)$.
\end{theo}

Par rapport à la conjecture~\ref{conj.equidis},
seule la première hypothèse est modifiée car le quantificateur
porte maintenant sur toutes les sous-variétés de~$Y$.
On peut quand même remarquer que la conjecture~\ref{conj.bogomolov2}
et le théorème~\ref{theo.equidis} réunis impliquent de manière évidente
la conjecture d'équidistribution~\ref{conj.equidis}. 
En effet, si $(x_n)$ est une suite vérifiant les hypothèses
de la conjecture~\ref{conj.equidis} et $Y$ une sous-variété
de~$X$ qui n'est pas prépériodique, la conjecture~\ref{conj.bogomolov2}
interdit qu'une infinité de termes de la suite~$(x_n)$
soit contenue dans~$Y$.

Tel quel, ce théorème n'a apparemment aucune conséquence
sur les conjectures de ce chapitre. 
Toutefois,
\cite{ullmo98} et \cite{zhang98},
par une très belle astuce   qui consiste à appliquer ce théorème
dans deux situations reliées et à comparer les résultats obtenus,
ont pu en déduire les conjectures~\ref{conj.bogomolov2}
et~\ref{conj.equidis} pour les systèmes dynamiques issus de variétés
abéliennes.
J'explique cela au paragraphe suivant.

Je donne ensuite quelques indications sur la preuve
du théorème~\ref{theo.equidis} 
et conclus en décrivant succinctement
sa variante non-archimédienne, due à
\cite{chambert-loir2006,baker2006,baker-rumely2006,favre-rl2006}.
Pour ces deux derniers points, le lecteur n'y trouvera
pas de démonstrations, au mieux une esquisse des idées
qui jalonnent les preuves.

En revanche,
je détaille au chapitre~\ref{chap.equip1} de ce texte la preuve,
d'après~\cite{baker2006},
du théorème~\ref{theo.equidis} lorsque $X$ est la droite projective.
J'y exposerai aussi la démonstration, due à~\cite{bilu97},
de la conjecture d'équidistribution~\ref{conj.equidis}
pour les systèmes dynamiques toriques.

\subsection{Du théorème d'équidistribution à la conjecture
de Bogomolov}\label{subsec.ullmo-zhang}
Présentons brièvement dans un cas simple comment
\cite{ullmo98} et~\cite{zhang98} déduisent
du théorème d'équidistribution~\ref{theo.equidis} (ou plutôt
sa version pour les sous-variétés de variétés abéliennes)
une preuve de la conjecture~\ref{conj.bogomolov2}
dans le cas des variétés abéliennes.

Soit $X$ une variété abélienne sur un corps de nombres,
et $h$ une hauteur de Néron-Tate (voir p.~\pageref{hNT})
sur~$X$ relativement
à un fibré en droites ample et symétrique~$\mathscr L$,
autrement dit une hauteur normalisée pour l'endomorphisme
de multiplication par~$2$.
Cette hauteur est attachée à une métrique hermitienne sur~$\mathscr L$
et la forme de courbure~$\omega$ qui lui est associée
est une forme de Kähler.

Soit $Y$ une sous-variété de~$X$ qui n'est pas une translatée
de sous-variété abélienne par un point de torsion.
Il s'agit de prouver qu'il n'existe pas dans~$Y(\bar\Q)$
de suite~$(y_n)$ telle que $h(y_n)$ tende vers~$0$
et telle que l'ensemble~$\{y_n\}$ soit dense dans~$Y$
pour la topologie de Zariski.

Raisonnons par l'absurde en supposant l'existence d'une telle suite.
Supposons aussi pour simplifier que $Y-Y=X$, c'est-à-dire
que tout point de~$X$ soit la différence de deux points de~$Y$; 
notons $d$ la dimension de~$Y$ et $k$ celle de~$X$.

D'après le théorème d'équidistribution,
la suite~$(\delta_{y_n})$ a  pour point adhérent
la mesure~$(\omega|Y)^d/\deg(Y)$,
où $\deg(Y)$ désigne le degré de~$Y$ relativement à~$\mathscr L$.
Quitte à considérer une sous-suite de la suite~$(y_n)$, 
on suppose que la suite~$(\delta_{y_n})$ converge vers cette mesure.

Par ailleurs, les points $y_n-y_m$, pour $m,n$ parcourant~$\N$
décrivent un sous-ensemble de~$X(\bar\Q)$ qui est dense pour
la topologie de Zariski. En outre, les propriétés
élémentaires de la hauteur de Néron-Tate entraînent
que $h(y_n-y_m)$ tend vers~$0$ lorsque $n$ et $m$ tendent
tous deux vers l'infini.  
En outre, on peut extraire de la suite double $(y_n-y_m)_{n,m}$
une suite $(y_{n_k}-y_{m_k})$ à laquelle le théorème
d'équidistribution s'applique.
Il en résulte que
la suite $(\delta_{y_n-y_m})$ a la mesure $(\omega)^k/\deg(X)$
pour point adhérent.

Si $f\colon Y\times Y\ra X$ désigne
le morphisme de différence, $(y,y')\mapsto y-y'$,
on en déduit que les
mesures 
$ (\omega^k)$ 
et 
$f_*( (\omega(y)|Y)^d\otimes (\omega(y')|Y)^d) $
sont proportionnelles.

La contradiction provient de ce que la première mesure est
une honnête mesure de Lebesgue sur~$X(\C)$,
tandis que la seconde a des singularités dues au défaut
de lissité du morphisme~$f$.
Cette mesure est en effet définie par intégration de la
mesure de Lebesgue $(\omega (y)|Y)^d\otimes (\omega(y')|Y)^d$
dans les fibres du morphisme~$f$;
or, ces fibres sont génériquement de dimension~$2d-k$,
mais certaines sont de dimension strictement supérieure,
c'est par exemple le cas de celle en l'origine~$o$ qui, 
égale à la diagonale de~$Y\times Y$, est 
de dimension~$d>2d-k$.
Lorsque $r$ tend vers~$0$,
le volume d'une petite boule de rayon~$r$ autour de~$o$ dans~$X(\C)$
est de l'ordre de~$r^{2\dim X}$;
en revanche, le volume de son image réciproque par~$f$ 
est au moins égal à $r^{2(\codim f^{-1}(o))/\mu}$, 
où $\mu$ est la multiplicité générique de la fibre $f^{-1}(o)$,
comme on le voit en évaluant la contribution d'un voisinage d'un
point général de cette fibre.
Puisque $\codim f^{-1}(o) = \dim Y-\dim f^{-1}(o)<\dim X$,
il vient \emph{a fortiori} $(\codim f^{-1}(o))/\mu <\dim X$,
et ces deux volumes ont des comportements asymptotiques différents
lorsque $r$ tend vers~$0$.
(Voir aussi le lemme~4.3 de~\cite{david-p98} pour un énoncé
dans ce sens.)

\cite{ullmo98} et \cite{zhang98} raisonnent un peu différemment
et utilisent d'autres morphismes.
Dans le cas d'\cite{ullmo98}, $Y$ est une courbe plongée dans sa jacobienne~$X$;
si $g$ est le genre de~$Y$, il considère alors le $g$-ième
produit symétrique de~$Y$, c'est-à-dire le quotient de $Y^g$
par l'action du groupe symétrique~$\mathfrak S_g$
et le morphisme $f\colon Y^{(g}\ra X$ déduit de la loi d'addition  sur~$X$.
Ce morphisme est birationnel mais n'est pas un isomorphisme,
puisque son application  tangente est de rang~$1$ en tout point de la
de la diagonale.

\cite{zhang98} se ramène d'abord au cas où le stabilisateur de~$Y$
est trivial dans~$X$ et considère alors le morphisme~$f$
de~$Y^m$ dans~$X^{m-1}$ défini par $(y_1,\ldots,y_m)\mapsto (y_2-y_1,\ldots,y_m-y_1)$ que~\cite{faltings1991} avait déjà considéré dans sa démonstration
de la conjecture de \textsc{Lang} pour les sous-variétés de variétés abéliennes.
Si $m$ est un entier assez grand, ce morphisme est birationnel;
ce n'est pas un isomorphisme puisque la diagonale de~$Y^m$
est envoyée sur l'origine de~$X^{m-1}$.

Notons $f\colon Y'\ra X'$ le morphisme ainsi défini.
D'après le théorème d'équidistribution
appliqué deux fois, à la source et au but de~$f$,
à des suites bien choisies construites à partir de la
suite initiale $(y_n)$,
on en déduit l'égalité de deux mesures sur~$X'(\C)$. 
Sur un ouvert dense de~$X'(\C)$  au-dessus duquel $f$ est un isomorphisme,
ces mesures proviennent de formes différentielles. On en déduit 
une égalité de formes différentielles de degré maximal
sur un ouvert dense de~$Y'(\C)$, 
donc partout.
Cependant,  l'une de ces formes ne s'annule pas,
tandis que l'autre, image inverse d'une forme différentielle
sur~$X$ par~$f$ s'annule là où le jacobien du morphisme non lisse~$f$ 
s'annule.

Pour plus de détails,
je renvoie aux articles originaux, ainsi qu'à l'exposé d'\cite{abbes1997}
au Séminaire Bourbaki.

Cette astuce de considérer un morphisme dont
les fibres n'ont pas toutes la même dimension
a été reprise par~\cite{david-p98}
qui ont donné une nouvelle démonstration
de la conjecture~\ref{conj.bogomolov2}
pour les variétés abéliennes
(et, par la suite, dans plusieurs autres cas,
voir \cite{david-philippon1999,david-p2000,david-p2002}).

En revanche, jusqu'à présent, personne n'a réussi à en tirer parti 
hors des situations liées aux groupes algébriques,
à plus forte raison pour un système dynamique polarisé
général.

\subsection{Vers le théorème d'équidistribution}
Dans ce paragraphe, je voudrais donner quelques indications
de la méthode de démonstration du théorème~\ref{theo.equidis}.
La théorie dans laquelle s'écrit cette preuve
est la \emph{géométrie d'Arakelov,} une théorie
de l'intersection pour les variétés algébriques sur les 
anneaux d'entiers de corps de nombres,
considérées comme analogues des variétés algébriques
fibrées au-dessus d'une courbe projective.
Inventée dans le cas des surfaces 
arithmétiques par~\cite{arakelov1974} et~\cite{faltings1984},
elle a été développée de manière systématique 
par H.~\textsc{Gillet} et C.~\textsc{Soulé} dans une série
d'articles.
Citons notamment \cite{gillet-s90,gillet-s90b,gillet-s92},
la monographie~\cite{soule-a-b-k92}, ainsi que
les notes de~\cite{faltings1992}.

De même que la formule du produit (prop.~\ref{prop.formule.produit})
fait intervenir les valeurs absolues archimédiennes,
la géométrie d'Arakelov ajoute de manière systématique
aux objets issus de la géométrie algébrique
sur des anneaux d'entiers de corps de nombres
des données de nature analytique.
Par exemple, les fibrés en droites sont systématiquement
munis de métriques hermitiennes.

Ainsi que l'a montré~\cite{faltings1991},
cette théorie fournit un cadre conceptuel efficace
pour étudier la théorie des hauteurs relatives
à un fibré en droites hermitien~$\overline{\mathscr L}$:  
de même que toute sous-variété d'une variété projective
dispose d'un degré relativement à un fibré en droites.
toute  variété~$X$ dispose alors d'une hauteur $h_{\overline{\mathscr L}}(X)$.
Ceci est développé en grand détail dans l'article de~\cite{bost-g-s94}.

En outre, an ajoutant à ce formalisme  de fibrés en droites hermitiens
des objets construits via un processus limite
dans l'esprit de celui qui définit les hauteurs canoniques,
\cite{zhang95b} a montré que l'on 
peut  y étudier la hauteur normalisée définie par un système
dynamique polarisé $(X,f,\mathscr L)$ et l'étendre en une notion
de hauteur normalisée pour les sous-variétés.
Notant $d$ le \emph{poids} de ce système dynamique, l'équation
$\hat h_{\mathscr L}(f(x))=d \hat h_{\mathscr L}(x)$ 
vérifiée par tout point $x\in X(\bar\Q)$
devient $\hat h_{\mathscr L}(f_*(Z))=d^{p+1} \hat h_{\mathscr L}(Z)$, 
si $Z$ est une sous-variété irréductible de~$X$ de dimension~$p$.
Dans cette formule, $f_*(Z)$ désigne l'image de~$Z$ au sens des cycles;
c'est un  multiple de~$f(Z)$. Par exemple, on a $f_*(X)=d^p X$,
d'où l'on déduit que $\hat h_{\mathscr L}(X)=0$.
Plus généralement, les sous-variétés prépériodiques 
(\emph{i.e.}, les sous-variétés~$Y$ telles qu'il existe des entiers~$n$ et~$m>0$
tels que $f^n(Y)=f^{n+m}(Y)$) 
sont de hauteur normalisée nulle. 

La question de la réciproque a été posée par \cite{philippon91} pour
des variétés abéliennes. 
Il  avait introduit dans cet article
une notion de hauteur pour les sous-variétés et, dans le contexte
des variétés abéliennes, démontré l'existence d'une hauteur normalisée;
ses arguments sont cependant généraux.
Étendue aux systèmes dynamiques polarisés, la conjecture 
qu'il y proposait est la suivante,
d'ailleurs équivalente à la conjecture~\ref{conj.bogomolov2}
d'après~\cite{zhang95b} (voir aussi~\citet{szpiro90}).
\begin{conj}\label{conj.philippon}
Soit $(X,f,\mathscr L)$ un système dynamique polarisé défini
sur un corps de nombres. 
Si $Y$ est une sous-variété de~$X$ telle que $\hat h_{\mathscr L}(Y)=0$,
alors $Y$ est prépériodique.
\end{conj}

Un des slogans de la géométrie d'Arakelov consiste à 
rechercher un analogue {\og arithmétique\fg}
à chaque théorème de géométrie algébrique.
La recette peut sembler assez banale: écrire le théorème considéré
dans le cas d'une fibration au-dessus d'une courbe projective
et comprendre quels termes analytiques doivent être
adjoints à la situation arithmétique analogue.
Toutefois, la réalisation en est presque toujours très ardue!

En géométrie algébrique, le théorème de \textsc{Hilbert}--\textsc{Samuel} peut s'énoncer
comme suit: soit $M$ une sous-variété de dimension~$k$
et de degré~$D$ de l'espace projectif~$\P^N$; lorsque
$n$ tend vers l'infini, le nombre maximum de polynômes
homogènes de degré~$n$ linéairement indépendants sur~$M$
est équivalent à $D n^{k}/k!$.
De manière équivalente, si $M$ est une variété de dimension~$k$
et $\mathscr L$ un fibré en droites ample sur~$M$,
la dimension de l'espace
$H^0(M,\mathscr L^{\otimes n})$ est équivalente 
à $(c_1(\mathscr L)^k|M) n^{k}/k!$ lorsque $n$ tend vers l'infini.
Une façon (peu économique)
de le démontrer consiste à déduire du théorème de \textsc{Riemann}--\textsc{Roch}
(sous la forme due à \textsc{Hirzebruch}--\textsc{Grothendieck})
le résultat correspondant pour la caractéristique d'Euler
\[ \chi(M,\mathscr L^{\otimes n}) \mathrel{:=} \sum_{i=0}^k (-1)^k h^i(M,\mathscr L^{\otimes n}) \sim \frac {n^k}{k!} (c_1(\mathscr L)^k|M) \]
puis à invoquer  le théorème d'annulation de Serre
selon lequel $h^i(M,\mathscr L^{\otimes n})=0$ pour $i>0$ et $n$
assez grand.

L'analogue en géométrie d'Arakelov de ce théorème de \textsc{Hilbert}--\textsc{Samuel}
a été démontré par~\cite{gillet-s88,gillet-s92} au moyen
de leur {\og théorème de Riemann-Roch arithmétique\fg}
et d'estimées analytiques dues à~\cite{bismut-vasserot1989}.
(Voir aussi la démonstration, 
directe et plus élémentaire, de~\cite{abbes-b95},
ainsi que l'extension de~\cite{zhang95b}.)
Selon ce théorème, un fibré en droites hermitien~$\overline{\mathscr L}$
{\og semi-positif\fg} sur une variété arithmétique~$X$
possède, quitte à le remplacer par une de ses puissances,
une section globale~$s$ de norme sup.\ contrôlée
par la hauteur $h_{\overline{\mathscr L}}(X)$.

On en déduit l'\emph{inégalité fondamentale:} si 
le fibré en droites hermitien $\overline{\mathscr L}$ 
est {\og semi-positif\fg},
alors  pour toute suite~$(x_n)$ 
de points de~$X(\bar\Q)$ telle qu'aucune sous-variété stricte de~$X$
n'en contienne une sous-suite, on a l'inégalité
\[ \liminf h_{\overline{\mathscr L}}(x_n) \geq \frac{h_{\overline{\mathscr L}}(X)}{(1+\dim X)\deg_{\mathscr L}(X)}.\]

Sans donner la définition précisément,
disons quelques mots de l'hypothèse {\og semi-positif\fg}.
C'est un analogue  en géométrie d'Arakelov 
d'un fibré en droites sur l'espace total d'une fibration au-dessus
d'une courbe pour lequel toute courbe contenue dans une fibre aurait
un degré positif ou nul.
Dans cette analogie, la partie analytique qui concerne
les métriques hermitiennes stipule que le courant de courbure
de la métrique hermitienne est un courant positif. 
Signalons aussi qu'elle est vérifiée
dans le cas qui nous concerne principalement dans ce texte,
à savoir les hauteur normalisées associées 
à un système dynamique polarisé.

Dans le cas d'un système dynamique polarisé,
c'est précisément au travers de cette inégalité 
que \cite{zhang95b} établit une relation 
entre les conjectures~\ref{conj.bogomolov2}
et~\ref{conj.philippon}. En effet, appliquée, non pas à~$X$,
mais à une sous-variété~$Y$, elle entraîne que si $\hat h_{\mathscr L}(Y)>0$,
alors l'ensemble des points de $Y(\bar\Q)$ de hauteur normalisée assez petite
n'est pas dense pour la topologie de Zariski.
L'autre implication résulte de l'analogue arithmétique du théorème 
de \textsc{Nakai}--\textsc{Moishezon} que démontre~\cite{zhang95b} 
et qui implique une inégalité dans l'autre sens.

Considérons l'inégalité fondamentale dans le cas d'un système dynamique 
polarisé $(X,f,\mathscr L)$. Comme on l'a vu, le membre de droite est nul.
\cite{szpiro-u-z97} déduisent alors de cette inégalité
un \emph{principe variationnel.} Si $(x_n)$ une suite
pour laquelle on a égalité, ils écrivent l'inégalité
analogue où le fibré~$\overline{\mathscr L}$ est remplacé
par le fibré $\overline{\mathscr L}(\eps u)$ 
obtenu en multipliant la métrique hermitienne
par $\exp(-\eps u)$, où $u$ est une fonction $\mathscr C^\infty$
sur $X(\C)$ et $\eps$ un petit paramètre. 
Si $\overline{\mathscr L}(\eps u)$ est encore {\og semi-positif\fg},
on trouve
\[  \liminf \left(h_{\overline{\mathscr L}(\eps u)}(x_n) \right)
  \geq \frac{h_{\overline{\mathscr L}(\eps u)}(X)}
            {(1+\dim X)\deg_{\mathscr L}(X)}. \]
Revenant à la définition des hauteurs, on en tire
l'inégalité 
\[ \liminf \left( h_{\overline{\mathscr L}(x_n)} + \eps \delta_{x_n}(u)\right)
        \geq \frac{h_{\overline{\mathscr L}}(X)}{(1+\dim X)\deg_{\mathscr L}(X)}
   +  \eps \mu_{\overline{\mathscr L}}(u) + \mathrm O(\eps^2). \]
Puisque, par hypothèse, $h_{\overline{\mathscr L}}(x_n)$
tend vers le premier terme du second membre, il vient 
\[ \liminf_n \left(\eps \delta_{x_n}(u) \right) \geq  \eps \mu_{\overline{\mathscr L}}(u) + \mathrm O(\eps^2). \]
Divisant cette relation par $\eps$ et
faisant tendre $\eps$ vers~$0$, par valeurs supérieures ou inférieures,
on obtient que $\delta_{x_n}(u)$ tend vers $\mu_{\overline{\mathscr L}}(u)$,
c'est-à-dire le théorème d'équidistribution (théorème~\ref{theo.equidis})
puisque $h_{\overline{\mathscr L}}=\hat h_{\mathscr L}$.

Dans l'argument précédent,  nous avons supposé que les petites perturbations 
$\overline{\mathscr L}(\eps u)$ sont effectivement
justiciables de l'inégalité fondamentale.
Cela couvre en particulier le cas des systèmes dynamiques abéliens
et constitue le résultat principal de~\cite{szpiro-u-z97}.
En effet, dans ce cas,
le courant de courbure $c_1(\overline{\mathscr L})$
évoqué plus haut est en fait une forme hermitienne définie positive 
sur le fibré tangent de la variété abélienne complexe~$X(\C)$,
et toute modification assez petite de cette forme est encore définie positive.

Le cas général est dû à~\cite{yuan2008} qui démontre
un analogue  en géométrie d'Arakelov d'un critère de~\cite{siu1993}
pour que la différence de deux fibrés en droites amples 
soit \emph{gros}, et en particulier qu'une de ses puissances 
possède une section globale non nulle.
De même, \textsc{Yuan} prouve que les fibrés en droites
hermitiens requis pour mettre en \oe uvre le principe variationnel
possèdent des sections globales de norme~sup.\ contrôlée,
et en déduit qu'on peut leur appliquer l'inégalité fondamentale.

\subsection{Analogues non archimédiens}
Le point de vue sur les hauteurs que nous avons décrit
met toutes les valeurs absolues d'un corps de nombres
sur le même plan, qu'elles soient archimédiennes
ou non.
La question d'un analogue non archimédien des théorèmes
d'équidistribution se pose alors naturellement:
fixons un nombre premier~$p$ et,
pour $x\in X(\bar\Q)$, considérons la mesure de probabilité~$\delta_x$
sur l'espace topologique $X(\C_p)$, toujours définie par les conjugués
de~$x$; si $(x_n)$ est une suite de points de~$X(\bar\Q)$
vérifiant les hypothèses du théorème~\ref{theo.equidis},
que peut-on dire du comportement des
suites de mesures~$(\delta_{x_n})$?

Posé tel quel, le problème n'est en fait pas très intéressant.
En effet, dans le cas le plus simple possible où 
l'on prend pour $X$ la droite projective munie de l'application
d'élévation au carré (qui redonne la hauteur naturelle),
les suites de mesures de probabilité en question
ne convergent pas dans l'espace des mesures de probabilité.
Le point est que cet espace n'est pas compact pour la topologie
de la convergence étroite, car le corps~$\C_p$ n'est pas
localement compact.

Les espaces analytiques introduits par~\cite{berkovich1990}
permettent de remédier à ce problème. 
Par exemple, la droite analytique~$\mathrm A^1_{\C_p}$
sur~$\C_p$, complémentaire 
du point à l'infini de la droite projective~$\mathrm P^1_{\C_p}$
sur~$\C_p$,
est l'ensemble des semi-normes multiplicatives sur l'anneau~$\C_p[T]$
qui étendent la valeur absolue $p$-adique sur~$\C_p$,
munie de la {\og topologie spectrale\fg}
la moins fine pour laquelle les applications naturelles,
$\operatorname{ev}_P\colon \nu\mapsto \nu(P)$ de~$\mathrm A^1_{\C_p}$
dans~$\R$, sont continues  pour tout polynôme $P\in\C_p[T]$.
Parmi ces semi-normes, on trouve bien sûr celles de la forme
$\nu_t\colon P\mapsto \abs{P(t)}_p$, où $t$ est un point de~$\C_p$,
mais il en est bien d'autres, 
par exemple la \emph{norme de Gauß}~$\norm{\cdot}_\Gamma$ définie par
\[ \norm{P}_{\Gamma} = \max(\abs{a_0}_p,\ldots,\abs{a_n}_p),
 \qquad (P=a_0+a_1T+\cdots+a_n T^n). \]
Que cette norme soit multiplicative est une observation
au fondement  de toute géométrie analytique $p$-adique
et découle du lemme de Gauß selon lequel $\norm{PQ}_\Gamma=1$ si $P$ et $Q$
sont deux polynômes de~$\C_p[T]$ tels que $\norm P_\Gamma=\norm Q_\Gamma=1$.)

Dans~\cite{chambert-loir2006}, je construis pour tout
système dynamique polarisé $(X,f,\mathscr L)$ sur un corps~$p$-adique
une mesure canonique~$\mu_{\mathscr L}$ sur l'espace analytique 
associé  à~$X$ au sens de \textsc{Berkovich}.
Dans le cas du système dynamique torique sur~$\P^1$,
pour lequel $f([x:y])=[x^2:y^2]$,
cette mesure n'est autre que la mesure de Dirac supportée
par le point~$\Gamma$ associé à la norme de Gauß.
Une adaptation du principe variationnel décrit au paragraphe
précédent me permet alors de démontrer un théorème d'équidistribution
$p$-adique des points de petite hauteur lorsque $X$ est de dimension~1;
le cas général est dû à~\cite{yuan2008}. Dans les deux cas,
le théorème qui est démontré est un énoncé général d'équidistribution
en géométrie d'Arakelov qui dépasse  donc
le strict cas des systèmes dynamiques polarisés.
Par une méthode de théorie du potentiel
inspirée de~\cite{bilu97}, 
\cite{favre-rl2006} et \cite{baker-rumely2006} ont traité indépendamment le
cas où $X$ est la droite projective
et \cite{baker-petsche2005} celui où $X$ est une courbe elliptique.

Ces mesures en géométrie non-archimédienne peuvent aussi
être définies lorsque le corps de base est un corps de fonctions
d'une variable. L'analogue du théorème~\ref{theo.equidis}
est alors valable.
En analysant précisément les mesures obtenues,
\cite{gubler2007a} a démontré quelques cas de l'analogue 
de la conjecture~\ref{conj.bogomolov2} sur un corps de fonctions.

\medskip

Indépendamment de  ce courant de pensée de géométrie  d'Arakelov,
ces dernières années ont vu l'apparition d'une théorie
ergodique sur un corps ultramétrique 
pour laquelle je renvoie aux notes des exposés
de \textsc{Yoccoz} dans ce volume.

\section{Exercices}

\begin{exer}\label{ex-involution}\label{exo.cremona.2}
a) On définit une application rationnelle~$\sigma$ de~$\P^2$
dans lui-même en associant à un point de coordonnées
homogènes~$[x:y:z]$ le point de coordonnées homogènes~$[yz:zx:xy]$.
Cette application est définie dès qu'au moins deux des coordonnées
$x,y,z$ ne sont pas nulles, c'est-à-dire en dehors des trois
points de~$\P^2$ correspondant aux trois axes de coordonnées.
Cette application est de degré~$2$ mais pour tout 
$x\in\P^2(\bar\Q)$ où $\sigma$ est définie, on a
$\sigma\circ\sigma(x)=x$ --- c'est une involution
(\og involution de \textsc{Cremona}\fg).
En particulier, tout point de~$\P^2\setminus E$
est prépériodique.

b) Soit $u=[\ell_1:\ell_2:\ell_3]$ un automorphisme linéaire de~$\P^2$ 
et posons $\tau=u^{-1}\circ \sigma\circ u$. C'est
encore une involution rationnelle de~$\P^2$, définie
par des polynômes $[F_1:F_2:F_3]$ de degré~$2$.  
Si le groupe de Galois $\Gal(\bar\Q/\Q)$,
agissant sur les coefficients des formes linéaires~$\ell_i$,
\emph{permute} ces trois formes linéaires, alors les polynômes~$F_i$
définissant~$\tau$ sont à coefficients rationnels.

c) Le lieu d'indétermination~$E'$ de~$\tau$
est l'image réciproque par~$u$ de~$E$.
Déterminer un automorphisme~$u$ de sorte
que $E'\cap \P^2(\Q)$ soit vide.
Alors, les polynômes~$F_i$ n'ont pas de zéro commun dans~$\P^2(\Q)$
sans pour autant que l'inégalité de la prop.~\ref{prop.fonctorialite-Q}
du chapitre~1 soit valide. Pourquoi cela ne contredit-il pas
cette proposition ?
\end{exer}

\begin{exer}[\cite{kawaguchi1999}]
Soit $f\colon X\ra X$ un endomorphisme d'une variété
projective et soit $\mathscr L$ un fibré en droites ample sur~$X$.
On suppose que $f^*\mathscr L\otimes\mathscr L^{-1}$ est ample.
Alors, il existe un nombre réel~$c>0$ tel que
pour tout point $x\in X(\bar\Q)$ qui est prépériodique pour~$f$,
$h_{\mathscr L}(x)<c$.
Par conséquent, le théorème~\ref{theo.northcott}
s'étend à ce contexte (un peu) plus général.
(Montrer  qu'il existe un entier~$n\geq 1$
$(f^*\mathscr L\otimes \mathscr L^{-1})^n\otimes\mathscr L^{-1}$
est ample. En déduire qu'il existe un nombre réel~$a$
tel que $h_{\mathscr L}(f(x))\geq (1+\frac1n) h_{\mathscr L}(x)-a$
pour tout $x\in X(\bar\Q)$.)
\end{exer}

\begin{exer}[\cite{lewis1972}]
Par linéarité, on étend la notion de hauteur
aux 0-cycles (c'est-à-dire aux
combinaisons linéaires formelles de points).
Soit $f$ et~$g$ des endomorphismes d'un espace projectif~$\P^n$
tels que $\deg(f)>\deg(g)$.

a)
Montrer qu'il existe une unique hauteur~$\hat h$ relative 
au fibré~$\mathscr O(1)$
sur~$\P^n$ telle que $\hat h(f_*g^*x)=\deg(f) \hat h(x)$
pour tout $x\in \P^n(\bar\Q)$.
(Dans cette formule,
$f_*$ et $g^*$ désigne les applications déduites de~$f$ et~$g$
au niveau des~$0$-cycles, voir~\cite{fulton98}, chapitre~1.)

b) Soit $X$ une partie de~$\P^n(\bar\Q)$ stable par $\Gal(\bar \Q/K)$
telle que $g|_X$ est injective et $g(X)\subset f(X)$.
Montrer que les éléments de~$X$ sont de hauteur bornée.

c) En déduire que pour tout corps de nombres~$K$,
il n'y a pas de partie infinie $X\subset \P^n(K)$
tels que $g|_X$ soit injective et $g(X)\subset f(X)$.
\end{exer}

\begin{exer}[\cite{poonen99}]
Soit $f\colon X\ra X$ un endomorphisme d'une variété
projective et soit $\mathscr L$ un fibré en droites ample sur~$X$.
Soit $h$ une fonction hauteur pour~$\mathscr L$.
On suppose qu'il existe des nombres réels~$a>1$ et~$b$
tel que $h(f(x))\geq a h(x)-b$ pour tout~$x\in X(\bar\Q)$.

a) Montrer qu'il existe des nombres réels~$c>1$ et~$M>0$ tels que
l'on ait $h(f(x))\geq c h(x)$ pour tout~$x$ tel que $h(x)\geq M$.

b) Pour tout point $x\in X(\bar\Q)$ qui n'est pas
prépériodique, la suite $(h(f^n(x)))$ tend vers l'infini.

c) Pour $x\in X(\bar\Q)$, on note~$N(x)$ le plus petit
entier~$n$ tel que $h(f^n(x))\geq M$, s'il en existe,
ou~$+\infty$ sinon. On dit qu'une suite $(x_k)$ de points de~$X(\bar\Q)$
est \emph{de petite hauteur} si la suite $(N(x_k))$ tend vers l'infini.

Supposons qu'il existe une hauteur normalisée~$\hat h$ associée à~$f$,
c'est-à-dire une fonction~$\hat h$ sur~$X(\bar\Q)$
telle que $\hat h-h$ soit bornée et que $\hat h(f(x))=\lambda h(x)$
pour tout $x\in X(\bar\Q)$. Montrer que $\lambda >1$.
Montrer qu'une suite $(x_k)$ est de petite hauteur si et seulement
si $\hat h(x_k)\ra 0$ (d'où la terminologie!).
\end{exer}

\begin{exer}\label{exer.lehmer2}
Cet exercice propose quelques cas (plus ou moins) faciles du problème de Lehmer.
Soit $\xi$ un nombre algébrique qui n'est ni nul ni une racine de l'unité;
soit $P$ son polynôme minimal. On rappelle que la mesure
de Mahler de~$P$ est égale à $\exp(h(\xi)[\Q(\xi):\Q])$.

a) Si $\xi$ n'est pas un entier algébrique, de même que son inverse,
alors $M(P)\geq 2$.

b) Si $\xi$ est totalement réel, c'est-à-dire si toutes
les racines de~$P$ sont réelles, alors
$M(P)\geq \big(\frac{1+\sqrt 5}2\big)^{1/2}$;
\cite{schinzel1975}.  Démonstration de~\cite{hohn-skoruppa1993}:
pour $x\in\R$,
observer l'inégalité  
\[ \max(1,\abs x)\geq 
 \abs x^{1/2}\abs{x-1/x}^{1/2\sqrt 5}\big(\frac{1+\sqrt 5}2\big)^{1/2}.\]
\end{exer}

\begin{exer}\label{exo.fg=gf}
a) Soit $M$ un espace métrique complet et soit $u$, $v$ deux applications
continues contractantes de~$M$ dans~$M$. Si $u$
et~$v$ commutent, leurs points fixes respectifs coïncident.

b) Soit $X$ une variété algébrique projective, soit $f$ et~$g$
des endomorphismes de~$X$,  soit $\mathscr L$
un fibré en droites sur~$X$ et $d$, $e$ des entiers~$\geq 2$
tels que $f^*\mathscr L\simeq \mathscr L^{\otimes d}$
et $g^*\mathscr L\simeq \mathscr L^{\otimes e}$.
Si $f$ et~$g$ commutent,
les hauteurs normalisées relatives à~$\mathscr L$
associées à~$f$ et~$g$ coïncident.

Cela s'applique notamment aux systèmes dynamiques toriques et abéliens.
\end{exer}

\begin{exer}\label{exo.morton-silv}
Soit $(X,f,\mathscr L)$ un système dynamique torique et soit $G$ le tore
sous-jacent.

a) Soit $K$ un corps de nombres ; soit $n$ le cardinal
du groupes des racines de l'unité contenues dans~$K$.
Montrer que $\phi(n)\leq [K:\Q]$; en déduire
que $n$ est majoré par une constante ne dépendant
que de~$[K:\Q]$.

b) Soit $X$ une variété torique projective lisse de dimension~$k$,
soit $G$ le tore sous-jacent à~$X$ et soit $f$ l'endomorphisme
de~$X$ qui prolonge l'endomorphisme $u\mapsto u^2$ de~$G$.
Montrer que les composantes irréductibles du diviseur~$D=X\setminus G$
sont stables par~$f$ et sont encore des systèmes 
dynamiques toriques.
Majorer en terme de~$c_1(\mathscr L)^k$ le nombre
de ces composantes.

c) Montrer que la restriction
de la conjecture~\ref{conj.morton-silverman}
à la classe des systèmes dynamiques toriques est vraie.
\end{exer}

\begin{exer}
Soit $X$ une variété projective sur un corps~$K$,
munie de deux endomorphismes~$f$ et~$g$. 
On suppose que le système dynamique $(X\times X,f\times f)$
vérifie la conjecture~\ref{conj.manin-mumford}.
On fait enfin l'hypothèse qu'il existe un ensemble~$Z$ de points de~$X$,
dense dans~$X$ pour la topologie de Zariski, tel que pour tout $z\in Z$,
$z$ et $g(z)$ soient prépériodiques pour~$f$.

a) En appliquant l'énoncé de la conjecture~\ref{conj.manin-mumford}
au graphe~$\Gamma_g$ de~$g$ dans~$X\times X$, montrer qu'il
existe des entiers~$a$ et~$b$, avec $a>0$, tels que $f^b gf^a=f^{a+b}g$.

b) Soit $d$ un entier au moins égal à~$2$ et soit $\alpha$
un élément non nul de~$K$.
Démontrer qu'un polynôme~$P\in K[X_0,\cdots,X_k]$
qui vérifie une équation de la forme $P(X_0^d,\ldots,X_k^d)=\alpha P(X)^d$
n'a qu'un seul monôme.

c) On revient au contexte de l'exercice
en supposant que $X=\P^k$ et que $f$ est l'élévation
des coordonnées à une certaine puissance~$d\geq 2$,
de sorte que les points de~$X$ qui sont prépériodiques pour~$f$
sont les points à coordonnées racines de l'unité.
Montrer qu'il existe des monômes $G_0,\dots,G_k$
et des racines de l'unité~$\zeta_0,\dots,\zeta_k$,
tels que $g([x_0:\cdots:x_k])=[\zeta_0 G_0(x):\cdots:\zeta_k G_k(x)]$;
voir aussi~\cite{kawaguchi-silverman2007a}.
\end{exer}

\begin{exer}\label{exo.hNT}
Soit $G$ et $A$ des groupes abéliens et soit $f\colon G\ra A$
une application.
On suppose que $f$ vérifie la relation~\eqref{eq.hNT}
et que la multiplication par~$2$ dans~$A$ est un isomorphisme.

a)
Démontrer que l'application~$b$ de $G\times G$ dans~$A$
donnée par $b(x,y)=f(x+y)-f(x)-f(y)+f(0)$
est bilinéaire et symétrique.

b)
Soit $\ell$ l'application de~$G$ dans~$A$
définie par $\ell(x)=f(x)-\frac12 b(x,x) + f(0)$.
Démontrer que $\ell$ est une application linéaire.

c)
En déduire que $f$ est la somme d'une forme quadratique,
d'une forme linéaire et d'une constante (à valeurs dans~$A$).
\end{exer}

%% file: equip1.tex
\def\Res{\operatorname{Res}}
\def\Dir{{\text{\upshape Dir}}}

\chapter{\'Equidistribution sur la droite projective}
\label{chap.equip1}

Le but de ce chapitre est d'exposer la d\'emonstration
de la conjecture d'\'equidistribution des points
de petite hauteur pour les
syst\`emes dynamiques sur la droite projective.
Cette conjecture a \'et\'e d\'emontr\'ee par~\cite{autissier2001b}
par des techniques de g\'eom\'etrie d'Arakelov,
mais le cas de la droite projective munie de la hauteur
naturelle (un syst\`eme dynamique torique) 
est d\^u \`a~\cite{bilu97}.
La m\'ethode utilis\'ee par \textsc{Bilu} a \'et\'e \'etendue au
cas g\'en\'eral dans~\cite{baker-rumely2006}
(voir aussi \cite{baker-hsia2005} pour
les syst\`emes dynamiques polynomiaux)
et ind\'ependamment par~\cite{favre-rl2006}, ces derniers
auteurs \'etablissant en outre une borne pour
la vitesse de convergence.

La version que je donne ici est issue 
de~\cite{baker-rumely2006} avec quelques simplifications 
que permettent 
un r\'esultat de~\cite{baker2006}
et l'emploi de techniques~$L^2$
classiques en th\'eorie du potentiel.

\section{Fonctions de Green}\label{sec.green}
Dans toute cette section, on se donne un corps valu\'e~$\C_v$
qui est ou bien le corps~$\C$ des nombres complexes,
ou bien le corps~$\C_p$, pour un nombre premier~$p$.
On note~$\abs{\cdot}$ la valeur absolue de~$\C_v$.
Soit $f$ un endomorphisme de degr\'e~$d$ de la droite projective, donn\'e par
deux polyn\^omes homog\`enes $(U,V)$ en deux variables~$X$ et~$Y$
\`a coefficients dans~$\C_v$, sans z\'ero commun dans~$\C_v^2$
autre que~$(0,0)$.

La $n$-i\`eme it\'er\'ee de~$f$ est donn\'ee par deux polyn\^omes
homog\`enes $(U_n,V_n)$ de degr\'e~$d^n$.
En fait, ces polyn\^omes ob\'eissent aux relations de r\'ecurrence
\begin{align*}
 U_{n+1}(X,Y) & =U(U_n(X,Y),V_n(X,Y))
  = U_n(U(X,Y),V(X,Y)), \\
 V_{n+1}(X,Y) &
 = V(U_n(X,Y),V_n(X,Y))= V_n(U(X,Y),V(X,Y)). \end{align*}
Ils sont de degr\'e~$d^n$ et sans z\'ero commun autre que~$(0,0)$.

Soit $\infty$ le point $[1:0]$ de~$\P^1$. Le diviseur~$f^*\infty$
est d\'efini par l'\'equation~$V=0$. On a ainsi 
$f^*\infty = d\infty + \div(V/Y^d)$.
Notons $\lambda _{v}$ la fonction de Green normalis\'ee
pour le diviseur~$\infty$ et la fraction rationnelle~$V(X,Y)/Y^d$.

\subsection{Ensemble de Julia rempli}

\begin{lemm}\label{lemm.green-homogene}
Pour tout $(x,y)\in \C_v^2$, la limite suivante
\begin{equation}
\Lambda_v(x,y) = \lim_{n\ra\infty} \dfrac1{d^n}\log \max\left(\abs{U_n(x,y)},
  \abs{V_n(x,y)}\right)
\end{equation}
existe.
En fait, on a  
\[ \Lambda_v(x,y) = \lambda_v ([x:y]) +  \log \abs{y} \]
pour tout $(x,y)\in\C_v^2$ tel que $y\neq 0$.
\end{lemm}
\begin{proof}
Lorsque $(x,y)=(0,0)$, la suite qui d\'efinit~$\Lambda_v$ 
est constante de valeur~$-\infty$. Nous allons montrer
qu'elle converge en tout autre point vers un nombre r\'eel.

Comme $U$ et~$V$ sont sans z\'ero commun,
on d\'eduit du th\'eor\`eme des z\'eros de \textsc{Hilbert} l'existence d'un encadrement
\begin{equation}\label{eq.compacite}
 c_v^{-1} \max\left(\abs{\strut x},\abs{y}\right)^d \leq \max\left(\abs{U(x,y)},\abs{V(x,y)}\right) \leq c_v \max\left(\abs{\strut x},\abs{y}\right)^d,
\end{equation}
valable pour tout couple $(x,y)\in\C_v\times\C_v$.
(La m\'ethode a \'et\'e expos\'ee plusieurs fois d\'ej\`a, par exemple
dans la preuve de la proposition~\ref{prop.fonctorialite}.)
On en d\'eduit que 
\[  \dfrac1{d^n}\log \max\left(\abs{U_n(x,y)}, \abs{V_n(x,y)}\right)
-
\dfrac1{d^{n+1}}\log \max\left(\abs{U_{n+1}(x,y)}, \abs{V_{n+1}(x,y)}\right)
\]
est major\'ee uniform\'ement par un multiple de~$d^{-n}$
sur le compl\'ementaire de~$(0,0)$ dans~$\C_v\times\C_v$,
d'o\`u la convergence normale de la suite $(d^{-n}\log\max(\abs{U_n},\abs{V_n}))$
sur cet ensemble.
Sa limite est donc une fonction continue sur~$\C_v^2\setminus \{(0,0)\}$
et la fonction
\[ (x,y)\mapsto \Lambda_v(x,y)-\log\max(\abs {\strut x},\abs y) \]
est born\'ee sur~$\C_v^2\setminus\{(0,0)\}$.
En particulier, $\Lambda_v$ est
born\'ee sur tout domaine de~$\C_v^2$ de la forme
\[ a \leq \max(\abs{\strut x},\abs{y}) \leq b, \]
o\`u $a$ et~$b$ sont  des nombres r\'eels strictement positifs.

Par construction, on a $\Lambda_v(U(x,y),V(x,y))=d \Lambda_v(x,y)$.
De plus, pour  $u\in\C_v^*$ et $(x,y)\in\C_v\times\C_v$
distinct de~$(0,0)$,
on a $\Lambda_v(ux,uy) = \log\abs u_v + \Lambda_v(x,y)$.
Par suite, l'application $(x,y)\mapsto \Lambda_v(x,y)-\log\abs y$
d\'efinit, par passage au quotient,
une application~$\phi$, continue, de $\P^1(\C_v)\setminus\{\infty\}$
dans~$\R$. 
L'ensemble des~$[x:y]\in\P^1(\C_v)$ tels que $\abs {\strut\smash y}\leq \abs {\strut x}$
est un voisinage du point~$\infty=[1:0]$ ; en un tel point~$[x:y]$
avec $y\neq 0$,
\[ \phi([x:y]) = \Lambda_v(x,y)- \log\abs y
      = - \log \frac{\abs y}{\max(\abs {\strut x},\abs y)} + \left(\Lambda_v(x,y)-\log\max(\abs {\strut x},\abs y)\right) \quad.\]
Par suite, $\phi$ est une fonction de Green pour le diviseur~$\infty$.

Les relations
\begin{align*} \phi(f([x:y])) & = \phi ([U(x,y):V(x,y)]) \\
 & = \Lambda_v(U(x,y),V(x,y))-\log\abs{V(x,y)}_v \\
&  = d \Lambda_v(x,y) - \log\abs{V(x,y)}_v \\
& = d \phi([x:y]) - \log \abs{\frac{V(x,y)}{y^d}}_v
\end{align*} 
montrent qu'elle v\'erifie en outre l'\'equation fonctionnelle
qui caract\'erise la fonction de Green normalis\'ee. On a donc
$\lambda_v=\phi$.
\end{proof}

Soit $\mathscr J_v$ l'ensemble des couples $(x,y)\in\C_v^2$
tels que la suite $(U_n(x,y),V_n(x,y))$ soit born\'ee dans~$\C_v^2$.
On l'appelle l'\emph{ensemble de Julia rempli homog\`ene} de~$f$.

\begin{prop}
Pour qu'un point $(x,y)$ de~$\C_v\times\C_v$,
appartienne \`a~$\mathscr J_v$, il faut et il suffit qu'il v\'erifie
$\Lambda_v(x,y)\leq 0$.
\end{prop}
\begin{proof}
Comme $\Lambda_v(x,y)
=\lim d^{-n} \max\left(\abs{U_n(x,y)}_v,\abs{V_n(x,y)}_v\right)$,
il est clair que l'on a
$\Lambda_v(x,y)\leq 0$ pour tout $(x,y)\in \mathscr J_v$.

Inversement, si l'on pose
$m=c_v^{-1/(d-1)}$,
l'in\'egalit\'e~\eqref{eq.compacite}
entra\^{\i}ne que
pour tout~$(x,y)\in\C_v\times\C_v$ et tout nombre
r\'eel~$R>0$, on a l'implication
\[ m \max(\abs{\strut x}_v,\abs{y}_v)\geq R \quad\Rightarrow\quad
 m \max(\abs{U(x,y)}_v,\abs{V(x,y)}_v) \geq R^d . \]
Pour un couple $(x,y)\in\C_v^2$ tel que $m\max(\abs{\strut x}_v,\abs{y}_v)\geq R$
avec $R>1$, on en d\'eduit que $m\max(\abs{U_n(x,y)}_v,\abs{V_n(x,y)}_v)
\geq R^{d^n}$,
d'o\`u l'in\'egalit\'e $\Lambda_v(x,y)\geq \log R>0$.
Il en est de m\^eme s'il existe un entier~$n$ tel que 
   $m\max(\abs{U_n(x,y)}_v,\abs{V_n(x,y)}_v)>1$.
Autrement dit, si 
$\Lambda_v (x,y)\leq 0$, alors
$\max(\abs{U_n(x,y)}_v,\abs{V_n(x,y)}_v) \leq 1/m$ pour tout~$n$, ce qui d\'emontre
que la suite $(U_n(x,y),V_n(x,y))$ est born\'ee dans~$\C_v^2$.
\end{proof}

\begin{exem}
Supposons que $U=X^d$ et $V=Y^d$. Alors, $U_n=X^{d^n}$
et $V_n=Y^{d^n}$ pour tout~$n$. On a donc
$\max(\abs{U_n(x,y)},\abs{V_n(x,y)})=\max(\abs {\strut x} ,\abs y)^{d^n}$
si bien que l'ensemble de Julia rempli $\mathscr J_v$ est 
le bidisque $B_v^2$, o\`u $B_v=\{x\in\C_v\sozat \abs {\strut x}\leq 1\}$.
\end{exem}

\begin{exem}
Supposons que $\C_v$ soit ultram\'etrique
et que les conditions suivantes soient satisfaites:
\begin{itemize}
\item les coefficients de~$U$ et~$V$ sont de valeur absolue au plus~$1$.
\item leur r\'esultant $\Res(U,V)$ est une unit\'e $v$-adique.
\end{itemize}
(On dit que le syst\`eme dynamique a \emph{bonne r\'eduction}.)
Alors, 
\emph{l'ensemble de Julia rempli homog\`ene $\mathscr J_v$ est exactement l'ensemble 
des couples $(x,y)\in\C_v^2$ tels que $\max(\abs{\strut x}_v,\abs{y}_v)\leq 1$.}

En effet, dans l'in\'egalit\'e~\eqref{eq.compacite},
on peut prendre la constante~$c_v$ \'egale \`a~$1$. 
C'est clair pour l'in\'egalit\'e de droite: si $\max(\abs{\strut x},\abs y)\leq 1$,
alors $\max(\abs{U(x,y)},\abs{V(x,y)})\leq 1$.
En particulier,
si un couple $(x,y)\in\C_v^2$ v\'erifie $\abs {\strut x}_v\leq 1$ et $\abs y_v\leq 1$,
il en est de m\^eme des couples $(U_n(x,y),V_n(x,y))$ et $(x,y)\in \mathscr J_v$.

Avant de d\'emontrer l'in\'egalit\'e de gauche,
rappelons que par d\'efinition, le r\'esultant~$\Res(U,V)$ est le d\'eterminant 
de l'application lin\'eaire $(A,B)\mapsto AU+BV$
de l'espace $(\C_v[X,Y]_{d-1})^2$ dans~$\C_v[X,Y]_{2d-1}$,
dans les bases standard donn\'ees par les mon\^omes; consid\'er\'e comme
fonction polynomiale des coefficients de~$U$ et~$V$, 
il est bihomog\`ene de bidegr\'e~$(d(d-1),d(d-1))$
et s'annule si et seulement si $U$ et~$V$ ont un z\'ero commun.
La th\'eorie du r\'esultant affirme qu'il existe des polyn\^omes homog\`enes
$A$ et~$B$ de degr\'es~$d-1$ tels que $AU+BV=\Res(U,V)X^{2d-1}$,
et les coefficients de ces polyn\^omes appartiennent au sous-anneau
de~$\C_v$ engendr\'e par les coefficients de~$U$ et~$V$;
ce sont donc des \'el\'ements de valeur absolue au plus~$1$.
Par suite, sous l'hypoth\`ese que $\abs{\Res(U,V)}=1$, on a
\[ \abs{x}^{2d-1} \leq \max(\abs x,\abs y)^{d-1} \max(\abs{U(x,y)},\abs{V(x,y)}), \]
et donc
\[ \abs x \leq \max(\abs{U(x,y)},\abs{V(x,y)})^{1/d}. \]
Par sym\'etrie, $\abs y$ v\'erifie la m\^eme majoration.
Par r\'ecurrence, on obtient l'in\'egalit\'e 
\[ \max(\abs x,\abs y) \leq  \max(\abs{U_n(x,y)},\abs{V_n(x,y)})^{1/d^n}, \]
valable pour tout~$n\geq 1$.
En particulier, si la suite $(U_n(x,y),V_n(x,y))$
est born\'ee, $\max(\abs x,\abs y)\leq 1$,
ce qui conclut la preuve que $\mathscr J_v$
est l'ensemble des couples~$(x,y)\in\C_v^2$
tels que $\max(\abs {\strut x},\abs y)\leq 1$.
\end{exem}

\subsection{L'in\'egalit\'e de \textsc{Baker}}

Pour tout couple 
de points distincts,
$(P_1,P_2)\in(\P^1\times\P^1)(\C_v)$,
de coordonn\'ees homog\`enes $[x_1:y_1]$ et $[x_2:y_2]$ respectivement,
posons
\begin{equation}\label{eq.Gv}
 G_v (P_1,P_2) = 
     - \log \abs{\det \begin{pmatrix} x_1 & y_1 \\ x_2 & y_2 \end{pmatrix}}_v
 + \Lambda_v(x_1,y_1) + \Lambda_v(x_2,y_2) - \frac1{d(d-1)}
 \log\abs{\Res(U,V)}_v .
\end{equation}
On v\'erifie imm\'ediatement que cette expression ne d\'epend
pas du choix des coordonn\'ees homog\`enes d\'efinissant les points~$P_1$ et~$P_2$.
Elle ne d\'epend pas non plus du choix des polyn\^omes homog\`enes~$U$ et~$V$ 
gr\^ace \`a l'homog\'en\'eit\'e du r\'esultant.

Si les coordonn\'ees homog\`enes de~$P_1$ et~$P_2$ 
sont choisies de la forme $[z_1:1]$ et $[z_2:1]$, 
on a 
\begin{equation}
 G_v (P_1,P_2) = -\log \abs{z_1-z_2}_v + g_v(P_1) + g_v(P_2) - \frac1{d(d-1)}
 \log\abs{\Res(U,V)}_v. 
\end{equation}
On voit ainsi que $G_v$
est une fonction de Green pour le diviseur diagonal de~$\P^1\times\P^1$
--- il ne faut cependant pas la confondre avec les fonctions de Green
de diviseurs sur~$\P^1$ ! La fonction~$G_v$ jouera ici le r\^ole 
(de l'oppos\'e du logarithme) d'une distance entre points de~$\P^1$.

La proposition essentielle 
sur laquelle reposera la d\'emonstration du th\'eor\`eme
d'\'equidistribution est la minoration suivante
de moyennes de la fonction~$G_v$.
C'est un analogue d'un th\'eor\`eme d'\textsc{Elkies}
(voir~\cite{lang1988}, th\'eor\`eme~5.1, p.~150) dans le cas
des fonctions de Green normalis\'ees; l'assertion
sur la limite inf\'erieure est le pendant d'une
minoration de~\cite{faltings1984}.

\begin{prop}[\cite{baker2006}]\label{prop.baker}
Il existe un nombre r\'eel~$c_v$ tel que pour
tout entier~$n$ et toute famille $(P_1,\dots,P_n)$
de points distincts de~$\P^1(\C_v)$, on ait
\[ \frac1{n(n-1)} \sum_{\substack{i,j=1 \\i\neq j}}^n G_v(P_i,P_j)
 \geq - c_v \frac{\log n}n. \]
En particulier,
\[ \liminf_{\substack{n\ra\infty \\ \text{$(P_1,\dots,P_n)$ distincts}}} \frac1{n(n-1)} \sum_{\substack{i,j=1 \\i\neq j}}^n G_v(P_i,P_j) \geq 0. \]
\end{prop}

La d\'emonstration de cette proposition fera l'objet
du~\S\ref{sec.baker}.

\subsection{Diam\`etre transfini homog\`ene}
Nous
terminons cette section par une premi\`ere  application 
de l'in\'egalit\'e de la prop.~\ref{prop.baker}, \`a savoir
le calcul du diam\`etre transfini homog\`ene de l'ensemble
de Julia rempli homog\`ene.
Il s'agit d'un analogue du fait, classique en dynamique polynomiale
sur~$\C$, que le diam\`etre transfini,
ou la capacit\'e, de l'ensemble Julia d'un polyn\^ome
$f=a_d x^d+\dots+a_0\in\C[x]$, de degr\'e~$d\geq 2$,
est \'egale \`a~$\abs{a_d}^{-1/(d-1)}$.
(Voir~\cite{ransford95}, th\'eor\`eme 6.5.1, p.~191).
Cette formule est due \`a~\cite{demarco2003} dans le cas archim\'edien
et \`a~\cite{baker-rumely2006} dans le cas g\'en\'eral;
on trouvera dans~\cite{demarco-r2007} la g\'en\'eralisation au cas de~$n$ polyn\^omes
homog\`enes de m\^eme degr\'e en~$n$ variables.

Pour $z_1=(x_1,y_1)$ et~$z_2=(x_2,y_2)$ dans~$\C_v^2$, on pose
\[ d(z_1,z_2) = \abs{x_1y_2-x_2y_1}. \]
Le \emph{diam\`etre transfini homog\`ene} d'une partie born\'ee~$K$
de~$\C_v^2$ est d\'efini par la formule
\[ \delta^h(K) = \lim_n \sup_{z_1,\dots,z_n\in K} \prod_{i\neq j} d(z_i,z_j)^{1/n(n-1)}. \]
Par la m\^eme m\'ethode que dans le cas classique,
on v\'erifie facilement que la limite existe:   notant 
$\delta_n(K)$ le $n$-i\`eme terme de cette suite, la suite
$(\delta_n(K))$ est d\'ecroissante.
Soit en effet $z_1,\dots,z_n,z_{n+1}\in K$.
Pour tout $s\in\{1,\dots,n+1\}$, on a
\[ \prod_{\substack{1\leq i,j\leq n+1 \\ i,j \neq s}} d(z_i,z_j) 
 \leq \delta_n(K)^{n(n-1)}. \]
Faisant le produit de ces in\'egalit\'es, on obtient
\[ 
 \prod_{\substack{1\leq i\neq j\leq n+1}} d(z_i,z_j) ^{n-1} \leq
      \delta_n(K)^{n(n-1)(n+1)}, \]
soit encore
\[  \prod_{\substack{1\leq i\neq j\leq n+1}} d(z_i,z_j)^{1/n(n+1)}
 \leq \delta_n(K). \]
Par suite, $\delta_{n+1}(K)\leq \delta_n(K)$.

On s'int\'eresse dans cette section au diam\`etre transfini
homog\`ene des ensembles de Julia remplis.

\begin{exem}
Supposons que $\C_v$ soit ultram\'etrique et 
montrons que le diam\`etre transfini homog\`ene  du bidisque~$K=B_v^2$,
o\`u $B_v=\{x\in\C_v\sozat \abs {\strut x}\leq 1\}$, est \'egal \`a~$1$. 

Par l'in\'egalit\'e ultram\'etrique, $d(z,w)\leq 1$ pour $z$ et~$w$
dans~$K$. On voit donc que $\delta^h(K)\leq 1$.
Par ailleurs, si $P_1,\dots,P_n$ sont $n$ \'el\'ements
de~$\C_v\times\C_v$ de la forme $(z_i,1)$, o\`u $z_1,\dots,z_n$
sont des \'el\'ements distincts de~$\C_v$, de valeur absolue~$\leq 1$
et dont les images dans le corps r\'esiduel soient distinctes,
on a $d(P_i,P_j)=1$ pour tout couple~$(i,j)$ tel que $i\neq j$.
Puisque le corps r\'esiduel de~$\C_v$ est infini,
il existe de telles familles. Avec les notations introduites
au d\'ebut de ce paragraphe, on a donc \mbox{$\delta_n(K)=1$,}
d'o\`u \mbox{$\delta^h(K)=1$.}
\end{exem}

L'exemple qui pr\'ec\`ede calcule donc le diam\`etre transfini homog\`ene
d'un syst\`eme dynamique \`a bonne r\'eduction ;
la formule suivante concerne le cas d'un syst\`eme dynamique g\'en\'eral.

\begin{prop}\label{prop.delta-res}
On a $\delta^h(\mathscr J_v) = \abs{\Res(U,V)}^{-1/d(d-1)}$.
\end{prop}
\begin{proof}
Nous faisons la d\'emonstration sous l'hypoth\`ese suppl\'ementaire
que les coefficients de~$f$ appartiennent \`a un corps de nombres,
en renvoyant 
\`a l'exercice~\ref{exer.delta-res} pour le cas g\'en\'eral
et \`a~\cite{baker2007}, cor.~A.16,
pour une d\'emonstration purement locale.

Rappelons que la fonction de Green~$G_v$ sur~$\P^1\times\P^1$
est d\'efinie par la formule
\begin{multline*} G_v([x_1:y_1],[x_2:y_2])= - \log d((x_1,y_1),(x_2,y_2))
 + \Lambda_v (x_1,y_1) + \Lambda_v(x_2,y_2) \\
     - \frac1{d(d-1)} \log \abs{\Res(U,V)}. \end{multline*}
Par suite, si les points $P_i=(x_i,y_i)$ appartiennent \`a~$\mathscr J_v$, on a
\[ \frac1{n(n-1)} \sum_{i\neq j} G_v(P_i,P_j)
 \leq - \log \left(\prod_{i\neq j} d(P_i,P_j)\right)^{1/{n(n-1)} }
 - \frac1{d(d-1)} \log\abs{\Res(U,V)}. \]
Faisons tendre~$n$ vers l'infini et appliquons la prop.~\ref{prop.baker};
on trouve donc
\[ \log \delta^h(\mathscr J_v) \leq - \frac1{d(d-1)} \log \abs{\Res(U,V)}_v. \]

Pour \'etablir l'in\'egalit\'e dans l'autre sens, on \'ecrit les
in\'egalit\'es analogues pour \emph{toutes} les places~$w$ du corps 
de nombres~$F$ engendr\'e par les coefficients de~$f$.
Pour $w\in M_F$, notons $d_w$ la fonction analogue \`a~$d$
mais avec la valeur absolue~$\abs\cdot_w$ de~$\C_w$.
Pour tout couple~$(P_1,P_2)$ d'\'el\'ements de~$\C_v\times\C_v$,
on a, notant $P_i=(x_i,y_i)$, l'\'egalit\'e
$ d_w(P_1,P_2)=\abs{x_1y_2-x_2y_1}_v$. Si $P_1$ et $P_2$
sont des \'el\'ements distincts de~$F\times F$,
la formule du produit entra\^{\i}ne donc
\[ \sum_{w\in M_F} \eps_w \log d_w(P_1,P_2) = 0.\]
Par cons\'equent, pour tout $n$-uplet $(P_1,\dots,P_n)$ d'\'el\'ements
de~$F$,  on a
\begin{align*}
 \sum_{w\in M_F} \eps_w \log \delta_n^h(\mathscr J_w) 
& \geq  \sum_{w\in M_F} \eps_w \frac1{n(n-1)}
   \sum_{1\leq i\neq j\leq n} \log d_w(P_i,P_j)\\
& \geq \sum_{1\leq i\neq j\leq n} 
 \sum_{w\in M_F}\eps_w   \log d_w(P_i,P_j)
 \geq 0 .\end{align*}
Comme 
\[ \sum_w \eps_w\log\delta^h(\mathscr J_w) 
\leq  - \frac1{d(d-1)} \sum_w \eps_w
\log \abs{\Res(U,V)}_w = 0, \]
toujours par la formule du produit,
on a \'egalit\'e place par place, comme il fallait d\'emontrer.
%
\end{proof}

\section{D\'emonstration du th\'eor\`eme d'\'equidistribution}

Il s'agit de d\'emontrer  le th\'eor\`eme suivant, cas
particulier du th\'eor\`eme~\ref{theo.equidis}
lorsque $X=\P^1$.

\begin{theo}
Soit $f\colon\P^1\ra\P^1$ un endomorphisme de degr\'e~$d\geq 2$
d\'efini sur un corps de nombres~$F$.
Soit $\hat h$ la hauteur normalis\'ee relativement au fibr\'e
en droites $\mathscr O(1)$ sur~$\P^1$.
Si $P\in\P^1(\bar\Q)$, notons $\delta_P$ la mesure de probabilit\'e
sur~$\P^1(\C)$, moyenne des masses de Dirac aux conjugu\'es
de~$P$ sous l'action du groupe~$\Gal(\bar\Q/F)$.

Pour toute suite $(P_n)$ de points distincts de~$\P^1(\bar\Q)$
telle que $\hat h(P_n)\ra 0$, la suite de mesures
$\delta_{P_n}$ sur $\P^1(\C)$ converge faiblement
vers la mesure canonique~$\hat\mu_f$.
\end{theo}

\begin{rema}
Consid\'erons deux endomorphismes de degr\'e~$\geq 2$ de~$\P^1$,
disons $f$ et~$g$, \`a coefficients dans un corps de nombres.
Supposons qu'ils aient une infinit\'e de points pr\'ep\'eriodiques
en commun. Puisque les points pr\'ep\'eriodiques
sont ceux de hauteur normalis\'ee nulle, cette hypoth\`ese
est satisfaite lorsque les hauteurs normalis\'ees~$\hat h_f$ et~$\hat h_g$ 
co\"{\i}ncident, comme dans~\cite{kawaguchi-silverman2007a}.
On peut alors construire une suite~$(P_n)$ de points
distincts de~$\P^1(\bar\Q)$ v\'erifiant $\hat h_f(P_n)=\hat h_g(P_n)=0$.
Par suite, les mesures canoniques $\hat\mu_f$ et~$\hat\mu_g$
co\"{\i}ncident.
C'est alors un th\`eme classique en dynamique complexe
que de relier $f$ et~$g$. Hormis des cas exceptionnels,
\cite{levin-p1997} d\'emontrent par exemple qu'il
existe des entiers~$m$ et~$n$ tels que $(f^{-1}\circ f)\circ f^m=
(g^{-1}\circ g)\circ g^n$. (Dans cette \'equation,
$f^{-1}\circ f$ est \`a consid\'erer comme une branche de
la correspondance qui serait not\'ee de la m\^eme fa\c{c}on.)
\end{rema}

\subsection{Hauteur et discr\'epance}

Commen\c{c}ons par relier hauteur d'un point et les
valeurs des fonctions de Green~$G_v$. 

Soit $P\in\P^1(\bar\Q)$; notons $n$ son degr\'e sur~$F$
et $P_1,\dots,P_n$ ses conjugu\'es.
Supposons \mbox{$n\geq 2$.}
On pose alors, pour toute place $v\in M_F$,
\begin{equation}
D_v (P) = \frac1{n(n-1)} \sum_{\substack{i,j=1 \\ i\neq j}}^n G_v(P_i,P_j). 
\end{equation}
C'est la \emph{discr\'epance} $v$-adique de~$P$:
elle mesure la proximit\'e mutuelle des conjugu\'es de~$P$
pour la topologie~$v$-adique; elle est d'autant plus grande
que ces conjugu\'es sont proches.

\begin{prop}
Pour tout point~$P\in\P^1(\bar\Q)$ priv\'e de~$\P^1(F)$, on a
\[ \hat h(P) = \frac12  \sum_{v\in M_F} \eps_v D_v (P). \]
\end{prop}
\begin{proof}
Comme $P\not\in\P^1(F)$, $P\neq\infty$. On peut alors
fixer les coordonn\'ees homog\`enes des~$P_i$
sous la forme $[z_i:1]$, avec $z_i\in\bar\Q$.
Par d\'efinition de~$G_v$,
 \[  D_v(P)   = 
 - \frac1{n(n-1)} \sum_{i\neq j} \log \abs{z_i-z_j}_v
 +  \frac 2n \sum_{i=1}^n \lambda_v(P_i) 
- \frac{1}{d(d-1)}\log\abs{\Res(U,V)}_v . \]
Par th\'eorie de Galois, ou d'apr\`es le th\'eor\`eme sur les fonctions
sym\'etriques \'el\'ementaires,
le produit $\prod_{i\neq j} (z_i-z_j)$ appartient \`a~$F$;
c'est, au signe pr\`es, le discriminant~$\Delta(P)$ 
du polyn\^ome $\prod_{i=1}^n (T-z_i)$.
Par d\'efinition de la hauteur locale~$\hat h_v$
relative au diviseur~$\infty$ et \`a la fraction rationnelle
$V(X,Y)/Y^d$,  on a donc
\[ D_v(P) = 2 \hat h_v(P)  
 - \frac1{n(n-1)} \log \abs{\Delta(P)}_v
- \frac1{d(d-1)} \log \abs{\Res(U,V)}_v , \]
La d\'ecomposition de la hauteur normalis\'ee en somme de hauteurs locales
normalis\'ees, pour $P\neq\infty$,
s'\'ecrit
\[ \hat h(P) = \sum_{v\in M_F} \eps_v  \hat h_v(P) . \]
D'autre part, la formule du produit implique que
\[ \sum_{v\in M_F} \eps_v \log \abs{\Delta(P)}_v = \sum_{v\in M_F} \eps_v
\log\abs{\Res(U,V)}_v = 0. \]
La proposition en r\'esulte.
\end{proof}

\begin{coro}
Soit $(P_n)$ une suite de points distincts de~$\P^1(\bar\Q)$
telle que $\hat h(P_n)$ tend vers~$0$. 
Pour toute place~$v$ de~$F$, on a
\[ \lim_{n\ra\infty} D_v(P_n) = 0 . \]
\end{coro}
\begin{proof}
Comme il n'y a qu'un nombre fini de points de hauteur born\'ee
et de degr\'e donn\'e, le degr\'e de~$P_n$, $[F(P_n):F]$, tend
vers l'infini quand $n\ra\infty$.
D'apr\`es la prop.~\ref{prop.baker}, $\liminf_n D_v(P_n)\geq 0$.
Alors, pour toute place~$w$ de~$F$, on a 
\begin{align*}
0 &=  \lim_{n\ra\infty} \sum_{v\in M_F} \left(\eps_v D_v (P_n)\right) \\
& \geq \limsup_n \left(\eps_w D_w(P_n)\right)
+ \sum_{v\neq w} \liminf_{n} \left(\eps_v D_v(P)\right) \\
& \geq \limsup_n \left( \eps_w D_w(P_n)\right), \end{align*}
d'o\`u le corollaire. 
\end{proof}

\subsection{\'Energie}

Soit $w$ la valeur absolue archim\'edienne de~$F$ associ\'ee
\`a l'inclusion de~$F$ dans~$\C$.
Pour simplifier les notations de ce paragraphe,
nous d\'esignons  par~$G$ la fonction de Green~$G_w$ et
par~$\mu$ la mesure canonique~$\hat\mu_f$. Nous notons~$M=\P^1(\C)$
et $\Delta$ la diagonale de~$M\times M$.

Par leurs d\'efinitions m\^emes, 
$G$ et~$\mu$ sont reli\'ees par l'\'equation aux d\'eriv\'ees partielles
\begin{equation}
 \ddc  G(z,w) + \delta_{\Delta} = \mathrm d\mu(z)+\mathrm d\mu(w) 
\end{equation}
sur~$M\times M$, o\`u $\delta_{\Delta}$ d\'esigne
le courant d'int\'egration sur la diagonale~$\Delta$.
En outre, $G$ est sym\'etrique ($G(z,w)=G(w,z)$) et
minor\'ee.

On appelle \emph{\'energie} d'une mesure de probabilit\'e~$\nu$ sur~$M$
la quantit\'e (\'eventuellement infinie) 
\begin{equation}
 E(\nu)  = \int_{M\times M} G(P,Q)\,d\nu(P)\,d\nu(Q). 
\end{equation}
Le \emph{potentiel} de~$ \nu$ est d\'efini par la formule
\begin{equation}
 u_\nu (P) = \int_{M} G(P,Q) \, d\nu(Q). 
\end{equation}

Notons $W_1$ l'espace de Sobolev
des fonctions~$u$  sur~$\P^1(\C)$
qui sont localement de carr\'e sommable,
ainsi que les coefficients des formes diff\'erentielles
donn\'ees par sa d\'eriv\'ee $\partial u$ au sens des distributions.
(Il en est alors de m\^eme de~$\bar\partial u$.)
Si $u\in W_1$, on note $\norm{u}_{\Dir}$ la semi-norme
de Dirichlet de~$u$, d\'efinie par
\[ \norm{u}_{\Dir}^2 
= \frac i{2\pi} \int_{M} \partial u\wedge \bar\partial u 
= - \int_{M} u\ddc  u. \]

\begin{lemm}
\begin{enumerate}
\item Le potentiel~$u_\nu$ d'une mesure de probabilit\'e~$\nu$ 
est une fonction semi-continue
inf\'erieurement sur $M$
qui v\'erifie l'\'equation
$ \ddc  u_\nu +  \nu = \mu$ au sens des distributions.
En outre, $u_\mu$ est constante, de valeur~$E(\mu)$.

\item Si $\nu$ est une mesure de probabilit\'e dont
le potentiel~$u_\nu$ est une forme
diff\'erentielle~$\mathscr C^\infty$, on a $ E(\nu) =
 E(\mu)+ \norm{u_\nu}_{\Dir}^2$.

\item Une mesure~$\nu$ est d'\'energie finie
si et seulement si son potentiel appartient \`a~$W_1$.
\end{enumerate}
\end{lemm}
\begin{proof}
a) Comme $\max(G,m)$ est continue pour tout nombre r\'eel~$m$, 
$u_\nu$ est limite croissante de fonctions continues;
elle est donc semi-continue inf\'erieurement.
L'\'equation $\ddc  u_\nu+\nu=\mu$ se d\'emontre ais\'ement
par int\'egrations par parties.
Appliquons cette formule \`a~$\nu=\mu$; il s'ensuit
que $u_\mu$ est harmonique, donc constante, de valeur
$E(\mu)=\int u_\mu(P) \,d\mu(P)$.

b) Si $u_\nu$ est lisse, on a 
\begin{align*}
E(\nu) &= \int_{M} u_\nu(P)\,\mathrm d\nu(P) \\
&= \int_M u_\nu(P) \left( d\mu(P)-\ddc  u_\nu(P)\right) \\
&= \norm{u_\nu}^2_\Dir +  \int_M u_\nu(P) \,d\mu(P) \\
&= \norm{u_\nu}^2_\Dir + \int_{M\times M} G(P,Q) \,d\mu(P)\,d\nu(Q) \\
&= \norm{u_\nu}^2_\Dir +  \int_M u_\mu(Q) \,d\nu(P) \\
&= \norm{u_\nu}^2_\Dir + E(\mu).
\end{align*}

c)  Si $u_\nu\in W_1$, on d\'eduit facilement
de l'\'egalit\'e pr\'ec\'edente que $E(\nu)<\infty$.
Inversement, on peut approcher~$\nu$ par une suite~$\nu_\eps$
de mesures de probabilit\'e lisses d'\'energies born\'ees.\footnote{%
\`A l'aide d'une partition de l'unit\'e, on se ram\`ene
au cas de mesures support\'ees par le disque unit\'e
et o\`u $G(z,w)=\log\abs{z-w}^{-1}$. Soit $\mu_\eps$
des mesures de probabilit\'es de la forme 
$\rho_\eps(r)\,r\,\mathrm dr\,\mathrm d\theta$,
et o\`u les fonctions~$\rho_\eps$ sont positives, \`a support
dans~$[0,1]$ et o\`u $\rho_\eps(r)$ tend uniform\'ement vers~$0$ sur tout
intervalle ferm\'e de la forme~$[a,1]$ lorsque~$\eps\ra 0$.
Posons
$\nu_\eps=\nu*\mu_\eps$; les mesures~$\nu_\eps$ sont \`a densit\'e~$\mathscr C^\infty$ et convergent vers~$\nu$ quand $\eps\ra 0$.
L'in\'egalit\'e
\[ \int \log\abs{z-w}^{-1} \mathrm d\mu_\eps(w) \leq \log\abs{w}^{-1} \]
montre que l'\'energie de~$\nu_\eps$ est au plus \'egale \`a celle de~$\nu$.}
Leur potentiel $u_\eps$ appartient alors \`a~$W_1$
(hors de~$P$, $u_\eps(\cdot)-G(P,\cdot)$ est de $\ddc $ lisse
et $G(P,\cdot)$ appartient \`a~$W_1$).
Les normes $\norm{u_\eps}_\Dir$ sont born\'ees lorsque $\eps\ra 0$,
puisque les \'energies des mesures~$\nu_\eps$ sont uniform\'ement born\'ees.
Il existe donc une sous-suite faiblement convergente dans~$W_1$
modulo les constantes; soit $u\in W_1$ sa limite.
Alors, $\ddc  u =\ddc  u_\nu$, si bien que $u-u_\nu$ est une
distribution harmonique, donc est constante. On a donc $u\in W_1$.
\end{proof}

\begin{theo}
On a $E(\nu)\geq E(\mu)\geq 0$ pour toute mesure de probabilit\'e~$\nu$.
De plus, $E(\nu)=E(\mu)$ \'equivaut \`a~$\nu=\mu$.
\end{theo}
\begin{proof}
 Soit $c$
le nombre r\'eel dont l'existence est affirm\'ee par
la prop.~\ref{prop.baker}.
Ainsi, si $P_1,\dots,P_n$ sont des points distincts de~$\P^1(\C)$,
on a
\[ \frac1{n(n-1)} \sum_{i\neq j} G(P_i,P_j) \geq -c \frac{\log n}{n}. \]
Comme $G(P,Q)$ vaut~$+\infty$ si $P=Q$, cette relation vaut
m\^eme si les $P_i$ ne sont pas distincts.
On peut alors int\'egrer cette relation par  rapport \`a
la mesure-produit $\otimes \mathrm d\nu(P_i)$.
Par sym\'etrie, on obtient
\[ \frac1{n(n-1)} \sum_{i\neq j} \int_{\P^1(\C)\times\P^1(\C)} G(P,Q) \,d\nu(P)\,d\nu(Q)  \geq -c \frac{\log n}{n}, \]
soit encore
\[ E(\nu) = \int_{\P^1(\C)\times\P^1(\C)} G(P,Q) \,d\nu(P)\,d\nu(Q)  \geq -c \frac{\log n}{n}. \]
Lorsque $n\ra\infty$, on obtient $E(\nu)\geq 0$.

Supposons alors $E(\nu)\leq E(\mu)$. En particulier, $\nu$
est d'\'energie finie. 
Il existe donc $u\in W_1$ tel que $\ddc  u + \nu=\mu$
et $E(\nu)=E(\mu)+\norm{u}_\Dir^2$. Par cons\'equent,
$E(\nu)\geq E(\mu)$, d'o\`u l'\'egalit\'e.
Alors, $\norm{u}_\Dir^2=0$ si bien que $u$ est constante.
Donc $\mu=\nu$.
\end{proof}

\subsection{Fin de la d\'emonstration}

L'ensemble des mesures de probabilit\'e sur la sph\`ere
de Riemann $\P^1(\C)$ est faiblement compact.
Pour montrer que la suite de mesures de probabilit\'es~$(\delta_{P_n})$
converge vers~$\hat\mu_f$, il suffit de montrer
que~$\hat\mu_f$ est sa seule valeur d'adh\'erence.
Quitte \`a extraire une sous-suite convergente de la suite~$(\delta_{P_n})$,
il est ainsi loisible de supposer que $(\delta_{P_n})$
converge vers une mesure de probabilit\'e~$\nu$. On doit montrer
que $\nu=\hat\mu_f$.

\begin{lemm}
On a 
\[ \int_{\P^1(\C)\times\P^1(\C)} G_w(P,Q) \,\mathrm d\nu(P)\,\mathrm d\nu(Q) \leq 0. 
\]
\end{lemm}
\begin{proof}
Notons~$\Delta$ la diagonale de~$\P^1(\C)\times\P^1(\C)$.
Par d\'efinition,
\[ D_w(P_n) = \int_{\P^1(\C)\times\P^1(\C) \setminus \Delta} G_w(z,z')
  d\delta_{P_n}(z) \,  d\delta_{P_n}(z'). \]
La fonction~$G_w$  est continue hors de la diagonale
et minor\'ee partout.  Si $M\in\R$, posons $G_w^M=\max(M,G_w)$.
Alors,
\[  \int_{\P^1(\C)\times\P^1(\C)} G_w^M(z,z')
 \delta_{P_n}(z) \, \delta_{P_n}(z')
 \leq  D_w(P_n) + \frac1n M. \]
Faisons tendre~$n$ vers l'infini; on obtient donc
\[ \int_{\P^1(\C)\times\P^1(\C)} G_w^M(z,z') \,\mathrm d\nu(z)\,\mathrm d\nu(z')
 \leq  0. \]
Faisons tendre maintenant~$M$ vers~$\infty$. Le lemme
de convergence monotone implique  alors
\[ \int_{\P^1(\C)\times\P^1(\C)} G_w(z,z') \,\mathrm d\nu(z)\,\mathrm d\nu(z')
 = \lim_{M\ra\infty}
 \int_{\P^1(\C)\times \P^1(\C)} G_w^M(z,z') \,\mathrm d\nu(z)\,\mathrm d\nu(z')
\leq  0. \]
\end{proof}

Autrement dit, la mesure~$\nu$ est d'\'energie n\'egative ou nulle.
Comme $E(\nu)\geq E(\mu)\geq 0$, on a $\nu=\mu$.

\section{Preuve de l'in\'egalit\'e de Baker}\label{sec.baker}

Nous d\'emontrons dans ce paragraphe la proposition~\ref{prop.baker}.

\begin{lemm}
Pour tout entier~$n$, posons
\[ D_n = \inf_{(P_1,\dots,P_n)\in \P^1(\C_v)^n} \frac1{n(n-1)} \sum_{i\neq j} G_v(P_i,P_j). \] 
La suite $(D_n)$ est croissante.
\end{lemm}
\begin{proof}
En effet, si $P_0,\dots,P_n$ sont des points de~$\P^1(\C_v)$,
\begin{align*}
 \sum_{\substack{i,j=0 \\ i\neq j}}^n G_v(P_i,P_j)
&  = \frac1{n-1} \sum_{k=0}^n \left( \sum_{\substack{i\neq j \\ i,j \neq k}} G_v(P_i,P_j) 
 \right)  \\
& \geq \frac1{n-1} \sum_{k=0}^n n(n-1) D_n 
= n(n+1) D_n , 
\end{align*}
d'o\`u l'in\'egalit\'e $D_{n+1}\geq D_n$.
\end{proof}

Il va ainsi suffir de minorer $D_N$ pour certains entiers~$N$;
on peut par exemple supposer que $N=d^{k+1}$, avec $k\in\N$.
Pour la d\'emonstration, il sera cependant plus ais\'e de supposer
$N=td^k$ avec $k\geq 0$ et $2\leq t\leq 2d-1$.

Posons $m=N-1$.
L'espace vectoriel~$\mathscr P_m$ des polyn\^omes homog\`enes de degr\'e~$m$
est de dimension~$N$.
Nous allons en consid\'erer deux bases.

La premi\`ere de ces bases,
$\mathscr B_N$ est la famille des mon\^omes $X^aY^b$, o\`u $a+b=m$.

Rappelons que l'on a d\'efini des polyn\^omes $U_n(X,Y)$
et $V_n(X,Y)$ pour tout entier~$n\geq 0$ en posant
$U_0=X$, $V_0=Y$, puis en d\'efinissant par r\'ecurrence, 
$U_{n+1}=U(U_n(X,Y),V_n(X,Y))$
et $V_{n+1}=V(U_n(X,Y),V_n(X,Y))$ si $n\geq 0$.
On a vu que ces polyn\^omes {\og rel\`event\fg} \`a~$\C_v^2$ la dynamique
de l'endomorphisme~$f$ de~$\P^1(\C_v)$; pour tout~$n$,
$U_n$ et~$V_n$ sont premiers entre eux, non nuls.
Notons alors $\mathscr C_N$ la famille des polyn\^omes de la forme
\[ U_0^{a_0} V_0^{b_0} U_1^{a_1} V_1^{b_1} \dots U_k^{a_k} V_k^{b_k}, \]
o\`u $a_i+b_i=d-1$ pour $0\leq i<k$ et $a_k+b_k=t-1$.
Cette famille est de cardinal~$N$ et contenue dans l'espace~$\mathscr P_m$.
Soit $A_N$ sa matrice dans la base~$\mathscr B_N$.

\begin{lemm}
Posons $r=N(N-(t+k(d-1))/2d(d-1)$.
Le d\'eterminant de~$A_N$ v\'erifie $\det(A_N) =\pm \Res(U,V)^r$.
\end{lemm}
\begin{proof}
Montrons d'abord que $\mathscr C_N$ est une base de~$\mathscr P_m$.
On peut pour cela raisonner par l'absurde et supposer que
$N$ est le plus petit entier tel que $\mathscr C_N$ soit li\'ee.
Consid\'erons une
relation de d\'ependance lin\'eaire  non triviale
\[ \sum c_{\mathbf a,\mathbf b} U_0^{a_0}V_0^{b_0} \dots U_k^{a_k}V_k^{b_k}=0,\]
qu'on peut r\'ecrire $\alpha U_k^{t-1}+\beta V_k=0$,
avec
\begin{align*}
\alpha &= \sum_{b_k=0} c_{\mathbf a,\mathbf b}  U_0^{a_0}V_0^{b_0} \dots U_{k-1}^{a_{k-1}}V_{k-1}^{b_{k-1}} \\
\beta & = \sum_{b_k\geq 1} c_{\mathbf a,\mathbf b}  U_0^{a_0}V_0^{b_0} \dots U_{k-1}^{a_{k-1}}V_{k-1}^{b_{k-1}}U_k^{a_k} V_k^{b_k-1}.
\end{align*}
Si $\alpha=0$, on divise tout par~$V_k$ et on remplace~$t$ par~$t-1$,
sauf si $t=2$ auquel cas on remplace $k$ par~$k-1$.
Par minimalit\'e de~$N$, $\alpha\neq 0$. Comme $U_k\neq 0$,
on a alors $\beta\neq 0$.
Notons que
\[ \deg(\alpha)= m-d^k(t-1)=td^k-1-d^k(t-1)=d^k-1. \]
Comme $U_k$ et~$V_k$  sont sans facteur commun,
$V_k$ divise~$\alpha$, ce qui contredit l'in\'egalit\'e
$\deg(\alpha)=d^k-1<\deg(V_k)=d^k$.

Inversement,
si $U$ et $V$ avaient eu un facteur commun, la famille $\mathscr C_N$ ne
serait manifestement pas libre.

Oublions alors un instant les valeurs sp\'ecifiques des polyn\^omes~$U$
et~$V$: le d\'eterminant de~$A_N$ est un polyn\^ome
\`a coefficients entiers en les coefficients de~$U$ et de~$V$.
Il s'annule si et seulement si ces polyn\^omes ont un facteur
commun, c'est-\`a-dire si et seulement si le r\'esultant
de~$U$ et~$V$ s'annule. Le th\'eor\`eme des z\'eros de Hilbert
entra\^{\i}ne alors que le d\'eterminant de~$A_n$ et le r\'esultant~$\Res(U,V)$
ont exactement les m\^emes facteurs irr\'eductibles dans
l'anneau (factoriel) des polyn\^omes.
Or, $\Res(U,V)$ est lui-m\^eme un polyn\^ome \emph{irr\'eductible}
en les coefficients de~$U$ et~$V$.
Il existe
donc des entiers~$c$ et~$s$ tels que 
\[ \det(A_N)= c \Res(U,V)^s. \]
Consid\'erons le degr\'e de~$A_N$ en les coefficients de~$U$:
il vaut
\begin{align*}
 \deg_U (\det(A_N) )& = (d-1) \deg(U_1)+ \dots+(d-1)\deg(U_{k-1})
+ (t-1) \deg(U_k) \\
&= (d-1) + (d^2-1) + \dots+(d^{k-1}-1) + (t-1) \frac{d^k-1}{d-1} \\
&= t \frac{d^k-1}{d-1} -k, 
\end{align*}
car le degr\'e de~$U_k$ en les coefficients de~$U$ est
\'egal \`a~$(d^k-1)/(d-1)$, comme on le v\'erifie par r\'ecurrence.
Le degr\'e du r\'esultant $\Res(U,V)$
en les coefficients de~$U$ est \'egal \`a~$2d$.
On a donc $s=(t(d^k-1)-k(d-1))/2d(d-1)=r$.

Pour calculer l'entier~$c$, il suffit 
de le d\'eterminer pour un choix particulier de~$U$
et~$V$, par exemple $U=X^d$ et $V=Y^d$. La famille~$\mathscr C_N$
co\"{\i}ncide alors avec la famille~$\mathscr B_N$, \`a l'ordre pr\`es.
On a donc $c=\pm 1$, ce qui ach\`eve la d\'emonstration du lemme.
\end{proof}

Notons $B^{(i)}$ les \'el\'ements de~$\mathscr B_N$, pour $1\leq i\leq N$,
et $C^{(i)}$ ceux de~$\mathscr C_N$.
Pour $1\leq j\leq N$, soit $[x_j:y_j]$ un syst\`eme de coordonn\'ees
homog\`enes du point~$P_j$.
On commence par calculer le d\'eterminant~$\beta$ de la matrice~$B$
dont le terme de ligne~$i$ et de colonne~$j$
est $B^{(i)}(x_j,y_j)$.
Il vaut 
\begin{align*}
 \beta & = x_1^{N-1}\dots x_N^{N-1} \det ( (y_j/x_j)^i ) \\
& =  x_1^{N-1}\dots x_N^{N-1} \prod_{i>j} \left( \frac{y_i}{x_i}-\frac{y_j}{x_j}\right)  \\
&= x_1^{N-1} \dots x_N^{N-1} \left( \prod_{i>j} x_ix_j\right)^{-1}
 \prod_{i>j} (y_ix_j-x_iy_j). 
\end{align*}
Par suite,
\[ \beta^2  = \pm \prod_{i\neq j} ( y_ix_j-x_i y_j). \]

\'Evaluons aussi le d\'eterminant~$\gamma$ de la matrice~$C$
dont le terme~$(i,j)$ est $C^{(i)}(x_j,y_j)$.
Soit $M$ un nombre r\'eel tel que l'ensemble de Julia rempli homog\`ene~$\mathscr J_v$
soit contenu dans l'ensemble des couples $(x,y)\in\C_v^2$
tels que $\abs{\strut x}_v\leq M$ et $\abs{y}_v\leq M$.
Si $(x,y)\in \mathscr J_v$, $(U_k(x,y),V_k(x,y))$ appartient \`a~$\mathscr J_v$
pour tout~$k$, d'o\`u $\abs{U_k(x,y)}_v\leq M$
et $\abs{V_k(x,y)}_v\leq M$.
Si de plus $v$ est une place ultram\'etrique,
on a alors, avec $C^{(i)}$ d\'efini par la famille~$(\mathbf a,\mathbf b)$,
\[ \abs{C^{(i)}(x,y)}_v = \prod_{i=0}^k \abs{U_i(x,y)}_v^{a_i} \abs{V_i(x,y)}_v^{b_i} \leq M^{a_0+b_0} \dots M^{a_k+b_k}
= M^s, \]
avec $s={k(d-1)+(t-1)}$.
Cela entra\^{\i}ne 
$ \abs{\gamma}_v \leq M^{sN}$.
Dans le cas o\`u $v$ est une place archim\'edienne, l'in\'egalit\'e d'Hadamard
implique seulement $\abs{\gamma}_v \leq N^{N/2} M^{s N}$.
Posons $\theta_v=1$  si $v$ est archim\'edienne et~$\theta_v=0$
sinon.

Enfin, la d\'efinition de la matrice~$A_N$ implique
que l'on a $\alpha\beta=\gamma$, o\`u $\alpha=\det(A_N)$.
Par cons\'equent,
\[ \prod_{i\neq j} \abs{y_ix_j-x_iy_j}_v
 = \abs{\gamma}_v^2 \abs{\alpha}_v^{-2}
 = \abs{\gamma}_v^2 \abs{\Res(U,V)}_v^{-2r}. \]
Passons au logarithme,
on obtient
\[
\frac1{N(N-1)} \sum_{i\neq j} \log \abs{y_ix_j-x_iy_j}_v
 \leq  \theta_v \frac{N\log N}{N(N-1)}
 + \frac{2s}{N-1} \log(M) - \frac{2r}{N(N-1)} \log \abs{\Res(U,V)}_v 
\]
soit encore, compte tenu de la d\'efinition~\eqref{eq.Gv} de
la fonction~$G_v$,
\begin{multline*} 
\frac1{N(N-1)} \sum_{i\neq j} G_v(P_i,P_j) \geq 
 - \theta_v \frac{\log N}{N-1} - \frac{2s}{N-1} \log M  \\
{} + \left( \frac{2r}{N(N-1)} - \frac1{d(d-1)}\right)
 \log \abs{\Res(U,V)}_v +2 \min_{i}  \Lambda_v(x_i,y_i). 
\end{multline*}
Si $2\leq t\leq 2d-1$ (il suffit que $t$ soit born\'e),
il en r\'esulte alors l'in\'egalit\'e
\[ \frac1{N(N-1)} \sum_{i\neq j} G_v(P_i,P_j) \geq 
 - c \frac{\log N}{N} +2 \min_{i} \Lambda_v(x_i,y_i). \]
Rappelons que les coordonn\'ees homog\`enes $[x_i:y_i]$
ont \'et\'e choisies de sorte que $(x_i,y_i)$ appartienne \`a
l'ensemble de Julia rempli homog\`ene~$\mathscr J_v$.
On a alors $\Lambda_v(x_i,y_i)\leq 0$, mais on peut
utiliser l'homog\'en\'eit\'e de~$\Lambda_v$ et la densit\'e dans~$\R$ de l'ensemble
des valeurs absolues des \'el\'ements de~$\C_v^*$  
pour choisir chaque couple~$(x_i,y_i)$
de sorte que $\Lambda_v(x_i,y_i)$ soit arbitrairement proche de~$0$.

Cela termine la d\'emonstration de la prop.~\ref{prop.baker}

\section{Le th\'eor\`eme d'\'equidistribution de Bilu}
\label{sec.bilu}

Lorsque $f\colon\P^1\ra\P^1$ est l'endomorphisme 
donn\'e par~$[x:y]\mapsto [x^d:y^d]$, la hauteur
canonique est la hauteur usuelle et la mesure canonique
est la mesure de Haar normalis\'ee du sous-groupe
compact~$\mathbf S_1$ de~$\C^*\subset\P^1(\C)$.
Le th\'eor\`eme d'\'equidistribution qu'on obtient dans ce cas
est d\^u  \`a~\cite{bilu97}, mais une variante
l\'eg\`erement plus faible remonte \`a~\cite{erdos-turan1950}.
C'est la proposition:
\begin{prop}\label{prop.equidis}
Soit $(x_n)$ une suite de nombres alg\'ebriques deux \`a deux distincts
telle que $h(x_n)$ tend vers~$0$.
La suite de mesures $(\delta_{x_n})$ sur~$\P^1(\C)$ converge
vers la mesure de Lebesgue port\'ee par le cercle unit\'e~$\mathbf S_1$.
\end{prop}

Dans ce paragraphe, je voudrais expliquer comment
\cite{bilu97} d\'eduit de cette proposition
une d\'emonstration de la conjecture d'\'equidistribution~\ref{conj.equidis}
lorsque $X$ est un syst\`eme dynamique torique.

Pour simplifier l'exposition, 
on consid\`ere $(\bar\Q^*)^k=\gm^k(\bar\Q)$ comme ouvert dense de
l'espace projectif~$(\P^1)^k$; la hauteur d'un point
$(x_1,\ldots,x_k)$ de~$(\P^1)^k(\bar\Q)$ est alors
la somme  des hauteurs usuelles des~$x_i$.

On appelle sous-vari\'et\'e de torsion
de~$\gm^k$ un translat\'e $gH$ d'un sous-tore $H$ de~$\gm^k$
par un point d'ordre fini $g\in\gm^k(\bar\Q)$.
La mesure canonique $\hat\mu_k$
 sur~$\gm^k(\C)=(\C^*)^k$, aussi not\'ee~$\hat\mu$,
est la mesure de Haar port\'ee par le sous-groupe compact
maximal~$(\mathbf S_1)^k$ de~$(\C^*)^k$; autrement dit,
pour toute fonction~$\phi$ \`a support compact sur~$(\C^*)^k$,
\[ \hat\mu(\phi) = \frac1{(2\pi)^k}\int_0^{2\pi}\cdots\int_0^{2\pi}
  \phi(e^{i\theta_1},\ldots,e^{i\theta_k})\,\mathrm d\theta_1\cdots
\mathrm d\theta_k. \]

\begin{theo}[\cite{bilu97}] \label{theo.bilu}
Soit $(x_n)$ une suite de points de $(\bar\Q^*)^k$
qui v\'erifie les deux assertions:
\begin{enumerate}
\item la suite $(h(x_n))$ tend vers~$0$;
\item toute sous-vari\'et\'e de torsion~$Z\subsetneq\gm^k$
ne contient qu'un nombre fini de termes de la suite~$(x_n)$.
\end{enumerate}
Alors, la suite de mesures de probabilit\'es~$(\delta_{x_n})$ sur~$\gm^k(\C)=(\C^*)^k$ converge vers la mesure~$\hat\mu$.
\end{theo}
Il s'agit de la convergence des mesures de probabilit\'es sur
un espace localement compact, aussi appel\'ee \emph{convergence \'etroite}.
C'est un peu plus fort que la convergence vague car cela garantit
que la masse ne part pas, m\^eme partiellement, \`a l'infini.
Bien s\^ur, cela co\"{\i}ncide ici avec la convergence
des mesures de probabilit\'es sur la compactification~$\mathbf P^1(\C)^k$.

\begin{lemm}
Soit $(x_n)$ une suite de points de~$\gm^k(\bar\Q)$ telle
que $h(x_n)$ tend vers~$0$. 
Alors, pour tout voisinage~$V$ du polycercle unit\'e~$(\mathbf S_1)^k$
de~$\gm^k(\C)$, la suite~$(\delta_{x_n}(V))$ tend vers~$1$.

Par cons\'equent,
la suite de mesures de probabilit\'e~$(\delta_{x_n})$
sur $\gm^k(\C)$ est \emph{tendue},
c'est-\`a-dire relativement compacte pour la topologie \'etroite.
De plus, toute mesure de probabilit\'e adh\'erente
\`a la suite~$(\delta_{x_n})$ est support\'ee par le polycercle unit\'e.
\end{lemm}
\begin{proof}
Soit $r$ un nombre r\'eel tel que $r>1$ et tel
que $V$ contienne l'ensemble~$K$ des points
$(z_1,\ldots,z_k)$ tels que $r^{-1}\leq \abs {z_i}\leq r$
pour $1\leq i\leq k$.
Pour $\alpha\in\bar\Q$, on a 
\begin{align*}
 h(\alpha) &= \frac1{[\Q(\alpha):\Q]}\sum_{p\leq\infty} \sum_{\sigma\colon \Q(\alpha)\hra\C_p} 
  \log\max(1,\abs{\sigma(\alpha)}_p) \\
& \geq \frac1{[\Q(\alpha):\Q]} \sum_{\sigma\colon\Q(\alpha)\hra\C}
  \log\max(1,\abs{\sigma(\alpha)} 
 = \int \log\max(1,\abs z) \,\mathrm\delta_{\alpha}(z). \end{align*}
Comme $h(\alpha)=h(\alpha^{-1})$ si $\alpha\neq 0$, on a aussi
\[ h(\alpha)\geq  \int \log\max(1,\abs z) \,\mathrm\delta_{\alpha^{-1}}(z)
= \int \log\max(1,\abs z^{-1}) \,\mathrm\delta_{\alpha}(z)
, \]
si bien que pour $\alpha\neq 0$,
\[ h(\alpha) \geq \frac12 \int\log\max(\abs z,\abs z^{-1})\mathrm \delta_{\alpha}(z). \]
Si $z\in\gm(\C)$  est tel que
$\abs z>r$ ou $\abs z<1/r$, on a l'in\'egalit\'e
 $ \log\max(\abs z,\abs z^{-1}) \geq \log r$.
La mesure pour $\delta_\alpha$
du compl\'ementaire de la couronne $r^{-1}\leq \abs z\leq r$
dans~$\C$ est donc major\'ee par
\[  \int \frac{\log\max(\abs z,\abs{z}^{-1})}{\log r} \,\delta_\alpha(z)\leq \frac 2{\log r} h(\alpha). \]

Par cons\'equent, additionnant ces in\'egalit\'es,
la mesure du compl\'ementaire de~$K$
dans~$\gm^k(\C)$ est major\'ee par
\[ \delta_{x_n}(\complement K) 
\leq \frac1{\log r} \int \left(\sum_{i=1}^r \log\max(\abs {z_i},\abs {z_i}^{-1}\right) \,\delta_{x_n}(z_1,\ldots,z_n)
\leq \frac 2{\log r} h(x_n). \]
Lorsque $n\ra\infty$, la suite $(\delta_{x_n}(\complement K))$
tend donc vers~$0$ ; la suite~$(\delta_{x_n}(V))$ tend vers~$1$.
\end{proof}

Pour conclure la d\'emonstration du th\'eor\`eme~\ref{theo.bilu},
il suffit donc de d\'emontrer que si une sous-suite
de la suite~$(\delta_{x_n})$ converge vers une mesure
de probabilit\'e~$\nu$,
alors $\nu=\hat\mu$. 
Quitte \`a remplacer la suite~$(x_n)$ par une sous-suite,
on peut supposer, et on le fait, que la suite~$(\delta_{x_n})$
converge vers une mesure de probabilit\'e~$\nu$.
Les deux hypoth\`eses concernant la suite~$(x_n)$
sont encore satisfaites.

\begin{lemm}
Pour qu'une mesure de probabilit\'e~$\nu$ support\'ee par~$(\mathbf S_1)^k$
soit \'egale \`a la mesure~$\hat\mu$, il faut et il suffit que 
pour tout $k$-uplet $(a_1,\ldots,a_k)\in\Z^k$ distinct de~$(0,\ldots,0)$,
notant $\psi\colon\gm^k(\C)\ra\gm(\C)$ l'application 
donn\'ee par $(z_1,\ldots,z_k)\mapsto z_1^{a_1}\cdots z_k^{a_k}$,
la mesure $\psi_*(\nu)$ soit \'egale \`a la mesure de Haar
normalis\'ee sur~$\mathbf S_1$.
\end{lemm}
\begin{proof}
La n\'ecessit\'e de la condition est \'evidente : comme $\psi$
induit par restriction un homomorphisme surjectif
de~$(\mathbf S_1)^k$ dans~$\mathbf S_1$, $\psi_*(\hat\mu)$ est la mesure 
de Haar normalis\'ee sur~$\mathbf S_1$.
D\'emontrons qu'elle est suffisante.
Il s'agit pour cela de prouver que pour toute fonction continue~$\phi$
sur~$(\mathbf S_1)^k$, $\int_{(\mathbf S_1)^k} \phi\,\mathrm d\nu
= \int_{(\mathbf S_1)^k}\phi\,\mathrm d\hat\mu$.
Une fonction continue sur~$(\mathbf S_1)^k$
peut \^etre approch\'ee uniform\'ement  par un polyn\^ome trigonom\'etrique,
c'est-\`a-dire une somme finie
\[ P=\sum_{\mathbf a\in\Z^k} c_{\mathbf a} z_1^{a_1}\cdots z_n^{a_n}. \]
Il suffit donc de prouver l'\'egalit\'e lorsque $\phi$ est un tel polyn\^ome,
puis, par lin\'earit\'e, lorsque $\phi$ est un mon\^ome~$z_1^{a_1}\cdots z_n^{a_n}$.
Si $(a_1,\ldots,a_n)=0$, $\phi=1$
et l'\'egalit\'e vaut car $\nu$ et~$\hat\mu$ sont des mesures de probabilit\'e.
Lorsque $(a_1,\ldots,a_n)\neq 0$, elle vaut encore par hypoth\`ese.
En effet,
\[ \int_{(\mathbf S_1)^k} z_1^{a_1}\cdots z_k^{a_k}\,\mathrm d\nu(z_1,\ldots,z_n) = \int_{(\mathbf S_1)^k} \psi(z_1,\ldots,z_k)\,\mathrm d\nu(z_1,\ldots,z_n)
= \int_{\mathbf S_1} z \,\mathrm d(\psi_*(\nu))(z). \]
Puisque $\psi_*(\nu)$ est la mesure de Haar normalis\'ee sur~$\mathbf S_1$,
cette derni\`ere int\'egrale vaut
\[ \int_0^{2\pi} e^{it}\, \frac{\mathrm dt}{2\pi}=0. \]
Pour la m\^eme raison,
\[ \int_{(\mathbf S_1)^k} z_1^{a_1}\cdots z_k^{a_k}\,\mathrm d\hat\mu(z_1,\ldots,z_n)= \int_{\mathbf S_1} z\, \mathrm d(\psi_*(\hat\mu))(z)=0.\]
Le lemme est ainsi d\'emontr\'e.
\end{proof}

Soit donc $(a_1,\ldots,a_k)$ un $k$-uplet d'entiers relatifs
non tous nuls et soit $\psi\colon\gm^k\ra\gm$ l'application
donn\'ee par
\[ \psi(z_1,\ldots,z_n)= z_1^{a_1}\cdots z_n^{a_n}. \]
Notons que tout $x\in\gm^k(\bar\Q)$, la mesure $\psi_*(\delta_x)$
sur~$\gm(\bar \C)$ est \'egale \`a la mesure~$\delta_{\psi(x)}$.
Lorsque $n$ tend vers l'infini, la suite de mesures
de probabilit\'e~$(\delta_{\psi(x_n)})$ sur~$\gm(\C)$
converge donc vers la mesure~$\psi_*(\nu)$.
Nous allons d\'eduire de la proposition~\ref{prop.equidis}
que $\psi_*(\nu)$ est la mesure
de Haar normalis\'ee~$\hat\mu_1$ port\'ee par~$\mathbf S_1$.

Le lemme suivant v\'erifie que les hauteurs des points~$\psi(x_n)$
tendent vers~$0$.
\begin{lemm}
Pour $x\in\gm^k(\bar\Q)$, on a
\[ h(\psi(x)) \leq \left( \abs{a_1}+\cdots+\abs{a_n}\right) h(x). \]
\end{lemm}
\begin{proof}
Compte tenu de la relation $h(\alpha^{-1})=h(\alpha)$,
il suffit de d\'emontrer que  pour
$\alpha$, $\beta\in\bar\Q^*$, on a 
$h(\alpha\beta)\leq h(\alpha)+h(\beta)$. On conclut alors
en effet par r\'ecurrence.
Revenons \`a la formule~\eqref{def.hauteur-p} d\'efinissant
la hauteur d'un point et observons l'in\'egalit\'e
\[ \max(0,u+v)\leq \max(0,u)+\max(0,v) , \]
valable pour tout couple~$(u,v)$ de nombres r\'eels.
Si $K$ est un corps de nombres
contenant~$\alpha$ et~$\beta$, il vient ainsi
\begin{align*} h(\alpha\beta) & =\frac1{[K:\Q]}
 \sum_{p\leq\infty} \sum_{\sigma\colon K\hra\C_p} \log\max(1,\abs{\alpha\beta}_p) \\
& \leq \frac1{[K:\Q]}
 \sum_{p\leq\infty} \sum_{\sigma\colon K\hra\C_p} 
\left(\log\max(1,\abs\alpha_p)+\log\max(1,\abs\beta_p)\right) \\
&\leq  h(\alpha)+h(\beta). \end{align*}
\end{proof}

Si une sous-suite de la suite~$(\psi(x_n))$ est constante, de valeur~$y$,
on a donc $h(y)=0$, si bien que $y$ est
une racine de l'unit\'e. Pour chaque entier~$n$ tel que $\psi(x_n)=y$,
le point $x_n$ appartient \`a la sous-vari\'et\'e de~$\gm^k$
d\'efinie par l'\'equation $\psi(x)=y$.  C'est une sous-vari\'et\'e
de torsion, translat\'ee d'un sous-tore de~$\gm^k$, le noyau de~$\psi$,
par n'importe quel point de torsion $x_1\in\gm^k$ tel que $\psi(x_1)=y$.
Une sous-suite de la suite~$(x_n)$ est alors toute enti\`ere
contenue dans une vari\'et\'e de torsion, ce qui contredit l'hypoth\`ese
du th\'eor\`eme.

Par cons\'equent, la suite~$(\psi(x_n))$ prend une infinit\'e de valeurs
distinctes et l'on peut extraire 
de la suite~$(x_n)$
une sous-suite $(x_{\phi(n)})$ 
telle que les termes de la suite~$(\psi(x_{\phi(n)}))$
soient deux \`a deux distincts. Cette derni\`ere suite est
justiciable de la proposition~\ref{prop.equidis},
si bien  que la suite~$(\delta_{\psi(x_{\phi(n)})})$ converge
vers la mesure~$\hat\mu_1$. Comme elle converge aussi
vers~$\psi_*(\nu)$, on a l'\'egalit\'e requise $\psi_*\nu=\hat\mu_1$.
Cela termine la d\'emonstration 
de~\textsc{Bilu}
du th\'eor\`eme~\ref{theo.bilu}.

\section{Exercices}

\begin{exer}
Soit $u\colon\C_v^2\ra\C_v^2$ une application affine.
Soit $K$ une partie born\'ee de~$\C_v^2$.
Montrer que le diam\`etre transfini homog\`ene de~$u(K)$
est \'egal \`a $\abs{\det u} \delta^h(K)$.
\end{exer}

\begin{exer}[\cite{demarco2003}]
Soit $K$ un corps et soit $(U,V)$ et $(F,G)$ deux couples de polyn\^omes homog\`enes
de~$K[X,Y]$.
On suppose que $U$ et~$V$ sont premiers entre eux et de m\^eme degr\'e~$d$,
et que $F$ et~$G$ sont premiers entre eux, de m\^eme degr\'e~$e$.
On pose alors $U_1=U(F,G)$ et $V_1=V(F,G)$.

a) Montrer que $U_1$ et $V_1$ sont premiers entre eux et 
de degr\'es $de$.

b) Montrer que le r\'esultant $\Res(U_1,V_1)$
s'annule si et seulement si l'un des r\'esultants $\Res(U,V)$
ou $\Res(F,G)$ s'annule.
En d\'eduire qu'il existe des entiers~$a$, $b$ et~$c$
tels que 
\[ \Res(U_1,V_1) = c \Res(U,V)^a \Res(F,G)^b. \]

c) En consid\'erant le degr\'e du r\'esultant $\Res(U_1,V_1)$ par rapport aux
coefficients des polyn\^omes~$U$, $V$, $F$, $G$, 
montrer que $a=e$ et $b=d^2$. Montrer aussi que $c=\pm 1$
(si un nombre premier~$p$ divisait~$c$, le r\'esultant $\Res(U_1,V_1)$
serait identiquement nul en caract\'eristique~$p$).

d) Consid\'erer un cas particulier et v\'erifier que $c=1$.
\end{exer}

\begin{exer}
a) Soit $(X,d)$ un espace m\'etrique et 
soit $\alpha$ un nombre r\'eel positif. 
Une fonction
$\phi\colon X\ra\R$ est dite $\alpha$-h\"olderienne si
$\abs{\phi(x)-\phi(y)}/d(x,y)^\alpha$ est major\'e
lorsque $(x,y)$ parcourt l'ensemble des couples
de points distincts de~$X$. La borne sup\'erieure de cette
expression est not\'ee $\norm{\phi}'_\alpha$. 
V\'erifier que l'on d\'efinit une norme sur l'espace vectoriel
$E$ des fonctions $\alpha$-h\"olderiennes
born\'ees en posant $\norm f_\alpha=\max(\norm f'_\alpha,\norm f_\infty)$,
et que cet espace vectoriel ainsi norm\'e est complet.

b) Soit $f\colon X\ra X$ une application $C$-lipschitzienne
(c'est-\`a-dire que $d(f(x),f(y))\leq C d(x,y)$ pour tout
couple $(x,y)$ de points de~$X$) ; soit $k$ un nombre
r\'eel tel que $\abs k<1$ et soit $\phi_0$ une fonction dans~$E$.
Montrer que l'application $\phi\mapsto T\phi=k \phi\circ f+\phi_0$
applique~$E$ dans lui-m\^eme ; d\'emontrer que
si $\alpha<\log \abs k^{-1}/\log C$, alors $T$ est contractante.
(Voir aussi~\cite{dinh-sibony2005}, lemme 5.4.2.)
\end{exer}

\begin{exer}
Soit $K$ un corps valu\'e complet.
Pour $P=[x:y]$ et $Q=[u:v]$ dans~$\P^1(K)$, on pose
$d(P,Q)=\abs{xv-yu}/\max(\abs x,\abs y)\max (\abs u,\abs v)$.

a) V\'erifier que $d(P,Q)$ ne d\'epend pas
du choix des coordonn\'ees homog\`enes de~$P$ et~$Q$.
Montrer que~$d$ est une distance sur~$\P^1(K)$.

b) Soit $F=[U:V]$ un endomorphisme de degr\'e~$d$ de~$\P^1$
dans lui-m\^eme, d\'efini par deux polyn\^omes homog\`enes~$U$ et~$V\in K[X,Y]$
sans facteur commun, de degr\'e~$d$.
Montrer que l'application induite par~$F$ sur~$\P^1(K)$ est
lipschitzienne.

c) On reprend les notations du paragraphe~\ref{sec.green}.
\`A l'aide de l'exercice pr\'ec\'edent, prouver
que l'application $[x:y]\mapsto \Lambda_v(x,y)-\log\max(\abs x,\abs y)$,
bien d\'efinie sur~$\P^1(\C_v)$, est h\"olderienne.
(Voir aussi~\cite{sibony1999,favre-rl2006,kawaguchi-silverman2007b}.)
\end{exer}

\begin{exer}
On consid\`ere deux applications continues~$f_1$
et~$f_2$ d'un espace m\'etrique~$(X,d)$ dans lui-m\^eme.
Soit $\alpha$ un nombre r\'eel et
soit $E$ l'espace vectoriel norm\'e des fonctions born\'ees
de~$X$ dans~$\R$ qui sont h\"olderiennes d'exposant~$\alpha$.
Soit $k$ un nombre r\'eel tel que $\abs k<1$.

On suppose que $f_1$ et~$f_2$ sont $C$-lipschitziennes
et que $\alpha<\log \abs k^{-1}/\log C$.
On consid\`ere enfin deux \'el\'ements~$\phi_1$ et~$\phi_2$
de~$E$ et les applications affines~$T_1$ et~$T_2$
de~$E$ dans lui-m\^eme donn\'ees par $T_1\phi=k\phi\circ f_1+\phi_1$
et $T_2\phi=k\phi\circ f_2+\phi_2$. 
Elles admettent chacune un et un seul point fixe dans~$E$,
disons $\hat\phi_1$ et $\hat\phi_2$.
Montrer que l'on a 
\[ \norm{\hat \phi_1-\hat\phi_2}\leq 
 \frac{\norm{\phi_1-\phi_2}_\alpha + \abs k \norm{\phi_2}_\alpha d(f_1,f_2)^\alpha}{1-\abs k C^\alpha}, \]
o\`u $d(f_1,f_2)$ d\'esigne la borne sup\'erieure des quantit\'es
$d(f_1(x),f_2(x))$ lorsque $x$ parcourt~$X$.

b) On reprend les notations du paragraphe~\ref{sec.green}.
Montrer que la fonction~$\Lambda_v$ d\'efinie dans le lemme~\ref{lemm.green-homogene} d\'epend de mani\`ere continue des coefficients des polyn\^omes~$U$ et~$V$,
uniform\'ement en~$(x,y)\in \C_v^2\setminus\{(0,0)\}$.
\end{exer}

\begin{exer}\label{exer.delta-res}
Soit $f\colon\P^1\ra\P^1$ un endomorphisme de degr\'e~$d$
de~$\P^1$ donn\'e par un couple~$(U,V)$ de polyn\^omes homog\`enes
de degr\'es~$d$, premiers entre eux,
\`a coefficients dans un corps valu\'e~$\C_v$. Soit $\mathscr J_v\subset\C_v^2$
son ensemble de Julia rempli homog\`ene.

a) \`A l'aide de l'exercice pr\'ec\'edent,
montrer que le diam\`etre transfini homog\`ene~$\delta^h(\mathscr J_v)$
d\'epend contin\^ument des coefficients de~$U$ et~$V$.

b) En d\'eduire que la formule de la prop.~\ref{prop.delta-res}
est vraie sans supposer que les coefficients de~$U$ et~$V$
appartiennent \`a un corps de nombres.
\end{exer}

\begin{exer}
a) Soit $\alpha_1,\dots,\alpha_d$ des nombres complexes.
\`A l'aide de l'in\'egalit\'e d'Hadamard, d\'emontrer que
\[ \prod_{i<j} \abs{\alpha_j-\alpha_i} 
\leq d^{d/2} \prod_{i=1}^d \max(1,\abs{\alpha_i})^{d-1}. \]

b) Soit $P\in\C[X]$ un polyn\^ome de degr\'e~$d\geq 1$.
Montrer que son discriminant~$\Delta(P)$ et sa mesure de Mahler~$\mathrm M(P)$
sont
reli\'es par l'in\'egalit\'e, due  \`a \textsc{Mahler},
\[ \abs {\Delta(P)} \leq d^{d} \mathrm M(P)^{2d-2}. \]

c) Quels sont les polyn\^omes pour lesquels l'in\'egalit\'e pr\'ec\'edente
est une \'egalit\'e ?
\end{exer}
\begin{exer}
Soit $d$ un entier au moins \'egal \`a~$2$ et soit $f$
l'endomorphisme de~$\P^1$ donn\'e par $f=[U:V]$, avec $U=X^d$ et $V=Y^d$.

a)  L'ensemble de Julia homog\`ene rempli~$\mathscr J$ est le bidisque unit\'e,
ensemble des couples $(x,y)\in\C^2$ tels que $\abs {\strut x}\leq 1$ et $\abs y\leq 1$.

b) Lorsque $z_1,\dots,z_n$ parcourent~$K$, $\prod_{i\neq j}d(z_i,z_j)$
atteint son maximum en des points de la forme
$z_i=(x_i,y_i)$, o\`u $x_i$ et~$y_i$ sont de module~$1$.
Il suffit alors de consid\'erer les points de la forme $(x_i,1)$,
avec $x_i$ de module~$1$. 

c) D\'eduire alors de l'exercice pr\'ec\'edent que 
\[ \delta_n(K) = n^{1/(n-1)} \]
puis retrouver la prop.~\ref{prop.baker} dans ce cas particulier.
\end{exer}

%% file: etats.bbl
\providecommand{\noopsort}[1]{}\providecommand{\url}[1]{\textit{#1}}
\begin{thebibliography}{132}
\ProvideTextCommand{\guillemotleft}{OT1}{%
  \leavevmode\raise .27ex\hbox{$\scriptscriptstyle\ll$}}
\ProvideTextCommand{\guillemotright}{OT1}{%
  \leavevmode\raise .27ex\hbox{$\scriptscriptstyle\gg$}}
\newcommand{\enquote}[1]{\og #1\fg}
\expandafter\ifx\csname natexlab\endcsname\relax\def\natexlab#1{#1}\fi
\expandafter\ifx\csname url\endcsname\relax
  \def\url#1{\texttt{#1}}\fi
\expandafter\ifx\csname urlprefix\endcsname\relax\def\urlprefix{URL }\fi
\input{frenchbst.tex}
\newcommand{\Capitalize}[1]{\uppercase{#1}}
\newcommand{\capitalize}[1]{\expandafter\Capitalize#1}
\providecommand{\eprint}[2][]{\url{#2}}

\bibitem[{Abbes(1997)}]{abbes1997}
A.~\textsc{Abbes} (1997), \enquote{Hauteurs et discr\'etude}.
  \emph{S{\'e}minaire Bourbaki, 1996/97}, Ast{\'e}risque~\textbf{245}, \bblpp{}
  141--166. Exp.~825.

\bibitem[{Abbes \& Bouche(1995)}]{abbes-b95}
A.~\textsc{Abbes} \& T.~\textsc{Bouche} (1995), \enquote{Th{\'e}or{\`e}me de
  {H}ilbert--{S}amuel {\og arithm{\'e}tique\fg}}. \emph{Ann. Inst. Fourier
  (Grenoble)}, \textbf{45}~(2), \bblpp{} 375--401.

\bibitem[{Agboola \& Pappas(2000)}]{agboola-pappas2000}
A.~\textsc{Agboola} \& G.~\textsc{Pappas} (2000), \enquote{Line bundles,
  rational points and ideal classes}. \emph{Math. Res. Letters}, \textbf{7},
  \bblpp{} 709--717.

\bibitem[{Amerik \emph{\bbletal{}}(1999)Amerik, Rovinsky \& Van~de
  Ven}]{amerik-r-v1999}
E.~\textsc{Amerik}, M.~\textsc{Rovinsky} \& A.~\textsc{Van~de Ven} (1999),
  \enquote{A boundedness theorem for morphisms between threefolds}. \emph{Ann.
  Inst. Fourier (Grenoble)}, \textbf{49}~(2), \bblpp{} 405--415.

\bibitem[{Amoroso \& David(2003)}]{amoroso-david2003}
F.~\textsc{Amoroso} \& S.~\textsc{David} (2003), \enquote{Minoration de la
  hauteur normalis\'ee dans un tore}. \emph{J. Inst. Math. Jussieu},
  \textbf{2}~(3), \bblpp{} 335--381.

\bibitem[{Amoroso \& David(2004)}]{amoroso-david2004}
F.~\textsc{Amoroso} \& S.~\textsc{David} (2004), \enquote{Distribution des
  points de petite hauteur dans les groupes multiplicatifs}. \emph{Ann. Sc.
  Norm. Super. Pisa Cl. Sci. (5)}, \textbf{3}~(2), \bblpp{} 325--348.

\bibitem[{Amoroso \& David(2006)}]{amoroso-david2006}
F.~\textsc{Amoroso} \& S.~\textsc{David} (2006), \enquote{Points de petite
  hauteur sur une sous-vari\'et\'e d'un tore}. \emph{Compos. Math.},
  \textbf{142}~(3), \bblpp{} 551--562.

\bibitem[{Arakelov(1974)}]{arakelov1974}
S.~J. \textsc{Arakelov} (1974), \enquote{An intersection theory for divisors on
  an arithmetic surface}. \emph{Izv. Akad. Nauk SSSR Ser. Mat.}, \textbf{38},
  \bblpp{} 1179--1192.

\bibitem[{Autissier(2001)}]{autissier2001b}
P.~\textsc{Autissier} (2001), \enquote{Points entiers sur les surfaces
  arithm\'etiques}. \emph{J. reine angew. Math.}, \textbf{531}, \bblpp{}
  201--235.

\bibitem[{Baker(2006)}]{baker2006}
M.~\textsc{Baker} (2006), \enquote{A lower bound for average values of
  dynamical {G}reen's functions}. \emph{Math. Res. Lett.}, \textbf{13}~(2-3),
  \bblpp{} 245--257. \url{arXiv:math.NT/0507484}.

\bibitem[{Baker(2009)}]{baker2007}
M.~\textsc{Baker} (2009), \enquote{A finiteness theorem for canonical heights
  attached to rational maps over function fields}. \emph{J. reine angew.
  Math.}, \bblpp{} 205--233. \url{arXiv:math.NT/0601046}.

\bibitem[{Baker \& Petsche(2005)}]{baker-petsche2005}
M.~\textsc{Baker} \& C.~\textsc{Petsche} (2005), \enquote{Global discrepancy
  and small points on elliptic curves}. \emph{Int. Math. Res. Not.},
  \textbf{61}, \bblpp{} 3791--3834.

\bibitem[{Baker \& Hsia(2005)}]{baker-hsia2005}
M.~H. \textsc{Baker} \& L.-C. \textsc{Hsia} (2005), \enquote{Canonical heights,
  transfinite diameters, and polynomial dynamics}. \emph{J. Reine Angew.
  Math.}, \textbf{585}, \bblpp{} 61--92. \url{arXiv:math.NT/0305181}.

\bibitem[{Baker \& Rumely(2006)}]{baker-rumely2006}
M.~H. \textsc{Baker} \& R.~\textsc{Rumely} (2006), \enquote{Equidistribution of
  small points, rational dynamics, and potential theory}. \emph{Ann. Inst.
  Fourier (Grenoble)}, \textbf{56}~(3), \bblpp{} 625--688.
  \url{arXiv:math.NT/0407426}.

\bibitem[{Beauville(1983)}]{beauville1983}
A.~\textsc{Beauville} (1983), \enquote{Vari\'et\'es {K}\"ahleriennes dont la
  premi\`ere classe de {C}hern est nulle}. \emph{J. Differential Geom.},
  \textbf{18}~(4), \bblpp{} 755--782.

\bibitem[{Beauville(2001)}]{beauville2001}
A.~\textsc{Beauville} (2001), \enquote{Endomorphisms of hypersurfaces and other
  manifolds}. \emph{Internat. Math. Res. Notices}, \textbf{1}, \bblpp{} 53--58.

\bibitem[{Benedetto(2007)}]{benedetto2007}
R.~L. \textsc{Benedetto} (2007), \enquote{Preperiodic points of polynomials
  over global fields}. \emph{J. Reine Angew. Math.}, \textbf{608}, \bblpp{}
  123--153.

\bibitem[{Berkovich(1990)}]{berkovich1990}
V.~G. \textsc{Berkovich} (1990), \emph{Spectral theory and analytic geometry
  over non-{A}rchimedean fields}, Mathematical Surveys and
  Monographs~\textbf{33}, American Mathematical Society, Providence, RI.

\bibitem[{Billard(1997)}]{billard1997}
H.~\textsc{Billard} (1997), \enquote{Propri\'et\'es arithm\'etiques d'une
  famille de surfaces {$K3$}}. \emph{Compositio Math.}, \textbf{108}~(3),
  \bblpp{} 247--275.

\bibitem[{Bilu(1997)}]{bilu97}
{\relax Yu}.~\textsc{Bilu} (1997), \enquote{Limit distribution of small points
  on algebraic tori}. \emph{Duke Math. J.}, \textbf{89}~(3), \bblpp{} 465--476.

\bibitem[{Bismut \& Vasserot(1989)}]{bismut-vasserot1989}
J.-M. \textsc{Bismut} \& {\'E}.~\textsc{Vasserot} (1989), \enquote{The
  asymptotics of the {R}ay-{S}inger analytic torsion associated with high
  powers of a positive line bundle}. \emph{Comm. Math. Phys.},
  \textbf{125}~(2), \bblpp{} 355--367.

\bibitem[{Bogomolov(1974{\natexlab{\emph{a}}})}]{bogomolov1974a}
F.~A. \textsc{Bogomolov} (1974{\natexlab{\emph{a}}}), \enquote{{\noopsort
  a}{K}\"ahler manifolds with trivial canonical class}. \emph{Izv. Akad. Nauk
  SSSR Ser. Mat.}, \textbf{38}, \bblpp{} 11--21.

\bibitem[{Bogomolov(1974{\natexlab{\emph{b}}})}]{bogomolov1974b}
F.~A. \textsc{Bogomolov} (1974{\natexlab{\emph{b}}}), \enquote{{\noopsort
  b}{T}he decomposition of {K}\"ahler manifolds with a trivial canonical
  class}. \emph{Mat. Sb. (N.S.)}, \textbf{93(135)}, \bblpp{} 573--575, 630.

\bibitem[{Bogomolov(1980)}]{bogomolov80b}
F.~A. \textsc{Bogomolov} (1980), \enquote{Points of finte order on abelian
  varieties}. \emph{Izv. Akad. Nauk. SSSR Ser. Mat.}, \textbf{44}~(4), \bblpp{}
  782--804, 973.

\bibitem[{Bombieri \& Gubler(2006)}]{bombieri-gubler2006}
E.~\textsc{Bombieri} \& W.~\textsc{Gubler} (2006), \emph{Heights in
  {D}iophantine geometry}, New Mathematical Monographs~\textbf{4}, Cambridge
  University Press, Cambridge.

\bibitem[{Bost \emph{\bbletal{}}(1994)Bost, Gillet \& Soul{\'e}}]{bost-g-s94}
J.-B. \textsc{Bost}, H.~\textsc{Gillet} \& C.~\textsc{Soul{\'e}} (1994),
  \enquote{Heights of projective varieties and positive {G}reen forms}.
  \emph{J. Amer. Math. Soc.}, \textbf{7}, \bblpp{} 903--1027.

\bibitem[{Brolin(1965)}]{brolin1965}
H.~\textsc{Brolin} (1965), \enquote{Invariant sets under iteration of rational
  functions}. \emph{Ark. Mat.}, \textbf{6}, \bblpp{} 103--144 (1965).

\bibitem[{Call \& Silverman(1993)}]{call-s93}
G.~\textsc{Call} \& J.~\textsc{Silverman} (1993), \enquote{Canonical heights on
  varieties with morphisms}. \emph{Compositio Math.}, \textbf{89}, \bblpp{}
  163--205.

\bibitem[{Call \& Goldstine(1997)}]{call-goldstine1997}
G.~S. \textsc{Call} \& S.~W. \textsc{Goldstine} (1997), \enquote{Canonical
  heights on projective space}. \emph{J. Number Theory}, \textbf{63}~(2),
  \bblpp{} 211--243.

\bibitem[{Cantat(2001)}]{cantat2001}
S.~\textsc{Cantat} (2001), \enquote{Dynamique des automorphismes des surfaces
  {$K3$}}. \emph{Acta Math.}, \textbf{187}~(1), \bblpp{} 1--57.

\bibitem[{Cantat(2003)}]{cantat2003}
S.~\textsc{Cantat} (2003), \enquote{Endomorphismes des vari\'et\'es
  homog\`enes}. \emph{Enseign. Math. (2)}, \textbf{49}~(3-4), \bblpp{}
  237--262.

\bibitem[{Chambert-Loir(2000)}]{chambert-loir2000b}
A.~\textsc{Chambert-Loir} (2000), \enquote{Points de petite hauteur sur les
  vari{\'e}t{\'e}s semi-ab{\'e}liennes}. \emph{Ann. Sci. {\'E}cole Norm. Sup.},
  \textbf{33}~(6), \bblpp{} 789--821.

\bibitem[{Chambert-Loir(2006)}]{chambert-loir2006}
A.~\textsc{Chambert-Loir} (2006), \enquote{Mesures et équidistribution sur des
  espaces de {B}erkovich}. \emph{J. reine angew. Math.}, \textbf{595}, \bblpp{}
  215--235. Math.NT/0304023.

\bibitem[{Colliot-Th{\'e}l{\`e}ne(1992)}]{colliot-thelene1992}
J.-L. \textsc{Colliot-Th{\'e}l{\`e}ne} (1992), \enquote{L'arithm{\'e}tique des
  vari{\'e}t{\'e}s rationnelles}. \emph{Annales de la facult{\'e} des sciences
  de Toulouse S{\'e}r. 6}, \textbf{1}~(3), \bblpp{} 295--336.

\bibitem[{David \& Hindry(2000)}]{david-hindry2000}
S.~\textsc{David} \& M.~\textsc{Hindry} (2000), \enquote{Minoration de la
  hauteur de {N}\'eron-{T}ate sur les vari\'et\'es ab\'eliennes de type {C}.
  {M}}. \emph{J. Reine Angew. Math.}, \textbf{529}, \bblpp{} 1--74.

\bibitem[{David \& Philippon(1998)}]{david-p98}
S.~\textsc{David} \& P.~\textsc{Philippon} (1998), \enquote{Minorations des
  hauteurs normalis{\'e}es des sous-vari{\'e}t{\'e}s de vari{\'e}t{\'e}s
  ab{\'e}liennes}. \emph{International Conference on Discrete Mathematics and
  Number Theory}, Contemp. Math.~\textbf{210}, \bblpp{} 333--364,
  Tiruchirapelli, 1996.

\bibitem[{David \& Philippon(1999)}]{david-philippon1999}
S.~\textsc{David} \& P.~\textsc{Philippon} (1999), \enquote{Minorations des
  hauteurs normalis{\'e}es des sous-vari{\'e}t{\'e}s des tores}. \emph{Ann.
  Scuola Norm. Sup. Pisa}, \textbf{28}~(3), \bblpp{} 489--543.

\bibitem[{David \& Philippon(2000)}]{david-p2000}
S.~\textsc{David} \& P.~\textsc{Philippon} (2000), \enquote{Sous-vari\'et\'es
  de torsion des vari\'et\'es semi-ab\'eliennes}. \emph{C. R. Acad. Sci. Paris
  S\'er. I Math.}, \textbf{331}~(8), \bblpp{} 587--592.

\bibitem[{David \& Philippon(2002)}]{david-p2002}
S.~\textsc{David} \& P.~\textsc{Philippon} (2002), \enquote{Minorations des
  hauteurs normalis\'ees des sous-vari\'et\'es de vari\'et\'es ab\'eliennes.
  {II}}. \emph{Comment. Math. Helv.}, \textbf{77}~(4), \bblpp{} 639--700.

\bibitem[{Demailly(1993)}]{demailly93}
J.-P. \textsc{Demailly} (1993), \enquote{Monge-{A}mp{\`e}re operators, {L}elong
  numbers and intersection theory}. \emph{Complex analysis and geometry}, Univ.
  Ser. Math., \bblpp{} 115--193, Plenum, New York.

\bibitem[{DeMarco(2003)}]{demarco2003}
L.~\textsc{DeMarco} (2003), \enquote{Dynamics of rational maps: {L}yapunov
  exponents, bifurcations, and capacity}. \emph{Math. Ann.}, \textbf{326}~(1),
  \bblpp{} 43--73.

\bibitem[{DeMarco \& Rumely(2007)}]{demarco-r2007}
L.~\textsc{DeMarco} \& R.~\textsc{Rumely} (2007), \enquote{{Transfinite
  diameter and the resultant}}. \emph{J. reine angew. Math.}, \textbf{611},
  \bblpp{} 145--161. \url{arXiv:math.CV/0601109}.

\bibitem[{Denis(1995)}]{denis1995b}
L.~\textsc{Denis} (1995), \enquote{Points p\'eriodiques des automorphismes
  affines}. \emph{J. Reine Angew. Math.}, \textbf{467}, \bblpp{} 157--167.

\bibitem[{Dinh \& Sibony(2005)}]{dinh-sibony2005}
T.-C. \textsc{Dinh} \& N.~\textsc{Sibony} (2005), \enquote{Green currents for
  holomorphic automorphisms of compact {K}\"ahler manifolds}. \emph{J. Amer.
  Math. Soc.}, \textbf{18}~(2), \bblpp{} 291--312.

\bibitem[{Dobrowolski(1979)}]{dobrowolski1979}
E.~\textsc{Dobrowolski} (1979), \enquote{On a question of {L}ehmer and the
  number of irreducible factors of a polynomial}. \emph{Acta Arith.},
  \textbf{34}~(4), \bblpp{} 391--401.

\bibitem[{Erd{\"o}s \& Tur{\'a}n(1950)}]{erdos-turan1950}
P.~\textsc{Erd{\"o}s} \& P.~\textsc{Tur{\'a}n} (1950), \enquote{On the
  distribution of roots of polynomials}. \emph{Ann. of Math. (2)}, \textbf{51},
  \bblpp{} 105--119.

\bibitem[{Er{\"e}menko(1989)}]{eremenko1989}
A.~{\`E}. \textsc{Er{\"e}menko} (1989), \enquote{Some functional equations
  connected with the iteration of rational functions}. \emph{Algebra i Analiz},
  \textbf{1}~(4), \bblpp{} 102--116.

\bibitem[{Fakhruddin(2003)}]{fakhruddin2003}
N.~\textsc{Fakhruddin} (2003), \enquote{Questions on self-maps of algebraic
  varieties}. \emph{J. Ramanujan Math. Soc.}, \textbf{18}~(2), \bblpp{}
  109--122.

\bibitem[{Faltings(1984)}]{faltings1984}
G.~\textsc{Faltings} (1984), \enquote{Calculus on arithmetic surfaces}.
  \emph{Ann. of Math. (2)}, \textbf{119}~(2), \bblpp{} 387--424.

\bibitem[{Faltings(1991)}]{faltings1991}
G.~\textsc{Faltings} (1991), \enquote{Diophantine approximation on abelian
  varieties}. \emph{Ann. of Math. (2)}, \textbf{133}~(3), \bblpp{} 549--576.

\bibitem[{Faltings(1992)}]{faltings1992}
G.~\textsc{Faltings} (1992), \emph{Lectures on the arithmetic {R}iemann-{R}och
  theorem}, Annals of Mathematics Studies~\textbf{127}, Princeton University
  Press, Princeton, NJ. Notes taken by Shouwu Zhang.

\bibitem[{Favre \& Rivera-Letelier(2006)}]{favre-rl2006}
C.~\textsc{Favre} \& J.~\textsc{Rivera-Letelier} (2006),
  \enquote{{Equidistribution quantitative des points de petite hauteur sur la
  droite projective}}. \emph{Math. Ann.}, \textbf{335}~(2), \bblpp{} 311--361.
  \url{arXiv:math.NT/0407471}.

\bibitem[{Flexor \& Oesterl{\'e}(1990)}]{flexor-o1990}
M.~\textsc{Flexor} \& J.~\textsc{Oesterl{\'e}} (1990), \enquote{Sur les points
  de torsion des courbes elliptiques}. \emph{S{\'e}minaire sur les pinceaux de
  courbes elliptiques},  \bbledby{} L.~\textsc{Szpiro},
  Ast{\'e}risque~\textbf{183}, \bblpp{} 25--36.

\bibitem[{Fujimoto(2002)}]{fujimoto2002}
Y.~\textsc{Fujimoto} (2002), \enquote{Endomorphisms of smooth projective
  3-folds with non-negative {K}odaira dimension}. \emph{Publ. Res. Inst. Math.
  Sci.}, \textbf{38}~(1), \bblpp{} 33--92.

\bibitem[{Fujimoto \& Nakayama(2005)}]{fujimoto-n2005}
Y.~\textsc{Fujimoto} \& N.~\textsc{Nakayama} (2005), \enquote{Compact complex
  surfaces admitting non-trivial surjective endomorphisms}. \emph{Tohoku Math.
  J. (2)}, \textbf{57}~(3), \bblpp{} 395--426.

\bibitem[{Fulton(1998)}]{fulton98}
W.~\textsc{Fulton} (1998), \emph{Intersection theory}, Springer-Verlag, Berlin,
  second \bbledition{}.

\bibitem[{Ghioca \& Tucker(2009)}]{ghioca-tucker2009}
D.~\textsc{Ghioca} \& T.~\textsc{Tucker} (2009), \enquote{A counter-example for
  the dynamical {M}anin-{M}umford conjecture}.

\bibitem[{Gillet \& Soul{\'e}(1988)}]{gillet-s88}
H.~\textsc{Gillet} \& C.~\textsc{Soul{\'e}} (1988), \enquote{Amplitude
  arithm{\'e}tique}. \emph{C. R. Acad. Sci. Paris S{\'e}r. I Math.},
  \textbf{307}, \bblpp{} 887--890.

\bibitem[{Gillet \& Soul{\'e}(1990{\natexlab{\emph{a}}})}]{gillet-s90}
H.~\textsc{Gillet} \& C.~\textsc{Soul{\'e}} (1990{\natexlab{\emph{a}}}),
  \enquote{Arithmetic intersection theory}. \emph{Publ. Math. Inst. Hautes
  {\'E}tudes Sci.}, \textbf{72}, \bblpp{} 94--174.

\bibitem[{Gillet \& Soul{\'e}(1990{\natexlab{\emph{b}}})}]{gillet-s90b}
H.~\textsc{Gillet} \& C.~\textsc{Soul{\'e}} (1990{\natexlab{\emph{b}}}),
  \enquote{Characteristic classes for algebraic vector bundles with {H}ermitian
  metric {I}, {II}}. \emph{Ann. of Math.}, \textbf{131}, \bblpp{} 163--203.

\bibitem[{Gillet \& Soul{\'e}(1992)}]{gillet-s92}
H.~\textsc{Gillet} \& C.~\textsc{Soul{\'e}} (1992), \enquote{An arithmetic
  {R}iemann--{R}och theorem}. \emph{Invent. Math.}, \textbf{110}, \bblpp{}
  473--543.

\bibitem[{Gromov(2003)}]{gromov2003}
M.~\textsc{Gromov} (2003), \enquote{On the entropy of holomorphic maps}.
  \emph{Enseign. Math. (2)}, \textbf{49}~(3-4), \bblpp{} 217--235.

\bibitem[{Gubler(1998)}]{gubler1998}
W.~\textsc{Gubler} (1998), \enquote{Local heights of subvarieties over
  non-archimedean fields}. \emph{J. reine angew. Math.}, \textbf{498}, \bblpp{}
  61--113.

\bibitem[{Gubler(2003)}]{gubler2003}
W.~\textsc{Gubler} (2003), \enquote{Local and canonical heights of
  subvarieties}. \emph{Ann. Scuola Norm. Sup. Pisa}, \textbf{2}~(4), \bblpp{}
  711--760.

\bibitem[{Gubler(2007)}]{gubler2007a}
W.~\textsc{Gubler} (2007), \enquote{The {B}ogomolov conjecture for totally
  degenerate abelian varieties}. \emph{Invent. Math.}, \textbf{169}~(2),
  \bblpp{} 377--400.

\bibitem[{Hindry(1988)}]{hindry1988}
M.~\textsc{Hindry} (1988), \enquote{Autour d'une conjecture de {S}erge {L}ang}.
  \emph{Invent. Math.}, \textbf{94}~(3), \bblpp{} 575--603.

\bibitem[{Hindry \& Silverman(2000)}]{hindry-silverman2000}
M.~\textsc{Hindry} \& J.~H. \textsc{Silverman} (2000), \emph{Diophantine
  geometry}, Graduate Texts in Mathematics~\textbf{201}, Springer-Verlag, New
  York. An introduction.

\bibitem[{H{\"o}hn \& Skoruppa(1993)}]{hohn-skoruppa1993}
G.~\textsc{H{\"o}hn} \& N.-P. \textsc{Skoruppa} (1993), \enquote{Un
  r{\'e}sultat de {S}chinzel}. \emph{J. Th{\'e}or. Nombres Bordeaux},
  \textbf{5}~(1), \bblp{} 185.

\bibitem[{Hrushovski(2001)}]{hrushovski2001}
E.~\textsc{Hrushovski} (2001), \enquote{The {M}anin-{M}umford conjecture and
  the model theory of difference fields}. \emph{Ann. Pure Appl. Logic},
  \textbf{112}~(1), \bblpp{} 43--115.

\bibitem[{Hrushovski(2004)}]{hrushovski2004}
E.~\textsc{Hrushovski} (2004), \enquote{{The Elementary Theory of the Frobenius
  Automorphisms}}. \url{arXiv:math.LO/0406514}.

\bibitem[{Kawaguchi(1999)}]{kawaguchi1999}
S.~\textsc{Kawaguchi} (1999), \enquote{Some remarks on rational periodic
  points}. \emph{Math. Res. Lett.}, \textbf{6}~(5-6), \bblpp{} 495--509.

\bibitem[{Kawaguchi(2006)}]{kawaguchi2006}
S.~\textsc{Kawaguchi} (2006), \enquote{Canonical height functions for affine
  plane automorphisms}. \url{arXiv:math.NT/0405007}.

\bibitem[{Kawaguchi(2008)}]{kawaguchi2008}
S.~\textsc{Kawaguchi} (2008), \enquote{{Projective surface automorphisms of
  positive topological entropy from an arithmetic viewpoint}}. \emph{Amer. J.
  Math.}, \textbf{130}~(1), \bblpp{} 159--186. \url{arXiv:math.AG/0510634}.

\bibitem[{Kawaguchi \&
  Silverman(2007{\natexlab{\emph{a}}})}]{kawaguchi-silverman2007a}
S.~\textsc{Kawaguchi} \& J.~H. \textsc{Silverman} (2007{\natexlab{\emph{a}}}),
  \enquote{Dynamics of projective morphisms having identical canonical
  heights}. \emph{Proc. London Math. Soc.}, \textbf{95}, \bblpp{} 519--544.

\bibitem[{Kawaguchi \&
  Silverman(2007{\natexlab{\emph{b}}})}]{kawaguchi-silverman2007b}
S.~\textsc{Kawaguchi} \& J.~H. \textsc{Silverman} (2007{\natexlab{\emph{b}}}),
  \enquote{Nonarchimedean {G}reen functions and dynamics on projective space}.
  \url{arXiv:0706.2169}.

\bibitem[{Lang(1983)}]{lang1983}
S.~\textsc{Lang} (1983), \emph{Fundamentals of {D}iophantine geometry},
  Springer-Verlag, New York.

\bibitem[{Lang(1988)}]{lang1988}
S.~\textsc{Lang} (1988), \emph{Introduction to {A}rakelov theory},
  Springer-Verlag, New York.

\bibitem[{Lang(1995)}]{lang1964}
S.~\textsc{Lang} (1995), \enquote{Les formes bilin\'eaires de {N}\'eron et
  {T}ate}. \emph{S\'eminaire Bourbaki, 1963/64}, \bblpp{} 435--445, Soc. Math.
  France, Paris. Expos\'e 274.

\bibitem[{Laurent(1983)}]{laurent1983}
M.~\textsc{Laurent} (1983), \enquote{Minoration de la hauteur de
  {N}\'eron-{T}ate}. \emph{Seminar on number theory, Paris 1981--82 (Paris,
  1981/1982)}, Progr. Math.~\textbf{38}, \bblpp{} 137--151, Birkh\"auser
  Boston, Boston, MA.

\bibitem[{Lehmer(1933)}]{lehmer1933}
D.~H. \textsc{Lehmer} (1933), \enquote{Factorization of certain cyclotomic
  functions}. \emph{Ann. of Math. (2)}, \textbf{34}~(3), \bblpp{} 461--479.

\bibitem[{Levin \& Przytycki(1997)}]{levin-p1997}
G.~\textsc{Levin} \& F.~\textsc{Przytycki} (1997), \enquote{When do two
  rational functions have the same {J}ulia set?} \emph{Proc. Amer. Math. Soc.},
  \textbf{125}~(7), \bblpp{} 2179--2190.

\bibitem[{Lewis(1972)}]{lewis1972}
D.~J. \textsc{Lewis} (1972), \enquote{Invariant sets of morphisms on projective
  and affine number spaces}. \emph{J. Algebra}, \textbf{20}, \bblpp{} 419--434.

\bibitem[{Lyubich(1983)}]{lyubich1983}
M.~J. \textsc{Lyubich} (1983), \enquote{Entropy properties of rational
  endomorphisms of the {R}iemann sphere}. \emph{Ergodic Theory Dynam. Systems},
  \textbf{3}~(3), \bblpp{} 351--385.

\bibitem[{Marcello(2003)}]{marcello2003}
S.~\textsc{Marcello} (2003), \enquote{Sur la dynamique arithm\'etique des
  automorphismes de l'espace affine}. \emph{Bull. Soc. Math. France},
  \textbf{131}~(2), \bblpp{} 229--257.

\bibitem[{Masser(1984)}]{masser1984}
D.~W. \textsc{Masser} (1984), \enquote{Small values of the quadratic part of
  the {N}\'eron-{T}ate height on an abelian variety}. \emph{Compositio Math.},
  \textbf{53}~(2), \bblpp{} 153--170.

\bibitem[{Mazur(1977)}]{mazur1977}
B.~\textsc{Mazur} (1977), \enquote{Modular curves and the {E}isenstein ideal}.
  \emph{Inst. Hautes \'Etudes Sci. Publ. Math.}, \textbf{47}, \bblpp{} 33--186
  (1978).

\bibitem[{McMullen(2002)}]{mcmullen2002}
C.~T. \textsc{McMullen} (2002), \enquote{Dynamics on {$K3$} surfaces: {S}alem
  numbers and {S}iegel disks}. \emph{J. Reine Angew. Math.}, \textbf{545},
  \bblpp{} 201--233.

\bibitem[{Merel(1996)}]{merel96}
L.~\textsc{Merel} (1996), \enquote{Bornes pour la torsion des courbes
  elliptiques sur les corps de nombres}. \emph{Invent. Math.},
  \textbf{124}~(1-3), \bblpp{} 437--449.

\bibitem[{Milnor(2006)}]{milnor2006}
J.~\textsc{Milnor} (2006), \enquote{On {L}att\`es maps}. \emph{Dynamics on the
  {R}iemann Sphere, A {B}odil {B}ranner {F}estschrift},  \bbledby{} P.~G.
  \textsc{Hjorth} \& C.~L. \textsc{Petersens}, European mathematical society.
  \url{arXiv:math.DS/0402147}, Stony Brook IMS Preprint \#2004/01.

\bibitem[{Mimar(1997)}]{mimar1997}
A.~\textsc{Mimar} (1997), \emph{On the preperiodic points of an endomorphism of
  {$\mathbf P^1\times\mathbf P^1$}}. Phd thesis, Columbia University.

\bibitem[{Morton \& Silverman(1994)}]{morton-silverman1994}
P.~\textsc{Morton} \& J.~H. \textsc{Silverman} (1994), \enquote{Rational
  periodic points of rational functions}. \emph{Internat. Math. Res. Notices},
  \textbf{2}, \bblpp{} 97--109.

\bibitem[{Mumford(1974)}]{mumford74}
D.~\textsc{Mumford} (1974), \emph{Abelian Varieties}, Oxford Univ. Press.

\bibitem[{Nakayama(2002)}]{nakayama2002}
N.~\textsc{Nakayama} (2002), \enquote{Ruled surfaces with non-trivial
  surjective endomorphisms}. \emph{Kyushu J. Math.}, \textbf{56}~(2), \bblpp{}
  433--446.

\bibitem[{N{\'e}ron(1965)}]{neron1965}
A.~\textsc{N{\'e}ron} (1965), \enquote{Quasi-fonctions et hauteurs sur les
  vari\'et\'es ab\'eliennes}. \emph{Ann. of Math. (2)}, \textbf{82}, \bblpp{}
  249--331.

\bibitem[{Northcott(1950)}]{northcott1950}
D.~G. \textsc{Northcott} (1950), \enquote{Periodic points on an algebraic
  variety}. \emph{Ann. of Math.}, \textbf{51}, \bblpp{} 167--177.

\bibitem[{Pazuki(2009)}]{pazuki2009}
F.~\textsc{Pazuki} (2009), \enquote{{Z}hang's conjecture and squares of abelian
  surfaces}.

\bibitem[{Peyre(2002)}]{peyre2002}
E.~\textsc{Peyre} (2002), \enquote{Points de hauteur born\'ee et g\'eom\'etrie
  des vari\'et\'es (d'apr\`es {Y}. {M}anin et al.)}. \emph{S\'eminaire
  Bourbaki, Vol. 2000/2001}, Ast{\'e}risque~\textbf{282}, \bblpp{} Exp. No.
  891, ix, 323--344.

\bibitem[{Philippon(1991)}]{philippon91}
P.~\textsc{Philippon} (1991), \enquote{Sur des hauteurs alternatives, {I}}.
  \emph{Math. Ann.}, \textbf{289}, \bblpp{} 255--283.

\bibitem[{Philippon(1995)}]{philippon95}
P.~\textsc{Philippon} (1995), \enquote{Sur des hauteurs alternatives, {III}}.
  \emph{J. Math. Pures Appl.}, \textbf{74}, \bblpp{} 345--365.

\bibitem[{Pink \& Roessler(2002)}]{pink-roessler2002}
R.~\textsc{Pink} \& D.~\textsc{Roessler} (2002), \enquote{On {H}rushovski's
  proof of the {M}anin-{M}umford conjecture}. \emph{Proceedings of the
  International Congress of Mathematicians, Vol. I (Beijing, 2002)}, \bblpp{}
  539--546, Higher Ed. Press, Beijing.

\bibitem[{Pink \& Roessler(2004)}]{pink-roessler2004}
R.~\textsc{Pink} \& D.~\textsc{Roessler} (2004), \enquote{On {$\psi$}-invariant
  subvarieties of semiabelian varieties and the {M}anin-{M}umford conjecture}.
  \emph{J. Algebraic Geom.}, \textbf{13}~(4), \bblpp{} 771--798.

\bibitem[{Poonen(1999)}]{poonen99}
B.~\textsc{Poonen} (1999), \enquote{{M}ordell-{L}ang plus {B}ogomolov}.
  \emph{Invent. Math.}, \textbf{137}~(2), \bblpp{} 413--245.

\bibitem[{Ransford(1995)}]{ransford95}
T.~\textsc{Ransford} (1995), \emph{Potential theory in the complex plane},
  Cambridge University Press, Cambridge.

\bibitem[{Raynaud(1983{\natexlab{\emph{a}}})}]{raynaud83}
M.~\textsc{Raynaud} (1983{\natexlab{\emph{a}}}), \enquote{Courbes sur une
  vari{\'e}t{\'e} ab{\'e}lienne et points de torsion}. \emph{Invent. Math.},
  \textbf{71}~(1), \bblpp{} 207--233.

\bibitem[{Raynaud(1983{\natexlab{\emph{b}}})}]{raynaud83c}
M.~\textsc{Raynaud} (1983{\natexlab{\emph{b}}}), \enquote{Sous-vari{\'e}t{\'e}s
  d'une vari{\'e}t{\'e} ab{\'e}lienne et points de torsion}. \emph{Arithmetic
  and Geometry. Papers dedicated to I.R. Shafarevich},  \bbledby{}
  M.~\textsc{Artin} \& J.~\textsc{Tate}, Progr. Math.~\textbf{35}, \bblpp{}
  327--352, Birkh{\"a}user.

\bibitem[{Ritt(1923)}]{ritt1923}
J.~F. \textsc{Ritt} (1923), \enquote{Permutable rational functions}.
  \emph{Trans. Amer. Math. Soc.}, \textbf{25}~(3), \bblpp{} 399--448.

\bibitem[{Roessler(2005)}]{roessler2005}
D.~\textsc{Roessler} (2005), \enquote{A note on the {M}anin-{M}umford
  conjecture}. \emph{Number fields and function fields---two parallel worlds},
  Progr. Math.~\textbf{239}, \bblpp{} 311--318, Birkh\"auser Boston, Boston,
  MA.

\bibitem[{Rumely(1999)}]{rumely99}
R.~\textsc{Rumely} (1999), \enquote{On {B}ilu's equidistribution theorem}.
  \emph{Spectral problems in geometry and arithmetic}, Contemp.
  Math.~\textbf{237}, \bblpp{} 159--166, Iowa City, IA, 1997.

\bibitem[{Schanuel(1979)}]{schanuel79}
S.~\textsc{Schanuel} (1979), \enquote{Heights in number fields}. \emph{Bull.
  Soc. Math. France}, \textbf{107}, \bblpp{} 433--449.

\bibitem[{Schinzel(1974/75)}]{schinzel1975}
A.~\textsc{Schinzel} (1974/75), \enquote{Addendum to the paper: ``{O}n the
  product of the conjugates outside the unit circle of an algebraic number''
  ({A}cta {A}rith. {\bf 24} (1973), 385--399)}. \emph{Acta Arith.},
  \textbf{26}~(3), \bblpp{} 329--331.

\bibitem[{Serre(1960)}]{serre60b}
J.-P. \textsc{Serre} (1960), \enquote{Analogues k{\"a}hleriens de certaines
  conjectures de {W}eil}. \emph{Ann. of Math.}, \textbf{71}, \bblpp{} 392--394.

\bibitem[{Serre(1986)}]{serre-85-86}
J.-P. \textsc{Serre} (1986), \enquote{R\'esum\'e des cours de 1985--1986}.
  \emph{Annuaire du Coll{\`e}ge de France (1986)}, \bblpp{} 95--99, Coll{\`e}ge
  de France. \OE uvres, IV, p.~33--37.

\bibitem[{Serre(1997)}]{serre1997}
J.-P. \textsc{Serre} (1997), \emph{Lectures on the {M}ordell-{W}eil theorem},
  Aspects of Mathematics, Friedr. Vieweg \& Sohn, Braunschweig, third
  \bbledition{}. Translated from the French and edited by Martin Brown from
  notes by Michel Waldschmidt, With a foreword by Brown and Serre.

\bibitem[{Serre(2000)}]{serre2000}
J.-P. \textsc{Serre} (2000), \emph{Local algebra}, Springer Monographs in
  Mathematics, Springer-Verlag, Berlin. Translated from the French by CheeWhye
  Chin and revised by the author.

\bibitem[{Sibony(1999)}]{sibony1999}
N.~\textsc{Sibony} (1999), \enquote{Dynamique des applications rationnelles de
  {$\mathbf P\sp k$}}. \emph{Dynamique et g\'eom\'etrie complexes (Lyon,
  1997)}, Panorama et Synth\`eses~\textbf{8}, \bblpp{} 97--185, Soc. Math.
  France, Paris.

\bibitem[{Silverman(1991)}]{silverman1991}
J.~H. \textsc{Silverman} (1991), \enquote{Rational points on {$K3$} surfaces: a
  new canonical height}. \emph{Invent. Math.}, \textbf{105}~(2), \bblpp{}
  347--373.

\bibitem[{Silverman(1994)}]{silverman1994}
J.~H. \textsc{Silverman} (1994), \enquote{Geometric and arithmetic properties
  of the {H}\'enon map}. \emph{Math. Z.}, \textbf{215}~(2), \bblpp{} 237--250.

\bibitem[{Silverman(2007)}]{silverman2007}
J.~H. \textsc{Silverman} (2007), \emph{The arithmetic of dynamical systems},
  Graduate Texts in Mathematics~\textbf{241}, Springer, New York.

\bibitem[{Siu(1993)}]{siu1993}
Y.~T. \textsc{Siu} (1993), \enquote{An effective {M}atsusaka big theorem}.
  \emph{Ann. Inst. Fourier (Grenoble)}, \textbf{43}~(5), \bblpp{} 1387--1405.

\bibitem[{Soul{\'e} \emph{\bbletal{}}(1992)Soul{\'e}, Abramovich, Burnol \&
  Kramer}]{soule-a-b-k92}
C.~\textsc{Soul{\'e}}, D.~\textsc{Abramovich}, J.-F. \textsc{Burnol} \&
  J.~\textsc{Kramer} (1992), \emph{Lectures on {A}rakelov geometry}, Cambridge
  studies in advanced mathematics~\textbf{33}, Cambridge University Press.

\bibitem[{Szpiro(1990)}]{szpiro90}
L.~\textsc{Szpiro} (1990), \enquote{Sur les propri{\'e}t{\'e}s num{\'e}riques
  du dualisant relatif d'une surface arithm{\'e}tique}. \emph{The
  {G}rothendieck {F}estschrift},  \bbledby{} P.~\textsc{Cartier},
  L.~\textsc{Illusie}, N.~M. \textsc{Katz}, G.~\textsc{Laumon}, {\relax
  Yu}.~\textsc{Manin} \& K.~\textsc{Ribet}, ~\textbf{3}, \bblpp{} 229--246,
  Birkh{\"a}user.

\bibitem[{Szpiro \emph{\bbletal{}}(1997)Szpiro, Ullmo \& Zhang}]{szpiro-u-z97}
L.~\textsc{Szpiro}, E.~\textsc{Ullmo} \& S.-W. \textsc{Zhang} (1997),
  \enquote{{\'E}quidistribution des petits points}. \emph{Invent. Math.},
  \textbf{127}, \bblpp{} 337--348.

\bibitem[{Teissier(1990)}]{teissier1990}
B.~\textsc{Teissier} (1990), \enquote{R\'esultats r\'ecents d'alg\`ebre
  commutative effective}. \emph{S\'eminaire Bourbaki, Vol.\ 1989/90},
  Ast\'erisque~\textbf{189-190}, \bblpp{} Exp.\ No.\ 718, 107--131.

\bibitem[{Ullmo(1998)}]{ullmo98}
E.~\textsc{Ullmo} (1998), \enquote{Positivit{\'e} et discr{\'e}tion des points
  alg{\'e}briques des courbes}. \emph{Ann. of Math.}, \textbf{147}~(1),
  \bblpp{} 167--179.

\bibitem[{Waldschmidt(2000)}]{waldschmidt2000}
M.~\textsc{Waldschmidt} (2000), \emph{Diophantine approximation on linear
  algebraic groups. Transcendence properties of the exponential function in
  several variables}, Grundlehren der Mathematischen
  Wissenschaften~\textbf{326}, Springer-Verlag.

\bibitem[{Yau(1978)}]{yau1978}
S.~T. \textsc{Yau} (1978), \enquote{On the {R}icci curvature of a compact
  {K}\"ahler manifold and the complex {M}onge-{A}mp\`ere equation. {I}}.
  \emph{Comm. Pure Appl. Math.}, \textbf{31}~(3), \bblpp{} 339--411.

\bibitem[{Yomdin(1987)}]{yomdin1987}
Y.~\textsc{Yomdin} (1987), \enquote{Volume growth and entropy}. \emph{Israel J.
  Math.}, \textbf{57}~(3), \bblpp{} 285--300.

\bibitem[{Yuan(2008)}]{yuan2008}
X.~\textsc{Yuan} (2008), \enquote{Big line bundles on arithmetic varieties}.
  \emph{Invent. Math.}, \textbf{173}, \bblpp{} 603--649.
  \url{arXiv:math.NT/0612424}.

\bibitem[{Zhang(1995{\natexlab{\emph{a}}})}]{zhang95}
S.-W. \textsc{Zhang} (1995{\natexlab{\emph{a}}}), \enquote{Positive line
  bundles on arithmetic varieties}. \emph{J. Amer. Math. Soc.}, \textbf{8},
  \bblpp{} 187--221.

\bibitem[{Zhang(1995{\natexlab{\emph{b}}})}]{zhang95b}
S.-W. \textsc{Zhang} (1995{\natexlab{\emph{b}}}), \enquote{Small points and
  adelic metrics}. \emph{J. Algebraic Geometry}, \textbf{4}, \bblpp{} 281--300.

\bibitem[{Zhang(1998)}]{zhang98}
S.-W. \textsc{Zhang} (1998), \enquote{Equidistribution of small points on
  abelian varieties}. \emph{Ann. of Math.}, \textbf{147}~(1), \bblpp{}
  159--165.

\bibitem[{Zhang(2006)}]{zhang2006}
S.-W. \textsc{Zhang} (2006), \enquote{Distributions in algebraic dynamics}.
  \emph{A tribute to Professor S. S. Chern}, Surveys in Differential
  Geometry~\textbf{10}, \bblpp{} 381--430, International Press.

\end{thebibliography}
